\documentclass[a4paper,12pt,leqno]{amsart} 
\usepackage{amsmath,amsthm,amssymb}  
\usepackage{amscd}
\usepackage[all]{xy} 
\usepackage{mathdots}
\usepackage{graphicx}

\numberwithin{equation}{section}

\theoremstyle{plain}
\newtheorem{theorem}{Theorem}[section]
\newtheorem{proposition}[theorem]{Proposition}         
\newtheorem{corollary}[theorem]{Corollary} 
\newtheorem{lemma}[theorem]{Lemma} 
\newtheorem{definition}[theorem]{Definition}  

\theoremstyle{definition}  
\newtheorem{example}[theorem]{Example} 
\newtheorem{remark}[theorem]{Remark} 

\newcommand{\C}{\mathbb C}   
\newcommand{\R}{\mathbb R}
\newcommand{\Z}{\mathbb Z}
\newcommand{\N}{\mathbb N}  

\newcommand{\al}{\alpha}

\newcommand{\de}{\delta}
\newcommand{\la}{\lambda}
\newcommand{\si}{\sigma} 
\newcommand{\La}{\Lambda}
\newcommand{\eps}{\epsilon}
\newcommand{\Om}{\Omega}
\newcommand{\De}{\Delta}
\renewcommand{\th}{\theta}
\newcommand{\om}{\omega}

\newcommand{\Ga}{\Gamma}
\newcommand{\ze}{\zeta}

\DeclareMathOperator{\diag}{diag}

\DeclareMathOperator{\id}{id}

\newcommand{\SL}{\textrm{SL}}

\renewcommand{\sl}{\frak s\frak l}

\newcommand{\cA}{\mathcal{A}}
\newcommand{\cR}{\mathcal{R}}

\newcommand{\cQ}{\mathcal{Q}}
\newcommand{\cD}{\mathcal{D}}
\newcommand{\cX}{\mathcal{X}}

\newcommand{\no}{\noindent}
\newcommand{\sub}{\subseteq}      
\newcommand{\pr}{\prime} 
\newcommand{\prr}{{\prime\prime}} 
\newcommand{\st}{\ \vert\ }   
\renewcommand{\ll}{\lq\lq}
\newcommand{\rr}{\rq\rq\ }
\newcommand{\rrr}{\rq\rq}  
\renewcommand{\b}{\partial}

\newcommand{\bze}{\b_\zeta}

\newcommand{\bp}{\begin{pmatrix}} 
\newcommand{\ep}{\end{pmatrix}} 
\newcommand{\bsp}{\left(\begin{smallmatrix}} 
\newcommand{\esp}{\end{smallmatrix}\right)}

\newcommand{\tbar}{  {\bar t}  }
\newcommand{\ttb}{ {t\bar t}  }

\renewcommand{\i}{ {\scriptscriptstyle\sqrt{-1}}\, }
\newcommand{\ii}{ {\scriptstyle\sqrt{-1}}\, }

\newcommand{\Psiz}{  \Psi^{(0)}  }
\newcommand{\Psii}{  \Psi^{(\infty)}  }
\newcommand{\psiz}{  \psi^{(0)}  }
\newcommand{\psii}{  \psi^{(\infty)}  }
\newcommand{\Sz}{  S^{(0)}  }
\newcommand{\Si}{  S^{(\infty)}  }
\newcommand{\Qz}{  Q^{(0)}  }
\newcommand{\Qi}{  Q^{(\infty)}  }

\newcommand{\Omz}{  \Om^{(0)}  }
\newcommand{\Omi}{  \Om^{(\infty)}  }
\newcommand{\thz}{  \th^{(0)}  }
\newcommand{\thi}{  \th^{(\infty)}  }

\newcommand{\Mz}{ M^{(0)}  }

\newcommand{\tq}{  \tilde  q }

\newcommand{\fPsi}{  \,\underline{\!\Psi\!}\, }
\newcommand{\fPsiz}{  \,\underline{\!\Psi\!}\,^{(0)}  }
\newcommand{\fPsii}{  \,\underline{\!\Psi\!}\,^{(\infty)}  }

\newcommand{\tPsi}{  \tilde\Psi  }
\newcommand{\tPsiz}{  \tilde\Psi^{(0)}  }
\newcommand{\tPsii}{  \tilde\Psi^{(\infty)}  }

\newcommand{\tSz}{  \tilde S^{(0)}  }

\newcommand{\tQz}{  \tilde Q^{(0)}  }
\newcommand{\tQi}{  \tilde Q^{(\infty)}  }

\newcommand{\tE}{  \tilde E  }

\newcommand{\tD}{  \tilde D  }
\newcommand{\tMz}{  \tilde M^{(0)}  }
\newcommand{\tY}{  \tilde Y }
\newcommand{\tG}{  \tilde G }

\newcommand{\ftPsii}{  \,\underline{\!\tPsi\!}\,^{(\infty)}  }
\newcommand{\fthz}{  \,\underline{\th\!}\,^{(0)}  }
\newcommand{\fthi}{  \,\underline{\th\!}\,^{(\infty)}  }

\newcommand{\Phiz}{\Phi^{(0)}}
\newcommand{\Phii}{\Phi^{(\infty)}}
\newcommand{\phiz}{\phi^{(0)}}
\newcommand{\phii}{\phi^{(\infty)}}

\newcommand{\tPhiz}{  \tilde\Phi^{(0)}  }

\newcommand{\cXz}{\cX^{(0)}}
\newcommand{\cXi}{\cX^{(\infty)}}
\newcommand{\xzer}{x^{(0)}}
\newcommand{\xinf}{x^{(\infty)}}
\newcommand{\cQz}{  \cQ^{(0)}  }

\newcommand{\cRz}{ \cR^{(0)}}

\newcommand{\Gad}{  {\displaystyle \Ga}  }

\begin{document}     

\title[tt*-Toda of $A_n$ type]{The tt*-Toda equations of $A_n$ type
}  

\author{Martin A. Guest, Alexander R. Its, and Chang-Shou Lin}      

\date{}   

\begin{abstract} 
In previous articles we have studied the $A_n$ tt*-Toda equations of Cecotti and Vafa, giving details mainly for $n=3$.  Here we give a proof of the existence
and uniqueness of global solutions for any $n$, and a new treatment of their asymptotic data, monodromy data,
and Stokes data.
\end{abstract}


\maketitle 

\section{The equations}\label{1}

The tt*-Toda equations  of $A_n$ type are:
\begin{equation}\label{ost}
 2(w_i)_{\ttb}=-e^{2(w_{i+1}-w_{i})} + e^{2(w_{i}-w_{i-1})}, \ 
 w_i:\C^\ast\to\R, \ 
 i\in\Z
\end{equation}
where the functions $w_i$ are subject to the conditions
$w_i=w_{i+n+1}$ (periodicity),
$w_i=w_i(\vert t\vert)$
(radial condition),
$w_i+w_{n-i}=0$
(\ll anti-symmetry\rrr).
These equations are a version of the 2D periodic Toda equations, and they first
appeared as formula (7.4) in the article \cite{CeVa91} of Cecotti-Vafa.
We refer to that article (and our previous articles listed in the references) for background information and for some of the surpisingly deep relations with physics and geometry.
An important feature of (\ref{ost}) is that it is the Toda system associated
to a certain Lie algebra $\sl^\De_{n+1}\R$ (isomorphic to $\sl_{n+1}\R$), which
will be defined in the next section.

Results on the globally smooth solutions were given, mainly in the
case $n=3$, in our previous work \cite{GuLi14},\cite{GIL1},\cite{GIL2},\cite{GIL3}.  
Results for the case $n=1$ were given by McCoy-Tracy-Wu (see \cite{FIKN06}), 
and for the case $n=2$ by Kitaev \cite{Ki89},
and our method gives alternative proofs of these. 

The methods of \cite{GuLi14},\cite{GIL1},\cite{GIL2},\cite{GIL3} extend in principle to the case of general $n$.  Nevertheless it would be laborious to use exactly the same arguments.  In the current article we give a new proof of the existence
and uniqueness of global solutions. We also give a more systematic description of the monodromy data and a more efficient application of the Riemann-Hilbert method. Many of these improvements are due to our (implicit) use of the Lie-theoretic point of view from \cite{GH1},\cite{GH2}, although the proofs presented here can be understood without
knowledge of Lie theory.

The existence of a family of solutions parametrized by points of
\[
\cA=\{ m\in \R^{n+1} \st m_{i-1}-m_i\ge -1, m_i+m_{n-i}=0 \},
\]
can be established by p.d.e.\ methods 
(cf.\ our previous articles \cite{GuLi14},\cite{GIL1}, 
the Higgs bundle method of Mochizuki in \cite{MoXX},\cite{Mo14}, 
and the proof given here in section \ref{10}).
However
p.d.e.\  methods provide only minimal information on the properties of the solutions.  By using the 
Riemann-Hilbert method we shall be able to give more detailed information.

We consider first the
solutions parametrized by points of the interior 
\[
\mathring\cA
=\{ m\in \R^{n+1} \st m_{i-1}-m_i> -1, m_i+m_{n-i}=0 \}
\]
of this region;
we refer to these as \ll generic solutions\rrr. For such solutions we have:

{\em (I) The asymptotics of $w_0,\dots,w_n$ at $t=0$:}   for $0\le i\le n$,
\[
w_i(x)= -m_i \log x - \tfrac12\log l_i + o(1),\quad x\to 0
\]
where 
\[
l_i=
\textstyle
{(n+1)}^{m_{n-i} - m_i}
\frac
{
\Gad\left(
\frac{
m^\pr_{n-i} - m^\pr_{n-i+1} 
}{n+1}
\right)
}
{
\Gad\left(
\frac{
m^\pr_{i} - m^\pr_{i+1} 
}{n+1}
\right)
}
\frac
{
\Gad\left(
\frac{
m^\pr_{n-i} - m^\pr_{n-i+2} 
}{n+1}
\right)
}
{
\Gad\left(
\frac{
m^\pr_{i} - m^\pr_{i+2} 
}{n+1}
\right)
}
\cdots
\frac
{
\Gad\left(
\frac{
m^\pr_{n-i} - m^\pr_{n-i+n} 
}{n+1}
\right)
}
{
\Gad\left(
\frac{
m^\pr_{i} - m^\pr_{i+n} 
}{n+1}
\right)
},
\]
$m^\pr_i=m_i-i$, and $x=\vert t\vert$; 
for $i\in\Z$ we put $m^\pr_i = m^\pr_{i+n+1}+n+1$.
(For a subfamily of solutions
 this formula was already known from \cite{TrWi98}.)

{\em(II) The asymptotics of $w_0,\dots,w_n$ at $t=\infty$:}    for  $j=1,2,\dots,d$,
\begin{align*}
w_0(x) \sin j \tfrac{\pi}{n+1}& +
w_1(x) \sin 3j \tfrac{\pi}{n+1} +
\cdots +
w_{d-1}(x) \sin (2d-1)j \tfrac{\pi}{n+1} 
\\
&=
-\tfrac{n+1}8
\
s_j  
\
(\pi L_j x)^{-\frac12} e^{-2L_j x}
+ O(x^{-\frac32} e^{-2L_jx })
,\quad x\to \infty
\end{align*}
where 
$L_j= 2\sin \tfrac {j}{n+1} \pi$ and we write $n+1=2d$ or $n+1=2d+1$ according to
whether $n+1$ is even or odd.
The real number
$s_j$ is the $j$-th symmetric function of
$\om^{m_0+\frac n2}, \om^{m_1-1+\frac n2}, \dots, \om^{m_n-n+\frac n2}$,
where $\om=e^{{2\pi \i}/{(n+1)}}$. 
Thus we have $d$ linear equations which determine $w_0,\dots,w_{d-1}$;
the significance of these linear equations will be explained later.

The explicit relation between the asymptotic data at $t=0$ (the numbers $m_i$)
and the asymptotic data at $t=\infty$ (the numbers $s_i$) constitutes a
solution of the \ll connection problem\rr for solutions of (\ref{ost}).  These results are proved in section \ref{8} (Corollary \ref{final}).

{\em(III) Explicit expressions for the monodromy data (Stokes matrices and connection matrices) of the
isomonodromic o.d.e.\ (\ref{ode-hatal}) associated to $w$.}
This is a linear o.d.e.\ in a
complex variable $\la$ whose coefficients depend on $w$ (and hence on $t,\bar t$). It has poles of order $2$ at $\la=0,\infty$. 

The meaning of \ll isomonodromic\rr is that the monodromy data at these poles is independent of $t,\bar t$. Thus, to each solution $w_0,\dots,w_n$ of (\ref{ost}), we can assign a collection of
monodromy data.
We shall show in section \ref{8} that the monodromy data associated to $w_0,\dots,w_n$ is equivalent to the \ll Stokes numbers\rr  $s_i$ which appear in the above formula for the
asymptotics at $t=\infty$. 

In order to obtain these results on the global solutions, we make a thorough study of a wider class of
solutions, namely those defined on regions of the form $0<\vert t\vert< \eps$. 
We refer
to such solutions as \ll smooth near zero\rrr.
They are constructed from a  \ll subsidiary o.d.e.\rr which 
has poles of order $2,1$ at $\la=0,\infty$. 
For these solutions, the Stokes matrices are given in Theorem \ref{siofost} and the connection matrices are given in Theorem \ref{explicitE1withe}, Corollary \ref{eiofost}.

We show (Corollary \ref{global2}) that the global solutions can be characterized by a condition on the
connection matrices of the subsidiary o.d.e., which can be expressed most intrinsically as follows:

\begin{corollary}  A solution $w_0,\dots,w_n$ which is smooth near zero
is a global solution if and only if all Stokes matrices and connection matrices for the subsidiary o.d.e.\  lie in
the group $\SL^\De_{n+1}\R$.
\end{corollary}

A more precise statement is given in Corollary \ref{global2}.
Our proof  is indirect; it relies on explicit computations of
Stokes matrices  and connection matrices.

We conclude with some comments on the 
\ll non-generic global solutions\rrr, i.e.\ those which correspond to the boundary
points of the region $\cA$. It turns out that most of our methods apply equally well to this case.
We have restricted attention to the generic solutions in this article only because the construction of solutions near zero is easier in this case.  Our argument can be summarized as follows.  First we construct solutions which are smooth near zero. 
Then we
use the Riemann-Hilbert method to extend the range of solutions towards infinity, and 
make contact with the global solutions. To complete the
argument we use the p.d.e.\ results of section \ref{10.1}.  The results of section \ref{10.2} 
are not actually needed at this point.

In fact it is possible to proceed in the other direction.  Namely, we could start
with the p.d.e.\ results of section \ref{10}, which apply to all global solutions, generic or non-generic. The
Riemann-Hilbert method of section \ref{7} applies without change, taking
the monodromy data (Stokes matrices and connection matrices) 
as an Ansatz, and it produces solutions which are smooth near infinity. The results of section \ref{10.2} are now essential, in order
to connect the Stokes data with the asymptotic data at zero.  

Arguing in
\ll reverse order\rr in this way, we see that the above statements (I), (II), (III)
hold also the non-generic case, except that (I) must be replaced by the following weaker statement:
\[
\lim_{x\to 0} \frac{w_i(x)}{\log x} = -m_i.
\]
Thus this argument alone would not produce the second
term in the asymptotics of $w$ at zero. Nevertheless, it shows that 
the Stokes matrices and connection matrices for the global solutions in the non-generic case are
given by exactly the same formulae as in the generic case.

To produce the second
term in the non-generic case, the \ll forward argument\rr (starting with solutions near
zero) has to be modified, and then the second term will be quite different from that in (I) above.
In the case $n=3$, these modifications can be found in \cite{GIL3}.

Acknowledgements:   
The first author was partially supported by JSPS grant 18H03668.
The second author was partially supported by NSF
grant DMS:1955265, by RSF grant No.\ 22-11-00070,
and  by a Visiting Wolfson Research Fellowship from the Royal Society.
The authors thank Nan-Kuo Ho for comments on a draft version of the paper.

\section{The zero curvature formulation}\label{2}

In this section we review the zero curvature formulation of (\ref{ost}),
and a construction of local solutions near $t=0$.

\subsection{The main connection form $\al$}\label{2.1} \ 

Equation (\ref{ost}) is the zero curvature condition $d\al+\al\wedge\al=0$
for the following connection form:

\begin{definition}\label{alpha}
$
 \al = (w_t+\tfrac1\la W^T)dt + (-w_{\tbar}+\la W)d\tbar,
$
where
 \[
 w=\diag(w_0,\dots,w_n),\ \ 
 W=
 \left(
\begin{array}{c|c|c|c}
\vphantom{(w_0)_{(w_0)}^{(w_0)}}
 & \!e^{w_1\!-\!w_0}\! & &  
 \\
\hline
  &  & \  \ddots\   & \\
\hline
\vphantom{\ddots}
  & &  &  e^{w_n\!-\!w_{n\!-\!1}}\!\!\!
\\
\hline
\vphantom{(w_0)_{(w_0)}^{(w_0)}}
\!\! e^{w_0\!-\!w_n} \!\!  & &  &  \!
\end{array}
\right).
\]
Here $W^T$ denotes the transpose of $W$, and $\la$ is a complex parameter.
\end{definition}

\no At this point we just take $w=w(t,\bar t)$ to be smooth on some non-empty open subset of $\C^\ast$. 
Thus, $\al$ is an $\sl_{n+1}\C$-valued connection form on that open set.

It will sometimes be convenient to write $\al=\al^\pr dt + \al^\prr d\tbar$ to indicate the $(1,0)$ and $(0,1)$ parts, and to write
$\al^T=A^\pr dt + A^\prr d\tbar$ for the transpose.
Using this notation we can say that (\ref{ost}) is the compatibility condition for the linear system
\begin{align}\label{ode-al}
\begin{cases}
\Psi_t&=A^\pr \Psi\\
\Psi_\tbar&=A^\prr\Psi
\end{cases}
\end{align}
where $\Psi$ takes values in $\SL_{n+1}\C$.   

We shall make use of the following automorphisms.  
For $X\in \sl_{n+1}\C$, $V\in \SL_{n+1}\C$, we define
\begin{align*}
\tau(X)&=d_{n+1}^{-1} X d_{n+1}  & \tau(V)&=d_{n+1}^{-1} V d_{n+1}
\\
\si(X)&=-\De\, X^T\De  &  \si(V)&=\De\, V^{-T}\,\De
\\
c(X)&=\De \bar X  \De  &  c(V)&=\De \bar V  \De
\end{align*}
where
$d_{n+1}=\diag(1,\om,\dots,\om^n)$,  $\om=e^{{2\pi \i}/{(n+1)}}$,
and $\De$ is the anti-diagonal matrix with $1$ in
positions $(i,n-i)$, $0\le i\le n$, and $0$ elsewhere.

By inspection we see that the connection form $\al$ satisfies:

\no{\em Cyclic symmetry: }  $\tau(\al(\la))=\al(e^{{2\pi \i}/{(n+1)}} \la)$

\no{\em Anti-symmetry: }  $\si(\al(\la))=\al(-\la)$

\no{\em $c$-reality: }  $c(\al^\pr(\la))=\al^\prr(1/\bar\la)$.

\no The $c$-reality property implies that 
\begin{equation}\label{realvalued}
\al\vert_{ |\la|=1 } \ \text{is an $\sl^\De_{n+1}\R$-valued connection form,}
\end{equation}
where 
\[
\sl^\De_{n+1}\R = \{ X\in \sl_{n+1}\C \st \De \bar X \De = X\}
\]
is the fixed point set of $c$.  
We note that 
$\sl^\De_{n+1}\R = P \, \sl_{n+1}\R\,  P^{-1} \cong  \sl_{n+1}\R$ and 
$\SL^\De_{n+1}\R = P\,  \SL_{n+1}\R\,  P^{-1} \cong \SL_{n+1}\R$,
where
\newcommand{\qmaspace}{\hspace{.08cm}}
\newcommand{\andq}{\qmaspace & \qmaspace}
\begin{align*}
P=
\tfrac{ 1-i}2
\bp
1 \andq \andq \andq i\\
 \andq \qmaspace\ddots\qmaspace \andq \qmaspace\iddots\qmaspace \andq \\
  \andq \iddots \andq \ddots \andq \\
i \andq \andq \andq 1
\ep.
\end{align*}
Here the diagonal entries are all $1$ and the anti-diagonal entries are all $i=\ii$, except that
the \ll middle\rr entry is $1+i$ if $n+1$ is odd. 

\begin{remark}
In the language of the theory of integrable systems, the linear system 
(\ref{ode-al})
forms {\em the Lax pair}
for the nonlinear p.d.e.\
(\ref{ost}).  
This Lax pair  was first found  by Mikhailov \cite{Mi81}. 
More precisely, in \cite{Mi81}, a  hyperbolic version of the  2D Toda 
equations was considered, i.e., (\ref{ost}) with the replacement,
$
(w_i)_{t{\bar{t}}} \to (w_i)_{\xi \eta}, 
$
$\xi, \eta \in \R$,
and neither  anti-symmetry nor radial conditions were assumed. 
\qed
\end{remark}

\begin{remark}
Using the Riemann-Hilbert method, Novokshenov \cite{No85}  constructed the long time asymptotics of the solution
of the Cauchy problem for the hyperbolic version  of equation (\ref{ost}),
considered in the laboratory coordinates
$x = \xi + \eta, \,\, t = \xi - \eta$. From the physical and analytical points of view our problem 
is very different from the one considered in \cite{No85}, but there are parallels. 
\qed
\end{remark}

\subsection{The subsidiary connection form $\om$} \ 

Solving (\ref{ost}) is equivalent to constructing connection forms $\al$ of the above type. We are going to construct 
$\al$ from a simpler connection form $\om$. 

\begin{definition}\label{omega}
Let\footnote{The connection form $\om$ should not be confused with the root of unity 
$\om=e^{{2\pi \i}/{(n+1)}}$ of course; the context will make the intended meaning clear.}
$\om=\tfrac1\la\eta dz$,
where
\begin{equation*}
\eta=
\bp
 & & & p_0\\
 p_1 & & & \\
  & \ddots & & \\
   & & p_n &
   \ep
\end{equation*}
and $p_i(z)=c_i z^{k_i}$ with $c_i>0,k_i\ge -1$ and $p_i=p_{n-i+1}$.
\end{definition}

Writing $\om^T=Bdz$, we have an associated linear o.d.e.
\begin{equation}\label{ode-om}
\Phi_z=B\Phi
\end{equation}
where $\Phi$ takes values in $\SL_{n+1}\C$.

By inspection we see that the connection form $\om$ satisfies:

\no{\em Cyclic symmetry: }  $\tau(\om(\la))=\om(e^{{2\pi \i}/{(n+1)}} \la)$

\no{\em Anti-symmetry: }  $\si(\om(\la))=\om(-\la)$.

\subsection{A family of local solutions near $t=0$}\label{2.3} \ 

We shall now construct some specific local solutions $w$ of (\ref{ost}).

If all $k_i>-1$ we have a unique local holomorphic\footnote{When $k_i\notin\Z$, \ll holomorphic\rr should be understood in the multi-valued sense.}
solution $\Phi=\Phi(z,\la)$ of \ref{ode-om} with $\Phi(0,\la)=I$.  
We refer to this case (i.e.\ where all $k_i>-1$, as opposed to $k_i\ge-1$) as the {generic} case. 
{\em From now on --- up to and including section \ref{9} --- we focus on the generic case. }

In this section we write
\[
L=\Phi^T,
\]
so that $L$ is the unique local solution of the o.d.e.\ 
\begin{equation}\label{holeqn}
L^{-1}L_z=\tfrac1\la \eta,\quad L\vert_{z=0}=I.
\end{equation}
Restricting to $S^1=\{ \la\in\C\st \vert\la\vert=1 \}$, we may regard $L$ as a
map (from a neighbourhood of $z=0$ in $\C$) to the loop group $\La \SL_{n+1}\C$.
For the next step we use the Iwasawa factorization for 
$\La SL_{n+1}\C$ with respect to its real form $\La SL^\De_{n+1}\R$, which can be described as follows.

First, we have an (elementary) Lie algebra decomposition  
\begin{equation}\label{iwa-vec}
\La \sl_{n+1}\C = \La \sl^\De_{n+1}\R \ +\  \La_+ \sl_{n+1}\C,
\end{equation}
where $\La_+ \sl_{n+1}\C$ denotes the space of loops 
$f(\la)=\sum_{i\in\Z} f_i \la^i$ in $\La \sl_{n+1}\C$ 
such that $f_i=0$ when $i<0$,
and 
\[
\La \sl^\De_{n+1}\R = \{ f\in \La \sl_{n+1}\C \st c(f(\la)) = f(1/\bar\la) \}.
\]
Concretely, let us write
$f= f_{<0} + f_0 + f_{>0}$ $(= f_{\le0} + f_{>0})$
for the decomposition of $f$ according to
terms $f_i\la^i$ with $i<0$, $i=0$, $i>0$ (respectively).  Then the decomposition
(\ref{iwa-vec}) follows from writing
$f= (f_{\le 0} + \hat c (f_{\le 0}) + (f_{>0} - \hat c (f_{\le 0}))$,
where $\hat c (f)(\la)= c(f(1/\bar\la))$. 
We note that the intersection
$\La \sl^\De_{n+1}\R \ \cap \  \La_+ \sl_{n+1}\C$
is $\sl^\De_{n+1}\R$, i.e.\ the space of constant loops, so
the decomposition (\ref{iwa-vec}) is not a direct sum. 

Next, the symmetries $\tau,\si$ may be imposed on (\ref{iwa-vec}), in
the sense that (\ref{iwa-vec}) remains valid if we assume that all loops
satisfy the \ll twisting conditions\rr
$\tau(f(\la))=f(e^{{2\pi \i}/{(n+1)}} \la)$, 
$\si(f(\la))=f(-\la)$.
The intersection of the twisted loop algebras in (\ref{iwa-vec}) is the set of matrices
$\diag(z_0,\dots,z_n)$ with $z_i+z_{n-i}=0$ (from the twisting conditions), such that
$z_i=\bar z_{n-i}$ (from the definition of $\sl^\De_{n+1}\R$). Thus all $z_i$
are pure imaginary. 

It follows that any (twisted) $f$ has a unique splitting
\begin{equation}\label{iwa-f}
f = f_\R + f_+
\end{equation}
with
$f_\R\in\La \sl^\De_{n+1}\R$, $f_+\in \La_+ \sl_{n+1}\C$
if we insist that the constant term of $f_+$ is real.  Thus, after imposing these
conditions, we see that (\ref{iwa-vec}) becomes a direct sum decomposition.

Finally, the inverse function theorem guarantees --- locally, near the identity loop --- the existence of a corresponding multiplicative splitting of
Banach Lie groups.  

This particular  Iwasawa factorization is all that we need. 
We refer to \cite{PrSe86},\cite{Gu97},\cite{BaDo01},\cite{Ke99} for the general theory of Iwasawa factorizations for loop groups.

It follows that we have (locally, near $z=0$) a factorization
\begin{equation}\label{iwa}
L=L_\R L_+
\end{equation}
of the $\La \SL_{n+1}\C$-valued function $L$
with $L_\R\vert_{z=0}=I=L_+\vert_{z=0}$. 
The normalization of $L_+$ means that 
$L_+(z,\bar z,\la)=b(z,\bar z)+ O(\la)$, where $b=\diag(b_0,\dots,b_n)$ and all $b_i>0$. We have $b(0)=I$.

As explained in section 2 of \cite{GIL3}, from each such $L_\R$ we can obtain a connection
form $\al$, and hence a solution of (\ref{ost}) near $t=0$.  We review this next, in order to introduce some essential notation.  
The required connection form $\al$ will be of the form
\begin{equation}\label{alpha-iwa}
\al=
(L_\R G)^{-1} (L_\R G)_t dt + (L_\R G)^{-1} (L_\R G)_{\bar t} d{\bar t},
\end{equation}
where $G$ is a certain diagonal matrix (depending on $t,\bar t$) and $t$ is a new variable.  To define these we introduce the following notation.

\begin{definition}\label{Nmchat}
Let $N=\sum_{i=0}^n (k_i+1)$, $c=\prod_{i=0}^n c_i$.
Matrices $m=\diag(m_0,\dots,m_n),\hat c=\diag(\hat c_0,\dots,\hat c_n)$ are defined by:
\begin{equation}
m_{i-1}-m_i=-1+\tfrac{n+1}N (k_i+1)
\end{equation}
\begin{equation}
\hat c_{i-1}/\hat c_i = \left( \tfrac{n+1}N \right)^{m_i-m_{i-1}} c_i c^{-(k_i+1)/N} .
\end{equation}
\end{definition}

\no It is easy to see that $m,\hat c$ are well-defined and unique.  They satisfy
$m_i+m_{n-i}=0, \hat c_{i}\hat c_{n-i}=1$ (because $p_i=p_{n-i+1}$).

\begin{definition}\label{thG}
Let 
\begin{equation}\label{tandz}
t=\tfrac{n+1}N \ c^{\frac1{n+1}} \ z^{\frac N{n+1}}
\end{equation}
and
\begin{equation}
h = \diag(\hat c_0,\dots,\hat c_n)\, t^{\diag(m_0,\dots,m_n)}=\hat c \, t^m.
\end{equation}
Let
\begin{equation}\label{defofG}
G=\diag( \vert h_0\vert/h_0, \dots, \vert h_n\vert/h_n )=\vert h\vert/h.
\end{equation}
\end{definition}

The definitions of $m,\hat c, t, h$ are made in order to have the simple formula
\begin{equation}\label{convert}
h^{-1} \, \tfrac{n+1}N z \eta^T \, h = t\Pi, 
\quad
\Pi=(\de_{i,i+1})_{0\le i\le n}=
\bp
  & \!1\! & & \\
 & & \ \ddots\  & \\
  & & & \!1\\
1\!   & & &
\ep.
\end{equation}
Computation now shows that
$
(L_\R G)^{-1} (L_\R G)_t dt+(L_\R G)^{-1} (L_\R G)_{\bar t} d\bar t
$
has the form of $\al$ in Definition \ref{alpha},  if we introduce
\begin{equation}\label{wdef}
w_i=\log ( b_i/\vert h_i\vert).
\end{equation}

\begin{proposition}\label{localsol}
Let $\om$ be as in Definition \ref{omega}.  Assume $k_i>-1$ for all $i$. Then the above construction produces a solution $w$ of (\ref{ost}) on some punctured
disk of the form $0<\vert t\vert < \eps$, with
\[
w_i= -m_i \log|t| - \log \hat c_i + o(1)
\]
as $t\to 0$.  
\end{proposition} 

\begin{proof}  We have $b=I+o(1)$ in a neighbourhood of $t=0$,  
and $\log\vert h_i\vert = m_i\log \vert t\vert + \log \hat c_i$.
\end{proof}

From Definition \ref{Nmchat} we see that the
condition $k_i>-1$ corresponds to the condition 
$m_{i-1}-m_i>-1$. 

It is easy to show
that $m_{i-1}-m_i\ge -1$ is a necessary and sufficient condition for the
existence of local (radial) solutions near $t=0$ satisfying $w_i\sim -m_i \log|t|$.
These $m$ constitute a compact convex region $\cA$, and the generic case 
corresponds to the interior  $\mathring\cA$ of this region. 

\begin{remark}
The arguments we have used in passing from the holomorphic equation (\ref{holeqn})
to the  zero curvature form (\ref{alpha-iwa})
are very similar to the technique used by Krichever in \cite{Kr80} 
where he introduced
a nonlinear analogue of d'Alembert's formula for the equations of 
the principal chiral field and the Sine-Gordon equation.
In our construction, a key role is played by the Iwasawa factorization (\ref{iwa}), while,
in Krichever's paper, a Riemann-Hilbert problem posed on two
circles (around zero and infinity) is used. In fact, properly 
modified for the case of (\ref{ost}),
Krichever's technique can be used to produce a proof
 of the
factorization (\ref{iwa})
which is independent from loop group theory.
Although this relation between the {\em group-theoretical}  Iwasawa factorization theory
and the {\em  analytic} Riemann-Hilbert theory is well known, we have 
found that Krichever's paper
\cite{Kr80} adds interesting  new aspects, which we plan to address in a forthcoming publication.
\qed
\end{remark}

\section{Isomonodromy formulation}\label{3}

In this section we review the isomonodromy formulation of (\ref{ost}).  This will eventually
(in section \ref{7}) be used in a construction of local solutions near $t=\infty$.

\subsection{The main isomonodromic connection form $\hat\al$}\label{3.1} \ 

As explained in section 1 of \cite{GIL1} (and section 2 of \cite{GIL3}), the radial property $w=w(\vert t\vert)$ leads to a homogeneity property of $\al$, and hence to another connection form $\hat\al$. 

\begin{definition}\label{hatalpha}  
$
\hat\al=
\left[
- \tfrac{t}{\la^2}
\ W^T
- \tfrac1\la xw_x + \tbar \,W
\right]
d\la,
$
where $x=\vert t\vert$.
\end{definition}

Writing $\hat\al^T=\hat A d\la$, we have an associated linear o.d.e.
\begin{equation}\label{ode-hatal}
\Psi_\la=\hat A\Psi.
\end{equation}
This o.d.e.\ is meromorphic in $\la$ with poles of order $2$ at $\la=0$ and $\la=\infty$. 

The argument of section 2.4 of \cite{GIL3} shows that
\[
\hat\al=(gL_\R G)^{-1}(gL_\R G)_\la d\la,
\]
where 
\begin{equation}\label{defofg}
g=\la^m.
\end{equation}
From the
previous section we have
\begin{align*}
\al&=
(L_\R G)^{-1} (L_\R G)_t dt + (L_\R G)^{-1} (L_\R G)_{\bar t} d{\bar t}
\\
&=
(gL_\R G)^{-1} (gL_\R G)_t dt + (gL_\R G)^{-1} (gL_\R G)_{\bar t} d{\bar t}.
\end{align*}
In other words, $gL_\R G$ is a fundamental solution matrix 
for the (linear system corresponding to the) combined connection $d+\al+\hat\al$. Hence this connection is flat. 

Here we are assuming (as in section \ref{2.1}) that we have a solution 
$w=w(\vert t\vert)$ for $t$ in some open
set. Then $d+\al+\hat\al$ is defined for such $t$ and any $\la\in\C^\ast$. 

Flatness implies that the $\la$-monodromy of $gL_\R G$ is independent of $t$.
Namely,
if analytic continuation of $gL_\R G$ around the origin in the $\la$-plane produces
$CgL_\R G$ for some $C=C(t,\bar t)$, then, after substituting into (\ref{ode-al}),(\ref{ode-hatal}), we deduce that $C_t=C_{\bar t}=0$.  In fact (\cite{JMU81}; see chapter 4 of \cite{FIKN06}), flatness implies that the system (\ref{ode-hatal}) is {\em isomonodromic} in the sense that all Stokes matrices and connection matrices of (\ref{ode-hatal}) (as well as the monodromy matrices) are independent of $t$. 

For future calculations it will be convenient to put $\ze=\la/t$.  This converts 
$\hat\al$ to
$
\hat\al(\ze)=
\left[
-\tfrac{1}{\ze^2} W^T - \tfrac{1}{\ze} xw_x + x^2 W
\right]
d\ze
$
and
(\ref{ode-hatal})
to
\begin{equation}\label{psizeta}
\Psi_\ze=
\left[
-\tfrac{1}{\ze^2} W - \tfrac{1}{\ze} xw_x + x^2 W^T
\right]
\Psi
\end{equation}
(recall that $x=\vert t\vert$).
We shall investigate the monodromy data of this meromorphic o.d.e.
 
By inspection we see that the connection form $\hat\al$ satisfies:

\noindent{\em Cyclic symmetry: }  $\tau(\hat\alpha(\zeta))=\hat\alpha(e^{{2\pi \i}/{(n+1)}} \zeta)$

\noindent{\em Anti-symmetry: }  $\sigma(\hat\alpha(\zeta))=\hat\alpha(-\ze)$

\noindent{\em $c$-reality: }  $c(\hat\alpha(\zeta))=\hat\alpha(1/(x^2\bar\ze\vphantom{\tfrac{a}{a}}))$

\noindent {\em $\theta$-reality: } $\th(\hat{\alpha}(\ze))=\hat{\alpha}(\bar{\ze})$.

\no Here $\tau,\si,c$ are as in section \ref{2} and the new involution $\th$ is defined by:
\[\th(X)=\bar X \quad\quad\quad
\th(V)=\bar V
\]
for $X\in \sl_{n+1}\C$, $V\in \SL_{n+1}\C$.

\subsection{The subsidiary connection form $\hat\om$}\label{3.2} \ 

Just as $\hat\al$ was constructed from $\al$, we can construct (see section 2 of \cite{GIL3}) a
meromorphic connection form $\hat\om=(gL)^{-1} (gL)_\la d\la$ from $\om$:

\begin{definition}\label{hatomega}  
$
\hat\om=
\left[
-\tfrac{n+1}N \tfrac{z}{\la^2}
\ \eta
+\tfrac1\la
\ m
\right]
d\la
$.
\end{definition}

Writing $\hat\om^T=\hat Bd\la$, we have an associated linear o.d.e.
\begin{equation}\label{ode-hatom}
\Phi_\la=\hat B\Phi.
\end{equation}
Again, the extended connection $d+\om+\hat\om$ is flat.  

Putting $\ze=\la/t$ converts $\hat\om$ to
$\hat\om(\ze)=
\left[
-\tfrac1{\ze^2} h^{-1} \Pi^T h + \tfrac1\ze m
\right]d\ze
$
and equation (\ref{ode-hatom}) to
\begin{equation}\label{phizeta}
\Phi_\ze= 
\left[
-\tfrac1{\ze^2} h \Pi h^{-1} + \tfrac1\ze m
\right]
\Phi.
\end{equation}
Here we are making use of (\ref{convert}).

By inspection we see that the connection form $\hat\om$ satisfies:    

\noindent{\em Cyclic symmetry: }  $\tau(\hat\om(\zeta))=\hat\om(e^{{2\pi \i}/{(n+1)}} \zeta)$

\noindent{\em Anti-symmetry: }  $\sigma(\hat\om(\zeta))=\hat\om(-\zeta)$

\noindent {\em $\theta$-reality: } $\th(\hat{\om}(\ze))=\hat{\om}(\bar{\zeta})$.

\subsection{How the isomonodromic connection forms will be used.} \ 

In section \ref{2.3} we gave a construction of local solutions (near $t=0$) of the tt*-Toda equations (\ref{ost}). This was done by constructing connection forms $\al$ --- whose coefficients contain solutions $w$ of (\ref{ost}) --- from connection forms $\om$ --- whose coefficients are given explicitly in terms of the data $c_i>0$, $k_i>-1$.  
The isomonodromic connection forms $\hat\al,\hat\om$ will give another way of constructing solutions. 

To exploit this, we shall first (in sections \ref{4}-\ref{6}) compute the monodromy data of the known local solutions which are smooth near zero. Motivated by this, we shall (in section \ref{7}) formulate and solve a Riemann-Hilbert problem which produces solutions which are smooth near infinity.   

By \ll monodromy data\rr we mean Stokes matrices and monodromy matrices associated to solutions of the linear system at each irregular pole (just monodromy matrices, in the case of simple poles), together with connection matrices which relate solutions at different poles.   In section \ref{4} we {\em define} the relevant Stokes matrices, and in section \ref{5} the connection matrices.  In section \ref{6} we {\em compute} all this data corresponding to the local solutions near $t=0$ of (\ref{ost}), first for
$\hat\om$ then for $\hat\al$.  This information will be used in sections \ref{7} and \ref{8} to identify which of these local solutions are actually global solutions.

\section{Definition of Stokes matrices}\label{4} 

In this section we define the Stokes matrices (and the monodromy matrices) for both $\hat\al$ and $\hat\om$ when $n+1$ is even. Their computation will be carried out in section \ref{6}. The modifications needed when $n+1$ is odd will be given in an appendix (section \ref{9}).

At first we assume $n+1\ge 3$; special features of the case $n+1=2$ will be explained at
the ends of the relevant subsections.

\subsection{Stokes matrices for $\hat\al$ at $\ze=0$}\label{stokesforhatal-zero}  \

These depend on a choice of formal solution at each pole. Following \cite{GIL1},\cite{GIL2},\cite{GIL3}
we diagonalize the leading term of equation (\ref{psizeta}) by writing
\[
W=e^{-w\, }\Pi\,  e^w,
\quad
\Pi
=
 \Omega \, d_{n+1}\,  \Omega^{-1}
\]
where $\Pi$ is as in formula (\ref{convert}) and  $\Omega=(\omega^{ij})_{0\le i,j\le n}$. Then there is a (unique) formal solution  
of equation (\ref{psizeta}) of the form
\[
\textstyle
\Psiz_f(\zeta) =  e^{-w}\, \Omega \left( I + \sum_{k\ge 1} \psiz_k \zeta^k \right) e^{\frac1\zeta  d_{n+1}}.
\]
The formal monodromy is trivial, i.e.\ there are no logarithmic terms  in the exponential factor (see section 2 of \cite{GIL1} and Lemma 4.4 of \cite{GH2}).  
For reasons to be given later, we shall also use the formal solution
$
\tPsiz_f(\zeta) =\Psiz_f(\zeta)  \, d_{n+1}^{\frac12}
$
where
$d_{n+1}^{\frac12}=\diag(1,\omega^{\frac12},\omega^{1},\dots,\omega^{\frac{n}2})$,
$\om=e^{{2\pi \i}/{(n+1)}}$.

Each Stokes sector  $\Omz_k$ at $\ze=0$ in the $\ze$-plane supports 
holomorphic solutions
$\Psiz_k,\tPsiz_k$ 
which are uniquely characterized by the properties
$\Psiz_k(\ze)\sim\Psiz_f(\ze)$, $\tPsiz_k(\ze)\sim\tPsiz_f(\ze)$
as $\zeta\to0$ in $\Omz_k$.  We
use the following Stokes sectors which were defined in \cite{GH1} (see section 4 of \cite{GIL1} for a more detailed explanation).  

As initial Stokes sector we take
\[
\Omz_1=\{ \zeta\in\C^\ast \ \vert \ -\tfrac{\pi}{2}-\tfrac{\pi}{n+1}<\text{arg}\zeta <\tfrac\pi2\}.
\]
We regard this as a subset of
the universal covering
$\tilde\C^\ast$, and define further Stokes sectors by
\[
\Omz_{k+\scriptstyle\frac{1}{n+1}}= e^{-\frac\pi{n+1}\i} \Omz_k,
\quad
k\in\tfrac1{n+1}\Z.
\]
The solutions $\Psiz_k,\tPsiz_k$ extend by analytic continuation to the whole of $\tilde\C^\ast$.

It is convenient to write
\[
\Omz_k
=\{ \zeta\in\tilde\C^\ast \st
\thz_k-\tfrac{\pi}{2}<\text{arg}\zeta <
\thz_{k-\scriptstyle\frac{1}{n+1}}+
\tfrac\pi2\}
\]
where
\[
\thz_k=-\tfrac{1}{n+1}\pi - (k-1)\pi,
\quad
k\in\tfrac1{n+1}\Z.
\]
\begin{figure}[h]
\begin{center}
\includegraphics[angle=90,origin=c,scale=0.3, trim= 0 200 0 150]{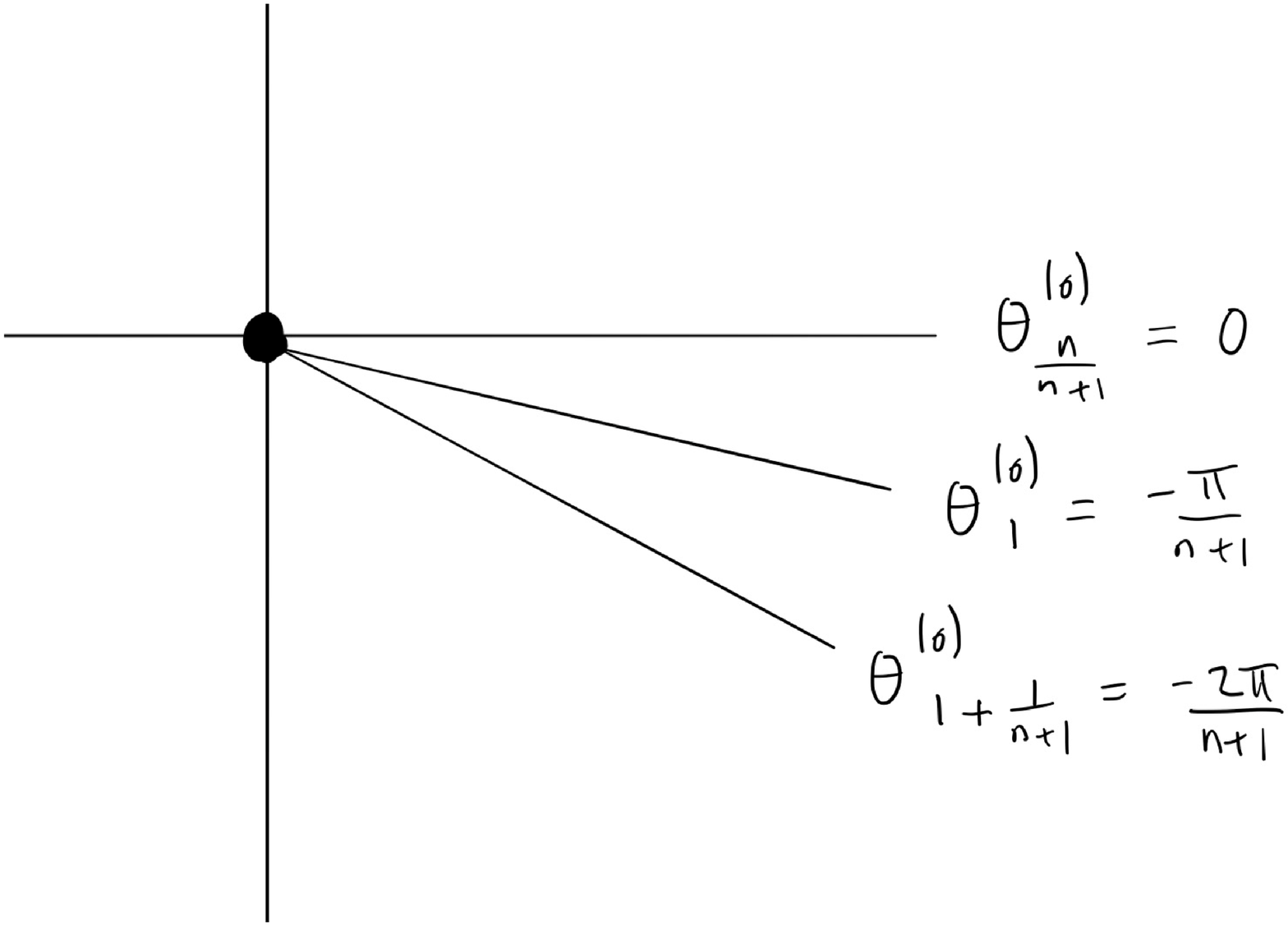}
\end{center}
\caption{Singular directions $\thz_k=-\tfrac{\pi}{n+1} - (k-1)\pi$.}\label{rays}
\end{figure}
Then the intersection of successive Stokes sectors 
\[
\Omz_k
\cap
\Omz_{k+\scriptstyle\frac1{n+1}}=
\{ \zeta\in\tilde\C^\ast \ \vert \ 
\thz_k-\tfrac{\pi}{2}<\text{arg}\zeta <
\thz_k +
\tfrac\pi2\}
\]
is seen to be the sector of width $\pi$ bisected by the ray with angle $\thz_k$. 
On this intersection, successive solutions must differ by (constant) matrices, i.e.\ 
\[
\Psiz_{k+\scriptstyle\frac1{n+1}} = \Psiz_k \Qz_k,
\]
and these \ll Stokes factors\rr $\Qz_k$ are thus indexed by the \ll singular directions\rr $\thz_k$.

It is conventional (cf.\ \cite{FIKN06}, section 1.4)
to define Stokes matrices 
$\Sz_k$ by
\[
\Psiz_{k+1}=\Psiz_k \Sz_k,
\]
for $k\in\Z$.
Thus each Stokes matrix is a product of $n+1$ Stokes factors:
\[
\Sz_k=\Qz_{k}\Qz_{k+\scriptstyle\frac1{n+1}}\Qz_{k+\scriptstyle\frac2{n+1}}\dots
\Qz_{k+\scriptstyle\frac{n}{n+1}}.
\]
\begin{figure}[h]
\begin{center}
\includegraphics[angle=90,origin=c,scale=0.3, trim= 0 200 0 150]{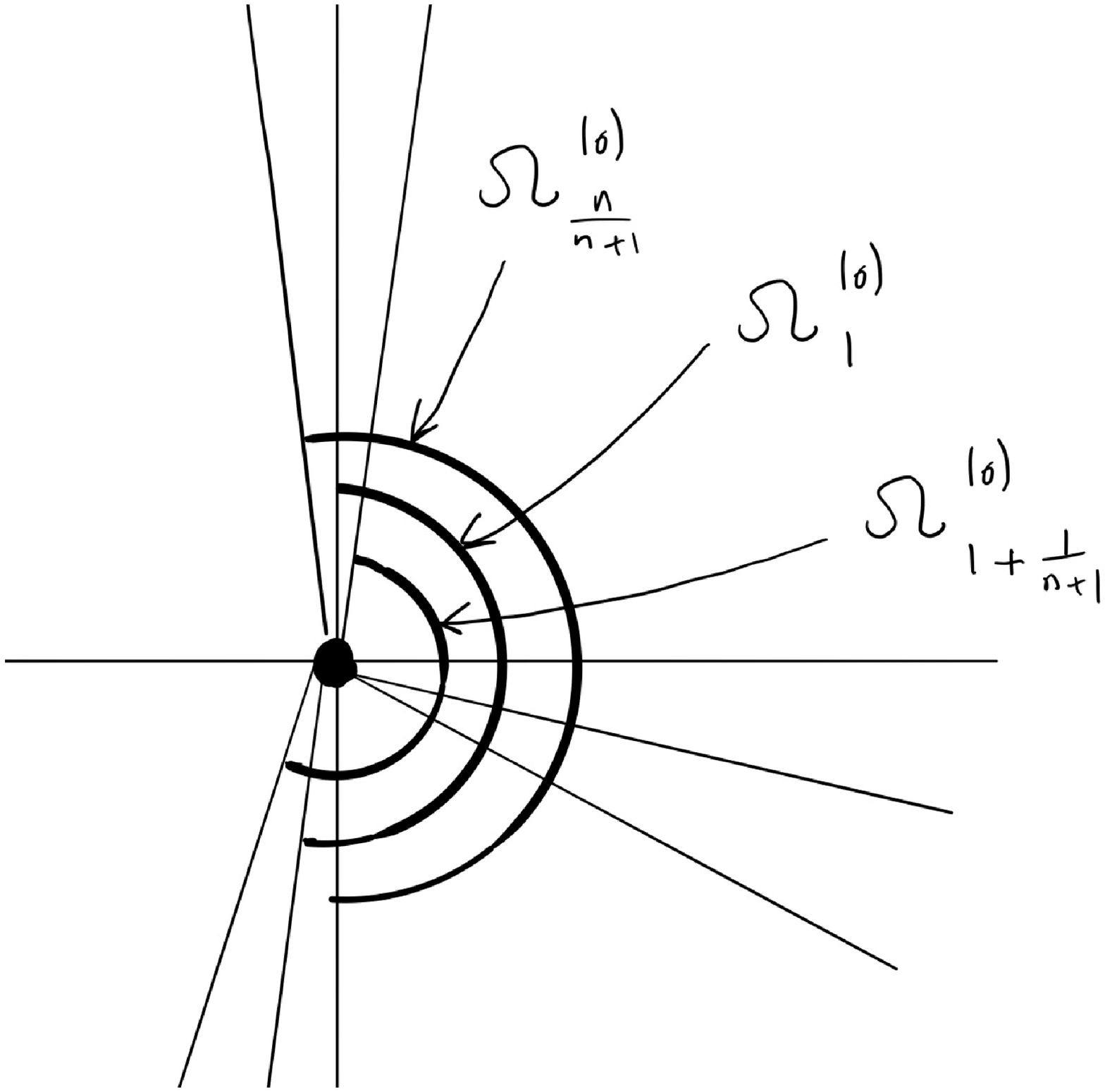}
\end{center}
\caption{Stokes sectors $\Omz_k$.  The ray $\thz_k$ bisects 
$\Omz_k
\cap
\Omz_{k+\scriptstyle\frac1{n+1}}$.}\label{sectors}
\end{figure}
The monodromy of the solution $\Psiz_k$ is given by
\begin{equation}\label{Psiz-mon}
\Psiz_k(e^{2\pi\i}\ze)=\Psiz_k(\ze) \Sz_k\Sz_{k+1}.
\end{equation}
In this formula, the left hand side indicates analytic continuation of $\Psiz_k(\ze)$ around the pole $\ze=0$ in the positive direction.  The formula
is an immediate consequence of 
\begin{equation}\label{con-zero}
\Psiz_{k+2}(e^{-2\pi\i}\ze)=\Psiz_k(\ze),
\end{equation}
which holds
because both sides have the same asymptotic expansion $\Psiz_f(\ze)$ 
as $\ze\to 0$ in $\Omz_k$. 

Thus the Stokes factors $\Qz_k$ are the fundamental data at $\ze=0$. 
From (\ref{con-zero}) we obtain
$\Qz_{k+2}=\Qz_k$ and $\Sz_{k+2}=\Sz_k$, so there are 
$2n+2$ independent Stokes factors, or
$2$ independent Stokes matrices.
 
From the symmetries of $\hat\al$ at the end of section \ref{3} we
obtain corresponding symmetries of 
$\Psiz_k,\tPsiz_k$ and $\Qz_k,\tQz_k$ (cf.\ section 2 of \cite{GIL2} and
section 3 of \cite{GH1}). 
The cyclic symmetry formula
\begin{equation}\label{Qz-cyc}
\Qz_{k+\scriptstyle\frac2{n+1}} =
\Pi\  \Qz_k \  \Pi^{-1}
\end{equation}
shows that all Stokes factors are determined by any two consecutive 
Stokes factors, and that
the monodromy matrix of $\Psiz_1$ can be
written
\[
\Sz_1 \Sz_2 = (\Qz_1\Qz_{1+\scriptstyle\frac1{n+1}}\Pi)^{n+1}.
\]
\begin{definition}\label{matrixM}
$\Mz=\Qz_1\Qz_{1+\scriptstyle\frac1{n+1}}\Pi$.
\end{definition}
\no
Thus the matrix $\Mz$ determines the monodromy.  We shall see later (see Proposition \ref{charpoly}) that it also determines 
the individual $\Qz_1$, $\Qz_{1+\scriptstyle\frac1{n+1}}$, and hence all Stokes factors. It is, therefore,  the key object in describing the Stokes data.

Using $\tPsiz_k$ instead of $\Psiz_k$ we can define
$\tQz_k,\tSz_k$ in a similar way.  From this we have
\begin{equation}\label{tQzandQz}
\tQz_k= d_{n+1}^{-\frac12} \Qz_k d_{n+1}^{\frac12}.
\end{equation}
A corresponding definition of $\tMz$ will be given later, in formula (\ref{matrixMtilde}).

\subsection{Stokes matrices for $\hat\al$ at $\ze=\infty$}\label{stokesforhatal-infinity}  \

The definitions at $\ze=\infty$ are very similar to the definitions at $\ze=0$.  
We diagonalize $W^T$ by writing
\[
W^T= e^w \Om^{-1} d_{n+1} \Om e^{-w}.
\]
Then we have formal solutions
\[
\Psii_f(\zeta) = 
\textstyle
e^w \Om^{-1} \left( I + \sum_{k\ge 1} \psii_k \zeta^{-k} \right) e^{x^2 \zeta  d_{n+1}}
\]
and  $\tPsii_f(\zeta)=
\Psii_f(\zeta)  \, d_{n+1}^{-\frac12}$. 
As Stokes sectors at $\ze=\infty$ we take
initial sector 
\[
\Omi_1
=\{ \zeta\in\C^\ast \st  -\tfrac\pi2<\text{arg}\zeta < \tfrac{\pi}{2}+\tfrac{\pi}{n+1}\}
=\overline{\Omz_1}
\]
and then define $\Omi_{k+\scriptstyle\frac1{n+1}}= e^{\frac\pi{n+1}\i} \Omi_k$.
We can write
\[
\Omi_k
=\{ \zeta\in\tilde\C^\ast \st
\thi_{k-\scriptstyle\frac{1}{n+1}}-\tfrac\pi2
<
\text{arg}\ze 
<
\thi_k+\tfrac{\pi}{2}
\}
=\overline{\Omz_k}
\]
where 
$\thi_k
= -\thz_k$, 
$k\in\tfrac1{n+1}\Z$.    
We also have $\thi_k=\thz_{\frac{2n}{n+1}-k}$.

On $\Omi_k$ we have canonical holomorphic solutions $\Psii_k,\tPsii_k$.
We define Stokes factors $\Qi_k$ by
$\Psii_{k+\scriptstyle\frac1{n+1}} = \Psii_k \Qi_k$,
and Stokes matrices $\Si_k$ by
$\Psii_{k+1}=\Psii_k \Si_k$.
The singular directions $\thi_k$ bisect successive intersections
$\Omi_k  \cap  \Omi_{k+\scriptstyle\frac1{n+1}}$, and
index the Stokes factors $\Qi_k$.  
We have
\[
\Si_k=\Qi_{k}\Qi_{k+\scriptstyle\frac1{n+1}}\Qi_{k+\scriptstyle\frac2{n+1}}\dots
\Qi_{k+\scriptstyle\frac{n}{n+1}}.
\]
The monodromy of the solution $\Psii_k$ at $\ze=\infty$ is given by
\begin{equation}\label{Psii-mon}
\Psii_k(e^{-2\pi\i}\ze)=\Psii_k(\ze) \Si_k\Si_{k+1}
\end{equation}
(analytic continuation in the negative direction in the $\ze$-plane).
This follows from the identity
\begin{equation}\label{con-infinity}
\Psii_{k+2}(e^{2\pi\i}\ze)=\Psii_k(\ze),
\end{equation}
cf.\ (\ref{con-zero}). 
For topological reasons, $\Si_k\Si_{k+1}$ must be
conjugate to the inverse of $\Sz_k\Sz_{k+1}$; we shall give the precise
relation in the next section.

Finally, the $c$-reality and $\theta$-reality conditions of section \ref{2.1} 
lead to
\begin{equation}\label{QzandQi}
\Qi_k= d_{n+1}^{-1} \Qz_k d_{n+1}
\end{equation}
(Lemma 2.4 of \cite{GIL2}). 
This means that the Stokes data at $\ze=\infty$ is {\em equivalent} to the Stokes data at $\ze=0$.
The version of (\ref{tQzandQz}) at $\ze=\infty$ is
\begin{equation}\label{tQiandQi}
\tQi_k= d_{n+1}^{\frac12} \Qi_k d_{n+1}^{-\frac12}.
\end{equation}
Using (\ref{tQzandQz}) we obtain the simple relation
\begin{equation}\label{tQzandtQi}
\tQi_k=  \tQz_k,
\end{equation}
i.e.\ the tilde versions of the Stokes factors at zero and infinity are not only
equivalent but actually {\em coincide.}

\subsection{Stokes matrices for $\hat\om$ at $\ze=0$}\label{stokesforhatom-zero}  \

The leading term of equation  (\ref{phizeta})
 at $\ze=0$ is diagonalized by $h\Om$, as
$\Pi=\Omega \, d_{n+1}\,  \Omega^{-1}$.
Hence there is a formal solution 
\[
\textstyle
\Phiz_f(\zeta) =  
h\Om\left( I + \sum_{k\ge 1} \phiz_k \zeta^k \right) e^{\frac1\zeta  d_{n+1}}.
\]
We shall also make use of the formal solution
$
\textstyle
\tPhiz_f(\zeta) = \Phiz_f(\zeta)  \, d_{n+1}^{\frac12}.
$
Using the same Stokes sectors and singular directions as for $\Psiz_k,\tPsiz_k$, we obtain 
canonical holomorphic solutions $\Phiz_k,\tPhiz_k$.  

For the Stokes factors we shall not need new notation, as (in the situation of interest to us) it turns out that
these are exactly the same as the Stokes factors in section \ref{stokesforhatal-zero}:

\begin{proposition}\label{sameStokes} The connection form $\hat\om$ has the
{\em same} Stokes data as $\hat\al$ at $\ze=0$, i.e.\  the solutions $\Phiz_k$ satisfy
$\Phiz_{k+\scriptstyle\frac1{n+1}} = \Phiz_k \Qz_k$ and
 $\Phiz_{k+1}=\Phiz_k \Sz_k$. 
\end{proposition}

It follows in particular that $\Phiz_k$ has the same monodromy as
$\Psiz_k$:
\begin{equation}\label{Phiz-mon}
\Phiz_k(e^{2\pi\i}\ze)=\Phiz_k(\ze) \Sz_k\Sz_{k+1}.
\end{equation}
The proof of the proposition will be given at the end of section \ref{5}.

\subsection{Monodromy of $\hat\om$ at $\ze=\infty$}\label{stokesforhatom-infinity}  \

At the regular singularity $\ze=\infty$ it is easier to find a canonical solution:

\begin{proposition}\label{frobenius}  Equation  (\ref{phizeta}) has
a (unique) solution of the form
\[
\textstyle
\Phii(\zeta) = \left( I + \sum_{k\ge 1} \phii_k \zeta^{-k} \right) \ze^m.
\]
The series in parentheses converges on a non-empty open disk at $\ze=\infty$.
\end{proposition}

This solution extends by analytic continuation to
a solution which is holomorphic on $\tilde \C^\ast$ but multivalued on $\C^\ast$ because of the factor
$\ze^m=e^{m \log\ze}$. We take the standard branch of $\log\ze$ which is positive
on the positive real axis.

\begin{proof}
As $\ze=\infty$ is a simple pole for equation (\ref{phizeta}), o.d.e.\ theory
says that there is a solution near $\ze=\infty$ of the form
\[
\textstyle
\Phii(\zeta) =  \left( I + \sum_{k\ge 1} \phii_k \zeta^{-k} \right) \ze^m \ze^M
\]
where $M$ is nilpotent. We shall show later --- see
Remark \ref{Miszero} --- that $M=0$. 
This depends on our assumption that 
$m_{i-1}-m_i+1>0$. 
We note that, if $m_{i-1}-m_i+1= 0$, it is possible to have $M\ne 0$ (see \cite{GIL3}
for examples of this).
\end{proof}

Evidently we have
\begin{equation}\label{Phii-mon}
\Phii(e^{-2\pi\i}\ze)=\Phii(\ze) e^{-2\pi\i m}
\end{equation}
so the monodromy of this solution is $e^{-2\pi\i m}$
(analytic continuation in the negative direction in the $\ze$-plane).

\begin{remark} The discussion in sections \ref{stokesforhatal-zero}, \ref{stokesforhatal-infinity}, \ref{stokesforhatom-zero},
applies when $n+1\ge 4$. Here the Stokes sectors 
$\Omz_k,\Omi_k$  are maximal. 
When $n+1=2$
they are still Stokes sectors (in the sense of \cite{FIKN06}, section 1.4),
but they are not maximal.  The rays $\thz_k,\thi_k$ are singular directions when 
$k\in \frac12 + \Z$, but not when $k\in\Z$.  This has the effect of making
all Stokes factors equal to the identity matrix when $k\in\Z$. With this caveat, the above discussion applies also in the case $n+1=2$. We have 
\[
\Sz_1=\Qz_1\Qz_{\scriptstyle\frac3{2}}=\Qz_{\scriptstyle\frac3{2}},
\quad
\Sz_2=\Qz_2\Qz_{\scriptstyle\frac5{2}}=\Qz_{\scriptstyle\frac5{2}},
\]
and $\Mz=\Qz_1 \Qz_{\scriptstyle\frac3{2}} \Pi=
\Qz_{\scriptstyle\frac3{2}} \Pi$.
\qed
\end{remark}

\begin{remark} From (\ref{tQzandtQi}) we know that $\tQz_k=  \tQi_k$. Furthermore, as we shall see in section \ref{6.2}, it will turn out that these matrices are {\em real}. 
Thus the $\tQz_k$ are apparently superior to the 
$\Qz_k$. On the other hand, the $\Qz_k$ are easier to define, as for them the awkward factor $d_{n+1}^{\frac12}$ is not needed in the formal solution.  And the apparent advantages of the $\tQz_k$ do not carry over to the Lie-theoretic context of \cite{GH2}, where neither $\tQz_k$ nor $\Qz_k$
are optimal; they are a useful but special feature of $\SL_{n+1}\C$. For these reasons we use both $\tQz_k$ and $\Qz_k$ pragmatically, choosing whichever is convenient for the task at hand.
\qed
\end{remark}

\section{Definition of connection matrices}\label{5}

In the previous section we have defined monodromy data of $\hat\al,\hat\om$ at each pole (Stokes factors/monodromy matrices). In this section we define the remaining part of the monodromy data, namely the connection matrices. As in section \ref{4} we assume $n+1$ is even.  

\subsection{Connection matrices for $\hat\al$}\label{connforhatal}  \

The solutions $\Psiz_k,\Psii_k$ (regarded as holomorphic functions on $\tilde\C^\ast$) must be related by \ll connection matrices\rr  $E_k$, i.e.\ 
\[
\Psii_k=\Psiz_k E_k,
\]
and similarly 
$\tPsii_k=\tPsiz_k\tE_k$.
We have $\tilde E_k = d_{n+1}^{-\frac12} E_k d_{n+1}^{-\frac12}$.

From the definition we obtain immediately
\begin{equation}\label{Efromdef}
E_k = \left(\Qz_{k-\scriptstyle\frac1{n+1}}\right)^{-1}  E_{k-\scriptstyle\frac1{n+1}} 
\ \Qi_{k-\scriptstyle\frac1{n+1}}.
\end{equation}
Thus it suffices to consider $E_1$. 

Using the monodromy formulae (\ref{Psiz-mon}), (\ref{Psii-mon}), we obtain the \ll cyclic relation\rr
\[
E_1 = \Sz_1\Sz_{2}\, E_1\   \Si_1\Si_{2}.
\]
This shows that $E_1$ conjugates $\Si_1\Si_{2}$ to $(\Sz_1\Sz_{2})^{-1}$.
In fact (see Lemma 2.5 of \cite{GIL2}) the cyclic symmetry leads to the stronger relation
\begin{equation}\label{E-cyc}
E_1 =\left(\Qz_1\Qz_{1+\scriptstyle\frac1{n+1}}\Pi \right) \  E_1\    
\left( \Qi_1\Qi_{1+\scriptstyle\frac1{n+1}}\Pi \right)
\end{equation}
(stronger in the sense that $(n+1)$-fold iteration produces the cyclic relation).

\subsection{Connection matrices for $\hat\om$}\label{connforhatom}  \

With respect to the solutions $\Phii,\Phiz_k$ we define connection matrices
$D_k$ by
\[
\Phii=\Phiz_k D_k,
\]
and similarly 
$\Phii=\tPhiz_k\tD_k$.
We have $\tilde D_k = d_{n+1}^{-\frac12} D_k$.

It follows that 
\begin{equation}\label{Dfromdef}
D_k=\Qz_k D_{k+\scriptstyle\frac1{n+1}}.
\end{equation}
The monodromy formulae
(\ref{Phiz-mon}) and (\ref{Phii-mon}) give the cyclic relation
\[
D_k= \Sz_k \Sz_{k+1} \, D_k \, e^{-2\pi\i m},
\]
which shows that $D_k$ conjugates $e^{-2\pi\i m}$ to $(\Sz_k \Sz_{k+1})^{-1}$.
The cyclic symmetry gives the stronger relation
\begin{equation}\label{D-cyc}
D_k= \left( \Qz_k\Qz_{k+\scriptstyle\frac1{n+1}}\Pi \right)\, D_k \, e^{-2\pi\i m/(n+1)} d_{n+1}.
\end{equation}

\begin{remark}\label{GILconvention}
With the conventions of \cite{GIL3},  indicated here by $\text{GIL}$, where we used the $\la$-o.d.e.\ (\ref{ode-hatom}), 
the canonical solution at $\ze=\infty$ was chosen to be of the form
$\Phii_{ \text{GIL} }(\la) = ( I + O(\la^{-1}) \la^m$.
Our current convention is
$\Phii(\ze)=(I + O(\ze^{-1}))\ze^m$. Thus $\Phii_{ \text{GIL} }(\la)=\Phii(\ze)t^m$. 
On the other hand, the canonical solutions at $\ze=0$ satisfy
$\Phiz_{ \text{GIL} }(\la)=\Phiz(\ze)$.  Thus the connection matrix $D_k$ is
related to the corresponding connection matrix $D^{ \text{GIL} }_k$ in \cite{GIL3} by 
$D^{ \text{GIL} }_k = D_k t^m$.  
\qed
\end{remark}

\subsection{Relation between $E_k$ and $D_k$}\label{EandD}  \

The construction of $\al$ from $\om$, and hence 
the construction of $\hat\al$ from $\hat\om$, leads to
a relation between their connection matrices:

\begin{theorem}\label{EintermsofD} 
$E_k=\tfrac1{n+1} D_kt^m\, \De\, \bar t^{-m}\bar D_{\frac{2n+1}{n+1}-k}^{-1}\, d_{n+1}^{-1}\,  C,$
where  
\begin{equation}\label{matrixC}
C = \Om \bar\Om^{-1}
= \tfrac1{n+1}\Om^2
=
{\scriptsize
\left(
\begin{array}{c|ccc}
\!\!1 & & & \\
\hline
 & & & 1\!\\
 & & \iddots & \\
 & \!\!1 & &
\end{array}
\right)
}. 
\end{equation}
In particular we obtain
\begin{equation}\label{E1intermsofD1} 
E_1=\tfrac1{n+1} D_1 t^m\, \De\, (\overline{ D_1 t^m })^{-1} 
(\bar Q^{(0)}_{\frac{n}{n+1}})^{-1}
d_{n+1}^{-1}\,  C.
\end{equation}
\end{theorem}

\begin{proof}  From section \ref{3.1} (and section 2 of \cite{GIL3}) we know that
$(g L_\R G)^T(\la)$ is a solution of equation (\ref{ode-hatal}),
hence $(g L_\R G)^T(\ze t)$  is a solution of equation
(\ref{psizeta}). We remind the reader that $G,g$ were defined in 
(\ref{defofG}), (\ref{defofg}).
Thus, we must have
\[
\Psiz_k(\ze)=(g L_\R G)^T(\ze t) Y_k,\quad \Psii_k(\ze)=(g L_\R G)^T(\ze t) Z_k,
\]
for some constant  
matrices $Y_k,Z_k$ (independent of $\ze$).  

We shall show that

\no (1) $Y_k=(D_k t^m)^{-1}$

\no(2) $Z_k= \tfrac1{n+1} \De\, \bar Y_{\frac{2n+1}{n+1}-k}\, d_{n+1}^{-1}\, C$,

\no from which the above formula for $E_k=Y_k^{-1}Z_k$ follows immediately.

\no{\em Proof of (1).}\ 
By definition of the canonical solutions $\Phiz_k,\Psiz_k$, we have
\begin{align*}
\Phiz_k(\ze) &\sim  \Phiz_f(\ze)=h\Om (I+O(\ze) ) e^{\frac 1\ze d_{n+1}}
\\
\Psiz_k(\ze) &\sim  \Psiz_f(\ze)= e^{-w}\Om (I+O(\ze) ) e^{\frac 1\ze d_{n+1}}
\end{align*}
as $\ze\to 0$ in the sector $\Omz_k$.
Let us examine
\begin{align*}
\Psiz_k &= (g L_\R G)^T Y_k 
\\
&= (g L L_+^{-1} G)^T Y_k
\\
&= G (L_+^{-1})^T (gL)^T Y_k
\end{align*}
more closely.  
From the Iwasawa factorization we have 
\[
(L_+^{-1})^T=b^{-1}+O(\la),
\]
where $b$ is real and diagonal. 
From section \ref{3.2}  (and section 2 of \cite{GIL3}) we know that
$(g L)^T(\la)$ is a solution of equation (\ref{ode-hatom}),
hence $(g L)^T(\ze t)$  is a solution of equation
(\ref{phizeta}). From the form of the respective series expansions
we must have\footnote{
With the conventions of \cite{GIL3} we had $(g L)^T(\la)= \Phii_{ \text{GIL} }(\la)$.   
With our current convention --- see Remark \ref{GILconvention} ---
we have $(g L)^T(\la)= \Phii(\ze)t^m$.  
This affects the formula for $Y_k$: here we have $D_k t^m=Y_k^{-1}$, 
while \cite{GIL3} had $D^{ \text{GIL} }_k=Y_k^{-1}$.  
} 
$(g L)^T(\ze t)= \Phii(\ze) t^m$.   Hence
\[
(g L)^T(\ze t)= \Phii(\ze) t^m=
\Phiz_k(\ze) D_k t^m
\sim
h\Om (I+O(\ze) ) e^{\frac 1\ze d_{n+1}} D_k t^m.
\]

Substituting these expressions for 
$(L_+^{-1})^T$ and $(gL)^T$, we obtain
\begin{align*}
\Psiz_k(\ze) &\sim G\, b^{-1} h \Om (I+O(\ze) ) e^{\frac 1\ze d_{n+1}} D_k t^m Y_k
\\
&= e^{-w} \Om (I+O(\ze) ) e^{\frac 1\ze d_{n+1}} D_k t^m Y_k,
\end{align*}
as $Gb^{-1} h = \vert h\vert h^{-1} b^{-1} h = e^{-w}$, from the definition of $w$ in 
(\ref{wdef}). 
By the uniqueness of the formal solution $\Psiz_f$, it follows that 
$D_k t^m Y_k=I$ This gives formula (1) for $Y_k$.

\no{\em Proof of (2).}\ 
If $\Psi(\ze)$ is a solution of (\ref{psizeta}), then 
$
\De \overline{
\Psi(
\frac{1}{x^2 \bar\ze\vphantom{\tfrac{a}{a}} }
)
}
$ 
is also a solution of (\ref{psizeta}),
hence the latter must be equal to $\Psi(\ze)$ times a constant matrix.  For the formal
solution $\Psi(\ze)=\Psiz_f(\ze)$ at $\ze=0$, 
$
\De \overline{
\Psiz_f(
\frac{1}{x^2 \bar\ze\vphantom{\tfrac{a}{a}}}
)
}
$  
is then a formal
solution at $\ze=\infty$, and the constant matrix is $(n+1) C d_{n+1}$. This
is easily verified, using the identities $\Om\De = (n+1)d_{n+1}^{-1} \Om^{-1}$ and
$d_{n+1} C d_{n+1} = C$ (cf.\ Appendix A of \cite{GIL2}). 
We obtain 
\[
\De \overline{
\Psiz_f(
\tfrac{1}{x^2 \bar\ze\vphantom{\tfrac{a}{a}}}
)
}
=
(n+1) 
\Psii_f(\ze)
C d_{n+1},
\]
from which it follows that
\begin{equation}\label{Psi-reality} 
\De \overline{
\Psiz_k(
\tfrac{1}{x^2 \bar\ze\vphantom{\tfrac{a}{a}}}
)
}
=
(n+1) 
\Psii_{\frac{2n+1}{n+1}-k}(\ze)
C d_{n+1}.
\end{equation}
Let us substitute 
\[
\Psiz_k(\ze)=(g L_\R G)^T(\ze t) Y_k,\quad \Psii_k(\ze)=(g L_\R G)^T(\ze t) Z_k
\]
into (\ref{Psi-reality}), using
$
\overline{
(g L_\R G)(
1/\bar \la
)
}
=
g^{-1} \De L_\R(\la) \De \bar G
$
(which follows from (\ref{iwasawaLR}) and the definitions of $g,G$).
We obtain
\[
\De \bar G \De L_\R(\la)^T \De g^{-1} \bar Y_k = (n+1) G  L_\R(\la)^T g Z_{\frac{2n+1}{n+1}-k} C d_{n+1}
\]
As $G^{-1}\De \bar G \De = I$ and $g^{-1}\De g^{-1}=\De$, this reduces to
\[
\De \bar Y_k = (n+1) Z_{\frac{2n+1}{n+1}-k} C d_{n+1},
\] 
which gives formula (2) for $Z_k$. This completes the proof of the formula for
$E_k$ in Theorem \ref{EintermsofD}.  Formula (\ref{E1intermsofD1}) for $E_1$
follows from this and (\ref{Dfromdef}).
\end{proof}

\begin{remark} 
More conceptually,  the key formula (\ref{Psi-reality}) in the proof arises from the fact (section \ref{2.3}) that the Iwasawa factorization $L=L_\R L_+$ is taken with respect to the real form $\La SL^\De_{n+1}\R$; thus, by definition,
$L_\R$ satisfies
\begin{equation}\label{iwasawaLR}
\De \overline{ L_\R(1/\bar \la) } \De = L_\R(\la),
\end{equation}
i.e.\ $c( L_\R(1/\bar \la) )=L_\R(\la)$. This gives rise to the $c$-reality property of $\al$ 
(section \ref{2.1}), and the corresponding $c$-reality property of $\hat\al$ 
(section \ref{3.1}).  The latter predicts that
$\De \overline{
\Psiz_f(
\tfrac{1}{x^2 \bar\ze\vphantom{\tfrac{a}{a}}}
)
}
\De$
is equal to $\Psi_f(\ze)$ times a constant matrix.  However, this constant
matrix depends on the chosen normalizations of $\Psiz_f,\Psii_f$, so
a direct calculation of this matrix is unavoidable. 
\qed
\end{remark} 

Formula (1) of the proof allows us to show that the Stokes factors associated to the solutions $\Phiz_k$ of (\ref{phizeta}) agree with the Stokes factors $\Qz_k$ associated to the solutions $\Psiz_k$ of (\ref{psizeta}), as
stated in Proposition \ref{sameStokes}. We give the omitted proof here.

\no{\em Proof of Proposition \ref{sameStokes}.}\   Using the notation
of the 
proof of Theorem \ref{EintermsofD}, we have
\begin{align*}
\Qz_k &= (\Psiz_k)^{-1} \Psiz_{k+\scriptstyle\frac1{n+1}}
\\
&= (  (g L_\R G)^T Y_k)^{-1}    (g L_\R G)^T Y_{k+\scriptstyle\frac1{n+1}}
\\
&= Y_k^{-1} Y_{k+\scriptstyle\frac1{n+1}}
\\
&=D_kt^m (D_{k+\scriptstyle\frac1{n+1}}t^m)^{-1}
=D_k (D_{k+\scriptstyle\frac1{n+1}})^{-1}
\\
&= (\Phii D_k^{-1})^{-1} \Phii (D_{k+\scriptstyle\frac1{n+1}})^{-1}
\\
&=(\Phiz_k)^{-1} \Phiz_{k+\scriptstyle\frac1{n+1}},
\end{align*}
which is the Stokes factor for $\Phiz_k$,
as required.
\qed

\section{Computation of some Stokes matrices and connection matrices}\label{6}

The monodromy data corresponding to the solutions of (\ref{ost}) which arise from Proposition \ref{localsol}  
will be calculated explicitly in this section.  Our target is the monodromy data of $\hat\al$, which will be calculated in terms of 
the monodromy data of $\hat\om$, and this will be found explicitly in terms of
$c_0 z^{k_0},\dots, c_n z^{k_n}$. The calculation will be simplified by normalizing $\hat\om$ suitably.

As in sections \ref{4} and \ref{5} we assume that $n+1$ is even. The modifications needed when $n+1$ is odd will be given in an appendix (section \ref{9}).
Special features of the case $n+1=2$ will be mentioned when appropriate, in remarks at the end of each subsection.

\subsection{The normalized o.d.e.}\label{6.1} \

Recall the $\ze$-equation (\ref{phizeta}) associated to the connection form $\hat\om$:
\begin{equation*}   
\Phi_\ze= 
\left[
-\tfrac1{\ze^2} h \Pi h^{-1} + \tfrac1\ze m
\right]
\Phi.
\end{equation*}
The matrices $m,h,\Pi$ were specified in Definitions \ref{Nmchat} and \ref{thG}.

Introducing $\cX=h^{-1}\Phi$, we have
\begin{equation}\label{node}
\cX_\ze= 
\left[
-\tfrac1{\ze^2}  \Pi  + \tfrac1\ze m
\right]
\cX.
\end{equation}
We regard (\ref{node}) as the fundamental (normalized) equation. We shall use it for all calculations in this section. 

First we give the relation between the monodromy data for $\cX$ and that for $\Phi$. 
At $\ze=0$, we have the formal solution
\[
\textstyle
\cXz_f(\ze) =  
\Om\left( I + \sum_{k\ge 1} \xzer_k \ze^{k} \right) e^{\frac1\ze  d_{n+1}}.
\] 
The sector $\Omz_k$ then determines a canonical holomorphic solution
$\cXz_k$. We define Stokes factors $\cQz_k$ by 
\[
\cXz_{k+\scriptstyle\frac1{n+1}} = \cXz_k  \cQz_k.
\]

At $\ze=\infty$ we have the canonical solution
\[
\textstyle
\cXi(\ze) = \left( I + \sum_{k\ge 1} \xinf_k \zeta^{-k} \right) \ze^m.
\]
We define connection matrices $\cD_k$ by
\[
\cXi=\cXz_k \cD_k,
\]
from which it follows that
\begin{equation}\label{cDfromdef}
\cD_k=\cQz_k \cD_{k+\scriptstyle\frac1{n+1}}.
\end{equation}

\begin{proposition}\label{cQandQ}
The monodromy data for (\ref{node}) is related to the
monodromy data for (\ref{phizeta})  by: 
(1) $\cQz_k=\Qz_k$,
(2) $\cD_k=D_k h$.
\end{proposition}

\begin{proof}
(1) As $\Phiz_f=h\cXz_f$, we have $\Phiz_k=h\cXz_k$. It follows that
$\cQz_k=\Qz_k$.
(2) As $h^{-1}\Phii$ is a solution of (\ref{node}), it must be $\cXi$ times
a constant matrix.  Comparison of the definitions of $\Phii,\cXi$ gives
$h^{-1}\Phii=\cXi h^{-1}$, i.e.\
$\Phii=h\cXi h^{-1}$.  It follows that $\cD_k=D_k h$.
\end{proof}

\subsection{Computation of the Stokes data}\label{6.2} \ 

The singular directions $\thz_k\in\R$ (regarded as angles of rays in the 
universal covering of $\C^\ast$) index the Stokes factors $\Qz_k = \cQz_k$.
We shall now compute these matrices.

\begin{proposition}\label{Sfactors1}
All diagonal entries of $\Qz_k$ are $1$. The $(i,j)$-entry 
$(0\le i\ne j\le n)$ can be nonzero only when $(i,j)$ satisfies the condition
\begin{equation}\label{anglecriterion}
\arg(\om^j-\om^i)=\thz_k \quad \text{mod}\ 2\pi.
\end{equation}
\end{proposition}

\begin{proof}
This is the classical criterion, based on the eigenvalues of the leading term $-\Pi=\Om^{-1}(-d_{n+1})\Om$ of (\ref{node}); cf.\ Lemma 3.1 of \cite{GH1}.
\end{proof}

\begin{definition}\label{calRzero}
For $k\in\tfrac1{n+1}\Z$, let
$\cRz_k=\{ (i,j)\st (i,j)\ \text{satisfies (\ref{anglecriterion})} \}$. 
\end{definition}

By the cyclic symmetry formula (\ref{Qz-cyc}),  the matrices
$\Qz_1$, $\Qz_{1+\scriptstyle\frac1{n+1}}$ determine all Stokes factors.  The corresponding $\cRz_1$, $\cRz_{1+\scriptstyle\frac1{n+1}}$ are as follows (Proposition 3.4 of \cite{GH1}):

\begin{proposition}\label{Sfactors2}  \ 
(1) If $n+1=4c$:

\no$\cRz_1=
\{
(2c-1,0),(2c-2,1),\dots,(c,c-1)\}
\newline
\quad\quad\quad\cup \{(2c,4c-1),(2c+1,4c-2),\dots,(3c-1,3c)
\}
$

\no$\cRz_{1+\scriptstyle\frac1{n+1}}=
\{
(2c-1,4c-1),(2c,4c-2),\dots,(3c-2,3c)\}
\newline
\quad\quad\quad\quad\cup\{(2c-2,0),(2c-3,1),\dots,(c,c-2)
\}
$

\no(2) If $n+1=4c+2$:

\no$\cRz_1=
\{(2c+1,4c+1),(2c+2,4c),\dots,(3c,3c+2)\}
\newline
\quad\quad\quad\cup
\{(2c,0),(2c-1,1),\dots,(c+1,c-1)
\}
$

\no$\cRz_{1+\scriptstyle\frac1{n+1}}=
\{
(2c-1,0),(2c-2,1),\dots,(c,c-1)\}
\newline
\quad\quad\quad\quad\cup
\{(2c,4c+1),(2c+1,4c),\dots,(3c,3c+1)
\}
$
\qed
\end{proposition} 

\begin{remark} The set $\cRz_k$ may be described more conceptually as a subset of the roots of the Lie algebra $\sl_{n+1}\C$.
As explained in \cite{GH1},\cite{GH2}, the significance of 
$\cRz_1\cup\cRz_{1+\scriptstyle\frac1{n+1}}$
is that it is a set of representatives for the orbits of the action of the Coxeter element $(n\ n-1\ \dots\ 1\ 0)$ on the roots.
We shall not exploit this Lie-theoretic point of view, but we note the consequence that 
$\cRz_{k+\scriptstyle\frac2{n+1}}$ is obtained from $\cRz_{k}$ by applying 
$(n\ n-1\ \dots\ 1\ 0)$. Thus the proposition determines all $\cRz_{k}$.
\qed
\end{remark}

Propositions \ref{Sfactors1} and \ref{Sfactors2} specify the \ll shape\rr of $\Qz_k$, i.e.\ the values of the diagonal entries and which off-diagonal entries can be nonzero.  
The tilde version $\tQz_k = d_{n+1}^{-\frac12} \Qz_k d_{n+1}^{\frac12}$
has the same shape, as $d_{n+1}^{-\frac12}$ is diagonal. 
It will be convenient to use $\tQz_k$ instead $\Qz_k$ from now on in this section, because 
all the entries of the $\tQz_k$ turn out to be real (Proposition \ref{qij} below). 

From the discussion above, we may write
\[
\tQz_k= I + \sum_{(i,j)\in\cRz_k} \tq_{i,j}  E_{i,j}   
\]
for some (a priori, complex) scalars $\tq_{i,j}$, where $E_{i,j}$
is the matrix which has $1$ in the $(i,j)$-entry and $0$ elsewhere.

The symmetries of the $\tQz_k$ imply the following symmetries
of the $\tq_{i,j}$:

\begin{proposition}\label{Qij}    Interpreting $i,j\in\{0,1,\dots,n\}$ mod $n+1$, we have:

(1) $\tq_{i-1,j-1}=
\begin{cases}
\ \  \tq_{i,j} \ \text{if $i,j\ge 1$}
\\
-\tq_{i,j} \ \text{if $i=0$ or $j=0$}
\end{cases}
$

(2) $\tq_{j,i} = -\tq_{i,j}$

(3) $\overline{\tq_{i,j}}=
\begin{cases}
-\tq_{n+1-i,n+1-j} \ \text{if $i,j\ge 1$}
\\
\ \ \tq_{n+1-i,n+1-j} \ \text{if $i=0$ or $j=0$}
\end{cases}
$
\end{proposition}
\begin{proof}
These correspond to the symmetries of the $\tQz_k$ given by the automorphisms $\tau,\si,c$.  
\end{proof}

In view of (1) and (3) let us modify the signs of the $E_{i,j}$ as follows:
\begin{equation*}
e_{i,j}=
\begin{cases}
\ \  E_{i,j} \ \text{if $0\le i<j\le n$}
\\
-E_{i,j} \ \text{if $0\le i>j\le n$}
\end{cases}
s_{i,j}=
\begin{cases}
\ \  \tq_{i,j} \ \text{if $0\le i<j\le n$}
\\
-\tq_{i,j} \ \text{if $0\le i>j\le n$}
\end{cases}
\end{equation*}
Then we may write 
\begin{equation}\label{Qintermsofsij}
\tQz_k= I + \sum_{(i,j)\in\cRz_k} s_{i,j}  \, e_{i,j},
\end{equation}
where the scalars $s_{i,j}$ satisfy the following simpler relations:

\begin{proposition}\label{qij} 
Interpreting $i,j\in\{0,1,\dots,n\}$ mod $n+1$, we have:

(1) $s_{i-1,j-1}=s_{i,j}$

(2) $s_{j,i} = s_{i,j}$

(3) $\bar s_{i,j}=
s_{n+1-i,n+1-j}$

\no These relations imply that all $s_{i,j}\in\R$.
\end{proposition}

\begin{proof}
(1),(2),(3) are direct translations from Proposition \ref{Qij}.  Using these, for $i<j$, we have
$\bar s_{i,j} = s_{n+1-i,n+1-j} = s_{j-i,0}=s_{0,j-i}= s_{i,j}$, so $s_{i,j}$ is real. 
\end{proof}

\begin{definition} For $i,j \in \{0,1,\dots,n\}$,  let $[i,j]$ denote the equivalence class of $(i,j)$ under  $(i,j)\sim (i-1,j-i)$ (interpreted mod $n+1$). 
\end{definition}

By (1) of Proposition \ref{qij}, $s_{i,j}$ depends only on the equivalence class $[i,j]$, so we
can denote it by $s_{[i,j]}$.

\begin{definition}\label{si}
$s_i=s_{[0,i]}$.
\end{definition}

By (1) and (2) of Proposition \ref{qij} we have
$s_i=s_{n+1-i}$.
Thus the essential Stokes parameters\footnote{In terms of the notation $s_1^\R,s_2^\R$
of \cite{GIL2},\cite{GIL3},\cite{GH1} (the case $n=3$), we have
$s_1=s_1^\R, s_2=-s_2^\R$.} are the real numbers
$s_1,\dots,s_{\frac12(n+1) }$.  

This completes our description of the {\em shape}  of the Stokes factor $\tQz_k$.
We should now recall that $\tQz_k$ was associated to a certain solution
of (\ref{ost}).  In section \ref{2.3}, this solution was constructed from the data $c_0 z^{k_0},\dots c_n z^{k_n}$ where $c_i>0$ and $k_i>-1$. To {\em compute} $\tQz_k$ 
it remains to express $s_1,\dots,s_{\frac12(n+1) }$ in terms of this data.
In order to do this, we need one more result on the structure of $\tQz_k$.

Recall (Definition \ref{matrixM}) that we have introduced
$\Mz=\Qz_1\Qz_{1+\scriptstyle\frac1{n+1}}\Pi$.   From formula
(\ref{tQzandQz}) we have
$\Qz_k= d_{n+1}^{\frac12} \tQz_k d_{n+1}^{-\frac12}$, so
\[
\Mz= d_{n+1}^{\frac12} \tQz_1\tQz_{1+\scriptstyle\frac1{n+1}} d_{n+1}^{-\frac12}\Pi
= 
\om^{\frac12}
d_{n+1}^{\frac12} \tQz_1\tQz_{1+\scriptstyle\frac1{n+1}} \hat\Pi d_{n+1}^{-\frac12}
\]
where
\begin{equation}\label{pihat}
\hat\Pi = \om^{-\frac12} d_{n+1}^{-\frac12} \Pi d_{n+1}^{\frac12}
=
\left(
\begin{smallmatrix}
  &  1  & & &\\
  &   & \ \ddots & &\\
  & & &  & 1\\
-1    & & &
\end{smallmatrix}
\right).
\end{equation}
In view of this we introduce
\begin{equation}\label{matrixMtilde}
\tMz=\tQz_1\tQz_{1+\scriptstyle\frac1{n+1}}\hat\Pi
=
\om^{-\frac12} d_{n+1}^{-\frac12} \Mz d_{n+1}^{\frac12}.
\end{equation}
This satisfies $(\tMz)^{n+1} = -\tSz_1\tSz_2$. 
Then the result we need is:

\begin{proposition}\label{charpoly}
The  characteristic polynomial of 
$\tMz=\tQz_1\tQz_{1+\scriptstyle\frac1{n+1}}\hat\Pi$ is
$\sum_{i=0}^{n+1} s_i \mu^{n+1-i}$.  Here we put $s_0=s_{n+1}=1$.
\end{proposition}

\no This shows that $\tMz$ (in fact, just the conjugacy class of $\tMz$)
determines the $s_i$ and hence all Stokes factors.

\begin{proof} A proof by direct computation is given in Proposition 3.4 \cite{Ho17}, where $s_i$ is denoted by $p_i$. 
\end{proof}

\begin{remark}
A Lie-theoretic proof could be given by observing, as in \cite{GH1}, that 
$\tMz$ is a Steinberg cross-section of the set of regular conjugacy classes.
\qed
\end{remark}

Using this, we obtain explicit expressions for the $s_i$:

\begin{theorem}\label{siofost}
The Stokes factor $\tQz_k$ associated to the data $c_0 z^{k_0},\dots c_n z^{k_n}$ (or to the corresponding local solution of (\ref{ost})) is 
\[
\tQz_k= I + \sum_{(i,j)\in\cRz_k} s_{\vert i-j\vert}  \, e_{i,j}
\]
where $s_i$ is the $i$-th symmetric function of
$\om^{m_0+\frac n2}, \om^{m_1-1+\frac n2}, \dots, \om^{m_n-n+\frac n2}$.
\end{theorem}

\begin{proof} The cyclic symmetry (\ref{D-cyc})
\[
\Qz_1\Qz_{1+\scriptstyle\frac1{n+1}}\Pi = D_k \,  d_{n+1}^{-1} \omega^m D_k^{-1}
\]
shows that the eigenvalues of $\Mz=\Qz_1\Qz_{1+\scriptstyle\frac1{n+1}}\Pi$ are those of $d_{n+1}^{-1} \omega^m $. 
Hence
the eigenvalues of $\tMz$ are those of 
$\om^{-\frac12} d_{n+1}^{-1} \omega^m=
\om^{\frac n2} d_{n+1}^{-1} \omega^m
$.

Let $\sigma_i$ be the $i$-th symmetric function of the $n+1$ complex numbers
$\om^{m_0+\frac n2}, \om^{m_1-1+\frac n2}, \dots, \om^{m_n-n+\frac n2}$.
These are the eigenvalues of $\om^{\frac n2} d_{n+1}^{-1} \om^m = - \om^{-\frac12} d_{n+1}^{-1}\om^m$ (we are assuming that $n+1$ is even).   Thus, the characteristic
polynomial of $-\tMz$ is 
$\sum_{i=0}^{n+1} (-1)^i\sigma_i \mu^{n+1-i}$, and 
the characteristic
polynomial of $\tMz$ is 
$\sum_{i=0}^{n+1} \sigma_i \mu^{n+1-i}$.  By Proposition \ref{charpoly}, we
must have $s_i=\sigma_i$.
\end{proof}

The theorem shows that the $s_i$ depend only on the $k_i$ (and not on the $c_i$). 

\begin{example}\label{2x2}
In the case $n+1=2$ we have $\cRz_k=\emptyset$ and $\tQz_k=I$ for $k\in\Z$.
There is only one Stokes parameter $s=s_1$, and
\[
\tQz_{\scriptstyle\frac3{2}}=
\bp
1 & s
\\
0 & 1
\ep,
\quad
\tQz_{\scriptstyle\frac5{2}}=
\bp
1 & 0
\\
-s & 1
\ep
\]
(and $\tQz_{k+2}=\tQz_k$).
We have 
\[
\tMz= \tQz_1 \tQz_{\scriptstyle\frac3{2}} \hat\Pi =
\bp
-s & 1
\\
-1 & 0
\ep,
\]
whose characteristic polynomial is $\mu^2+s\mu+1$.
Theorem \ref{siofost} gives $s=-2 \sin \pi m_0$.
We have $m_0=\frac{k_1-k_0}{2k_0+2k_1+4}$ (from Definition \ref{Nmchat}). 
\qed
\end{example}

\subsection{Computation of the connection matrices $D_k$}\label{6.3} \ 

To compute $D_k$, we need some particular solutions $\cX$ of equation (\ref{node})
whose behaviour at both $\ze=0$ and $\ze=\infty$ can be computed.
First, we write the fundamental solution matrix as
\[
\cX=
\bp
| & & | \\
X^{(0)} &\cdots & X^{(n)} \\
| & & | 
\ep,
\]
and then study the system
\begin{equation}\label{vectorode}
\bze X = [-\tfrac1\ze\Pi + m] X,
\quad
X=
\bp
x_0\\
\vdots\\
\vdots\\
x_n\\
\ep.
\end{equation}
In fact, we shall focus on the scalar o.d.e. for $x_0$:
\begin{equation}\label{scalarode}
\ze(\bze-m_n) \cdots \ze(\bze-m_0) x_0 = (-1)^{n+1} x_0.
\end{equation}
The function $X$ may be recovered from $x_0$ and the formula
$x_{k+1}= -\ze(\bze-m_k) x_k$.

Note that (\ref{scalarode}) can be written
\[
\left[
(\bze-m_n^\pr)(\bze-m_{n-1}^\pr)\cdots(\bze-m_0^\pr) -  (-1)^{n+1} \ze^{-(n+1)}
\right]
x_0=0
\]
where 
\begin{equation}\label{mdash}
m_i^\pr=m_i-i. 
\end{equation}
Thus the indicial roots are $m_0^\pr,\dots,m_n^\pr$.

It will be convenient to allow $i\in\Z$ in (\ref{mdash}).  
(Note that $m_i=m_{i+n+1}$ but $m^\pr_i=m^\pr_{i+n+1} + n+1$.)
Then the condition $k_i>-1$ in Definition \ref{Nmchat} is equivalent to
\begin{equation}\label{mdashineq}
m^\pr_i<m^\pr_{i-1}, \quad i\in\Z.
\end{equation}

\begin{lemma}\label{integral} Let $c$ and $b_0,\dots,b_n$ be real numbers
such that $\frac {b_i+i}{n+1}>c$ for all $i$. Let
$g(\ze)=g^{b_0,\dots,b_n}(\ze)=$
\[
\int_{c-\i\infty}^{c+\i\infty}
\textstyle
\Ga(\frac {b_0+0}{n+1} - t)
\Ga(\frac {b_1+1}{n+1} - t)
\cdots
\Ga(\frac {b_n+n}{n+1} - t)
(n+1)^{-(n+1)t} \ze^{-(n+1)t}
dt.
\]
This defines a holomorphic function of $\ze$ on the sector 
$-\tfrac{\pi}2<\arg \ze <\tfrac\pi2$ 
(taking the standard branch of $\log\ze$ in $\ze^{-(n+1)t}$).
Furthermore:

(1) 
$g$ satisfies the equation
$\ze(\bze+b_n) \cdots \ze(\bze+b_0) g(\ze) = g(\ze)$, i.e.
$
\left[
(\bze+(b_0+0))(\bze+(b_1+1))\cdots(\bze+(b_n+n)) - \ze^{-(n+1)}
\right]
g(\ze)=0.
$

(2) We have
$\ze(\bze+b_0) g^{b_0,\dots,b_n}(\ze) = g^{b_1,\dots,b_n,b_0}(\ze)$; let
us call this $g^{[1]}(\ze)$.
For $i=1,2,\dots$ we define $g^{[i+1]}=(g^{[i]})^{[1]}
= g^{b_{i+1},\dots,b_n,b_0,\dots,b_i}$.

(3) 
$
g(\ze)\sim 
\ii (2\pi)^{\frac12(n+2)}
(n+1)^{-\frac12 - \frac1{n+1}\sum_0^n b_i}
\,\ze^{- \frac1{n+1}\sum_0^n b_i}
\, e^{-\frac1\ze}$
as $\ze\to 0$.

(4)  Assume that $0<\vert (b_i+i)-(b_j+j)\vert < n+1$ for all $i,j$. 
Then
$g(\ze)=\sum_{i=0}^n C_i \ze^{-(b_i + i)} (1 + O(\ze^{-(n+1)}))$ for some $C_0,\dots,C_n$
as $\ze\to \infty$.  Only $C_0$ will be needed later. It is given by
\[
C_0=
2\pi\ii (n+1)^{-b_0} 
\textstyle
\Ga(\frac {b_1-b_0+1}{n+1})
\Ga(\frac {b_2-b_0+2}{n+1})
\cdots
\Ga(\frac {b_n-b_0+n}{n+1}).
\]
\end{lemma}

\begin{proof}  Stirling's Formula shows that the integral converges 
whenever $-\tfrac{\pi}2<\arg \ze <\tfrac\pi2$.  Note that the poles of
$\Ga(\frac {b_i+i}{n+1} - t)$ are $\frac {b_i+i}{n+1}+k$, $k=0,1,2,\dots$, all
of which lie to the right of the contour of integration.
We sketch the proofs of (1)-(4) below.

(1) Application of $\bze+b_i+i$ to $\ze^{-(n+1)t}$ produces
the factor $b_i+i-(n+1)t=(n+1)
\left[
\frac {b_i+i}{n+1} - t
\right]$.  The factor $\frac {b_i+i}{n+1} - t$ 
combines with 
$\Ga(\frac {b_i+i}{n+1} - t)$ to produce $\Ga(\frac {b_i+i}{n+1} - t+1)$.  
Application of
$(\bze+(b_0+0))(\bze+(b_1+1))\cdots(\bze+(b_n+n))$ to $g(\ze)$
thus produces, after the change of variable $t^\pr=t-1$, 
$\ze^{-(n+1)} g(\ze)$, as required.

(2)  Similar to the proof of (1).  

(3) Substituting the definition $\Ga(k)=\int_0^\infty e^{-\tau} \tau^{k-1} d\tau$ of
the gamma function  into
the integral, and computing\footnote{
In Proposition 3.11 of \cite{GIL3}, the initial  factor $2$ was omitted. The authors
thank Yudai Hateruma for this correction. 
}
as in Proposition 3.11 of \cite{GIL3}, we obtain
$g(\ze)=$
\[
\tfrac{2\pi \sqrt{-1}}
{(\ze(n+1))^{\sum_{0}^{n} B_i} }
\int_{0}^{\infty} \!\!\!\!\!\!\cdots\!\! \int_{0}^{\infty}
\!\overset{\scriptscriptstyle n}{\underset{\scriptstyle i=1}\Pi} x_i^{ B_i-1-B_0 }
e^{
-\frac1\ze \frac1{n+1} 
\left(
x_1+\cdots+x_n+\frac1{x_1\cdots x_n}
\right)
}
dx_1\!\dots\! dx_n,
\]
where $B_i=\frac{b_i + i}{n+1}$.
We shall apply Laplace's Method to obtain the asymptotics as $\ze\to 0$ from this formula.  

In general, for a (real) function $\phi$ on $(0,\infty)$ with a nondegenerate minimum at $c$, 
Laplace's Method says that
$\int_0^\infty f(x) e^{-\frac1\ze\phi(x)} dx \sim f(c) e^{-\frac1\ze\phi(c)} (2\pi)^{\frac12} 
\ze^{\frac12}\phi^{\pr\pr}(c)^{-\frac12}$ as $\ze\to 0$ with
$-\tfrac\pi2 <\arg \ze< \tfrac\pi2$.   

The function
\[
\phi(x_1,\dots,x_n)= \tfrac1{n+1}
(x_1+\cdots+x_n+\tfrac1{x_1\cdots x_n})
\]
has a critical point at $(x_1,\dots,x_n)=(1,\dots,1)$, and the
eigenvalues of the Hessian at this point are $1$ (with multiplicity $1$),
$\frac1{n+1}$ (with multiplicity $n-1$).  We obtain
\[
g(\ze)\sim
\tfrac{2\pi \sqrt{-1}}
{(\ze(n+1))^{\sum_{0}^{n} B_i} }
\ 
e^{-\frac1\ze}
\ 
(2\pi)^{\frac12 n}
\ 
\ze^{  \frac12 n} 
\ 
(n+1)^{\frac12 (n-1)}.
\]
Now, $\sum_0^n B_i=\tfrac1{n+1}( \tfrac12 n(n+1) + \sum_{i=0}^n b_i)
=\tfrac12 n + \tfrac1{n+1} \sum_{i=0}^n b_i$.
This gives the stated formula. 

(4)
As stated earlier, all poles of the integrand lie in the 
half-plane $\text{Re\,} t>c$.  By the assumption of (4), these poles
are distinct.

Consider a semicircle of radius $R$ in the 
half-plane $\text{Re\,} t>c$ which, together with the segment $\{ c+iT \st -R \le T \le R \}$, 
forms a closed contour. It can be shown that the integral over the semicircle tends to zero
as $R\to\infty$ --- this argument is summarized in Lemma A.6 of \cite{BoHo06}, for example.  

Cauchy's Integral Theorem then shows that 
$g(\ze)/(-2\pi \ii)$ is equal to the sum of the residues of the integrand.  
The residue of $\Ga(u-t)$ is $-(-1)^k/k!$ 
at its simple pole $u+k$. The stated
formula follows from these facts (as in the proof of Corollary 3.10 of \cite{GIL3}). 
\end{proof}

\begin{remark}\label{Miszero} 
The residue expansion in the proof of (4) shows that, near $\ze=\infty$,  $g(\ze)$ is the solution of the scalar equation predicted
by the Frobenius Method.  Note that there are no \ll polynomial in $\log\ze$\rr terms, even though the differences of indicial roots may be integers.  This justifies our assertion in the proof of Proposition \ref{frobenius}
that $M=0$ there.
\qed
\end{remark}

We shall use Lemma \ref{integral} with $b_i=-m_i$, so we begin by verifying the assumptions on
$b_0,\dots,b_n$.  For convenience let us write 
$b^\pr_i=b_i+i$ and extend this to $i\in\Z$ by defining
$b_{i+n+1}=b_i$, $b^\pr_{i+n+1}=b^\pr_i + n+1$
(so that $b_i=-m_i$ and $b^\pr_i=-m^\pr_i$ for all $i\in\Z$).  
Then (\ref{mdashineq}) gives
\begin{equation}\label{bdashineq}
b^\pr_n - (n+1) = b^\pr_{-1} <
b^\pr_0 < \cdots < b^\pr_n < b^\pr_{n+1} = b^\pr_0+n+1.
\end{equation}
From $m_0+m_n=0$ we have $b^\pr_0+b^\pr_n=n$, so
the left-most inequality gives $-b^\pr_0 - 1 < b^\pr_0$, hence $b^\pr_0>-\tfrac12$. 
It follows that $b^\pr_i>-\tfrac12$ for $i\ge 0$, hence 
$\frac {b_i+i}{n+1}>-\tfrac1{2n+2}$.  Thus the initial
assumption in Lemma \ref{integral} holds if 
we take $c\le -\tfrac1{2n+2}$.  The validity of the assumption in (4) 
of Lemma \ref{integral} follows directly from (\ref{bdashineq}), as
$b^\pr_{n+1}-b^\pr_0=n+1$.  

To construct a solution $\cX$ of (\ref{node}), we begin with
the scalar equation (\ref{scalarode}).  
As we are assuming that $n+1$ is even,  
the scalar equation of (1) coincides with  (\ref{scalarode}) if we take $x_0=g$.  As 
$x_{k+1}= -\ze(\bze-m_k) x_k$, (2) implies that $x_{k+1}= (-1)^{k+1} g^{[k+1]}$. 
Thus we have a solution
\[
X=
\bp
x_0\\
\vdots\\
\vdots\\
x_n\\
\ep=
\bp
\ \ \ \ \ \, g\\
(-1)^1g^{[1]}\\
\vdots\\
(-1)^ng^{[n]}
\ep
\]
of (\ref{vectorode}).

Next we need a basis of solutions $x_0^{(0)},\dots,x_0^{(n)}$
of the scalar equation, in order to construct $\cX$ from $X$.  We shall take $n+1$ functions of the form $g(\om^i \ze)$.  These are solutions of the scalar equation as only $\ze^{n+1}$
occurs in the coefficients  of the equation, and $(\om^i\ze)^{n+1}=\ze^{n+1}$. It follows from (4) that any consecutive $n+1$ such functions are linearly independent. We shall take
\[
\cX=
\bp
 & \dots & g(\om \ze) & g(\ze) & g(\om^{-1} \ze) & \dots  & 
\\
\
\\
& & \vdots & \vdots & \vdots &  & 
\\
\
\ep
\]
where row $i+1$ is given by applying $-\ze(\bze - m_i)$ 
to row $i$ ($0\le i\le n-1$), and where $g(\ze)$ 
is in column $\tfrac12(n+1)$ (out of columns $0,1,\dots,n$). 

Thus we obtain the particular solution  
\[
\cX=
\bp
\vphantom{ \boxed{{A^B_C}}}
 & \dots & g(\om\ze) & g(\ze) & g(\om^{-1} \ze) & \dots  & 
 \\
\vphantom{ \boxed{{A^B_C}}}
  & \dots & -\om^{-1} g^{[1]}(\om\ze) & -g^{[1]}(\ze) & -\om g^{[1]}(\om^{-1} \ze) & \dots  & 
\\
\vphantom{ \boxed{{A^B_C}}}
  & \dots & \om^{-2}g^{[2]}(\om\ze) & g^{[2]}(\ze) & \om^{2} g^{[2]}(\om^{-1}\ze) & \dots  & \\
& & \vdots & \vdots & \vdots &  & \\
\ep
\]
of (\ref{node}).
We shall use this to compute $\cD_k$, by relating $\cX$ to the canonical solutions $\cXi, \cXz_k$. 

\no{\em The relation between $\cX$ and $\cXi$.}\ 
Let us write 
\[
\cXi(\ze)=\cX(\ze) B.
\]
Then 
\[
B^{-1}= \lim_{\ze\to \infty} \cXi(\ze)^{-1} \cX(\ze) =  \lim_{\ze\to\infty} \ze^{-m} \cX(\ze).
\]
Applying (4) of Lemma \ref{integral}, we obtain
\[
\textstyle
g^{-m_0,\dots,-m_n}(\ze)=\sum_{i=0}^n C_i \, \ze^{m_i^\pr} (1 + O(\ze^{-(n+1)}))
\]
where
\[
C_0=
2\pi\ii (n+1)^{m_0} 
\textstyle
\Ga(\frac {-m_1^\pr +m_0^\pr}{n+1})
\Ga(\frac {-m_2^\pr+m_0^\pr}{n+1})
\cdots
\Ga(\frac {-m_n^\pr+m_0^\pr}{n+1}).
\]

By (\ref{mdashineq}), $m^\pr_i<m^\pr_0=m_0$ for $1\le i\le n$.
Hence
\[
\textstyle
\lim_{\ze\to\infty} \ze^{-m_0} g(\ze) = 
\lim_{\ze\to\infty} 
\sum_{i=0}^n C_i \, \ze^{m_i^\pr-m_0} (1 + O(\ze^{-(n+1)}))
=C_0.
\]
It follows that the first row of $B^{-1}$ is
\[
\dots \quad C_0 \om^{2m_0} \quad C_0 \om^{m_0} \quad C_0 \quad C_0 \om^{-m_0} \quad \dots
\]
For the second row, we use the fact that $-\ze(\bze+m_0)g(\ze)=-g^{[1]}(\ze)$.
A similar argument gives the second row of $B^{-1}$ as
\[
\dots \quad -C_0^{[1]} \om^{-2} \om^{2m_1} \quad -C_0^{[1]} \om^{-1} \om^{m_1} \quad -C_0^{[1]} \quad -C_0^{[1]} \om^1 \om^{-m_1} \quad \dots
\]
where
\[
C_0^{[1]}=
2\pi\ii (n+1)^{m_1} 
\textstyle
\Ga(\frac {-m_2^\pr+m_1^\pr}{n+1})
\cdots
\Ga(\frac {-m_n^\pr+m_1^\pr}{n+1})
\Ga(\frac {-m_{n+1}^\pr+m_1^\pr}{n+1}).
\]

Continuing in this way, with
$C_0^{[i+1]}=(C_0^{[i]})^{[1]}$,
we obtain  
\[
B^{-1}= \Ga_m \ V_m
\]
where $\Ga_m$, $V_m$ are (respectively)
\[
\bp
C_0 & & & \\
 & -C_0^{[1]} & & \\
 & & \ddots & \\
 & & & (-1)^n C_0^{[n]} 
\ep,
\bp
\dots & \om^{2m_0^\pr} & \om^{m_0^\pr} & 1 & \om^{-m_0^\pr} & \dots  \\
\dots & \om^{2m_1^\pr} & \om^{m_1^\pr} & 1 & \om^{-m_1^\pr} & \dots  \\
 & \vdots &  \vdots &  \vdots &  \vdots &  \\
\dots & \om^{2m_n^\pr} & \om^{m_n^\pr} & 1 & \om^{-m_n^\pr} & \dots 
  \ep.
\]
In the matrix $V_m$ (on the right),
column $\tfrac12(n+1)$ (out of columns $0,1,\dots,n$) is
the column whose entries are all equal to $1$.

\no{\em The relation between $\cX$ and $\cX_1$.}\ 
Let us write 
\[
\cX_1(\ze)=\cX(\ze) B_1.
\]
To calculate $B_1$, we begin by writing the first row of $\cX_1$ as
\[
(l_0(\ze),\dots,l_n(\ze)).
\]
Since
$\cX_1 \sim \cX_f$ as $\ze\to 0$ on $\Om_1$, we have
\[
(l_0(\ze),\dots,l_n(\ze)) \sim 
(e^{\frac 1{\ze}},e^{\frac{\om}{\ze}},\dots,e^{\frac {\om^n}{\ze}}).
\]
In particular this holds when $\ze\to 0$ on the positive real axis $\ze>0$. Since 

--- \ $g(\ze)$ must be a
linear combination of $l_0(\ze),\dots,l_n(\ze)$

--- \ $g(\ze)\to 0$ by (3) of Lemma \ref{integral}

--- \ amongst $l_0(\ze),\dots,l_n(\ze)$ only $l_{\frac12(n+1)}(\ze)\sim e^{-\frac1\ze}\to 0$

\no we conclude that $g(\ze)$ must be a scalar multiple of $l_{\frac12(n+1)}(\ze)$. 
By (3) of Lemma \ref{integral}
again, this scalar is just
$\ii (2\pi)^{\frac12(n+2)} (n+1)^{-\frac12}$,
because  $\sum_0^n b_i=0$ when $b_i=-m_i$. 
We have proved that
\begin{equation}\label{scalar}
g(\ze)=\ii (2\pi)^{\frac12(n+2)} (n+1)^{-\frac12} l(\ze),\quad l(\ze)=
l_{\frac12(n+1)}(\ze).
\end{equation}

Next we shall express $(l_0(\ze),\dots,l_n(\ze))$ in terms of $l(\ze)$:

\begin{proposition}\label{compare} We have $(l_0(\ze),\dots,l_n(\ze))=$
\[
\left(
l(\om^{\frac12(n+1)}\ze), \ \dots,\  l(\om \ze),l(\ze),l(\om^{-1}\ze),\ \dots,\  
l(\om^{-\frac12(n-1)}\ze)
\right)
P_m^T
\]
where $P_m^T=d_{n+1}^{\frac12}\, F\, d_{n+1}^{-\frac12}$, and $F$ is
the matrix in formula (3.7) of \cite{Ho17}.
\end{proposition}

\begin{proof} Proposition 3.4 of \cite{Ho17} gives
\begin{equation}\label{neutral}
F \tMz F^{-1} = 
\bp
-s_n & 1 & 0 & \cdots & 0
\\
-s_{n-1} & 0 & 1 & \cdots & 0
\\
\vdots & \vdots &  \vdots &  & \vdots 
\\
-s_1 & 0 & 0 & \cdots & 1
\\
-1 & 0 & 0 & \cdots & 0
\ep.
\end{equation}
Using (\ref{matrixMtilde}), it follows from this that $P_m^T \Mz P_m^{-T}$ has the same form, but with $s_1,\dots,s_n$ modified. Thus
the last $n$ columns of $P_m^T \Mz P_m^{-T}$ are exactly as in (\ref{neutral}).

The cyclic symmetry for the o.d.e.\ (\ref{node}) says that 
\[
d_{n+1}\cXz_k(\om^{-1}\ze)\Pi = \cXz_{k-\scriptstyle\frac2{n+1}}(\ze)
\]
(as $\cXz_k=h^{-1}\Phiz_k$ this is equivalent to the cyclic symmetry for $\Phiz_k$;
the latter is given in section 3 of \cite{GIL2}).  Hence we obtain
\[
d_{n+1}\cXz_1(\om^{-1}\ze)\Qz_1\Qz_{1+\scriptstyle\frac1{n+1}}\Pi = 
\cXz_1(\ze).
\]
For the first rows of these matrices, this means that
\[
(l_0(\om^{-1}\ze),\dots,l_n(\om^{-1}\ze))\Mz=(l_0(\ze),\dots,l_n(\ze)),
\]
and therefore
\[
(l_0(\om^{-1}\ze),\dots,l_n(\om^{-1}\ze))P_m^{-T}
\left( P_m^T \Mz P_m^{-T} \right)
=(l_0(\ze),\dots,l_n(\ze)) P_m^{-T}.
\]
Let us introduce
\[
(u_0,\dots,u_n)=(l_0(\om^{-1}),\dots,l_n(\om^{-1}))P_m^{-T}.
\]
Then the version of formula (\ref{neutral}) for $P_m^T \Mz P_m^{-T}$ gives
\[
u_i(\om^{-1}\ze)=u_{i+1}(\ze),\quad
0\le i\le n-1.
\]
We obtain $(u_0(\ze),\dots,u_n(\ze))=$
\[
(
u_{\frac12(n+1)}(\om^{\frac12(n+1)}\ze), \ \dots,\  
u_{\frac12(n+1)}(\ze),
\ \dots,\  
u_{\frac12(n+1)}(\om^{-\frac12(n-1)}\ze))
P_m^T.
\]
To complete the proof it suffices to show that
$u_{\frac12(n+1)}$ is $l_{\frac12(n+1)}$ (i.e.\ $l$).  But this is immediate
from the definition of $F$ (formula (3.7) of \cite{Ho17}), as
$F=
\bsp L &  0
\\
0 & U 
\esp
$
where $L$ and $U$ are, respectively, lower and upper triangular 
$\frac12(n+1) \times \frac12(n+1)$
matrices
with
all diagonal entries equal to $1$.
\end{proof}

From Proposition \ref{compare} we obtain 
\[
B_1^{-1}= 
\ii (2\pi)^{\frac12(n+2)} (n+1)^{-\frac12}
P_m^{-T},
\]
and we arrive at the main result of this section: 

\begin{theorem}\label{explicitD1}
$\cD_1=B_1^{-1}B = 
\ii (2\pi)^{\frac12(n+2)} (n+1)^{-\frac12}
P_m^{-T}
V_m^{-1} \Ga_m^{-1}$,
where $P_m,V_m,\Ga_m$ are as defined above.
\end{theorem} 

\begin{proof} By definition, $\cXi=\cX_1\cD_1$. We have defined $B,B_1$ by $\cXi=\cX B$ and
$\cX_1=\cX B_1$.  Thus $\cD_1=B_1^{-1}B$.  Substituting the values of $B,B_1$ obtained earlier, we obtain the result.
\end{proof}

The connection matrix $D_1$ for $\hat\om$ is
given by $D_1=\cD_1 h^{-1}$ where $h=\hat c t^m$ (Proposition \ref{cQandQ}).  

\begin{remark}
In the case $n+1=2$ we have $P_m=I$. 
\qed
\end{remark}
 
\subsection{Computation of the connection matrices $E_k$}\label{6.4} \ 

Theorem \ref{EintermsofD} expresses $E_1$ in terms of $D_1$, and we
have just calculated $D_1$. We shall deduce:

\begin{theorem}\label{explicitE1}
$
E_1=
-(V_m P_m^T)^{-1} 
(\hat c \Ga_m)^{-1} \De\,  \hat c \Ga_m\,  \De 
(V_m P_m^T) E_1^{\id},
$
where 
$E_1^{\id}=\tfrac1{n+1} C \Qi_{\frac n{n+1}}$.
\end{theorem}

The significance of $E_1^{\id}$ will become clear in the next section.

\begin{proof} Formula (\ref{E1intermsofD1}) gives
$E_1=\tfrac1{n+1} D_1 t^m\, \De\, (\overline{ D_1 t^m })^{-1} 
(\bar Q^{(0)}_{\frac{n}{n+1}})^{-1}
d_{n+1}^{-1}\,  C$. 
Let us substitute the value of $D_1t^m=\cD_1 h^{-1}t^m=\cD_1 \hat c^{-1}$ from 
Theorem \ref{explicitD1}, which is
\[
D_1t^m=
\kappa
P_m^{-T}
V_m^{-1} \Ga_m^{-1}
\hat c^{-1},
\quad
\kappa= \ii \, (2\pi)^{ \frac12(n+2) } 
(n+1)^{ -\frac12  } .
\]
Noting that $\bar\kappa=-\kappa$ and $\bar \Ga_m=-\Ga_m$, we obtain
\[
E_1=
\tfrac1{n+1}
(V_m P_m^T)^{-1} \Ga_m^{-1} \hat c^{-1} \De \hat c \Ga_m \bar V_m \bar P_m^T 
(\bar Q^{(0)}_{\frac{n}{n+1}})^{-1} d_{n+1}^{-1} C
\]
By direct calculation we find that
\[
\bar V_m = - \De V_m d_{n+1}.
\]
We also have
\[
\bar P_m^T = d_{n+1}^{-1} P_m^T d_{n+1},
\]
as this is equivalent to the fact that the matrix
$F=d_{n+1}^{-\frac12} P_m^T d_{n+1}^{\frac12}$
(defined in Proposition \ref{compare}) 
is real.
We obtain
\[
E_1=
-(V_m P_m^T)^{-1} \Ga_m^{-1} \hat c^{-1} \De \hat c \Ga_m \De (V_m P_m^T)
\tfrac1{n+1} d_{n+1} 
(\bar Q^{(0)}_{\frac{n}{n+1}})^{-1} d_{n+1}^{-1} C.
\]
Now, the $\th$-reality condition (the general formula is stated later on as (\ref{RHreality}) in section \ref{7})
gives
\[
(\bar Q^{(0)}_{\frac{n}{n+1}})^{-1} = C  \Qz_{\frac{n}{n+1}} C.
\]
Using the fact that $d_{n+1} C = C d_{n+1}^{-1}$, we obtain
\[
d_{n+1} 
(\bar Q^{(0)}_{\frac{n}{n+1}})^{-1} d_{n+1}^{-1} C= 
d_{n+1} C \Qz_{\frac{n}{n+1}} C d_{n+1}^{-1} C=
C d_{n+1}^{-1} \Qz_{\frac{n}{n+1}} d_{n+1} =
C \Qi_{\frac{n}{n+1}}.
\]
The stated result follows.
\end{proof} 

Observe that
\[
(\hat c \Ga_m)^{-1} \De\,  \hat c \Ga_m\,  \De =
\bp
\frac{ \hat c_n \Ga^{(n)}_m }{ \hat c_0 \Ga^{(0)}_m } & & &
\\
& \frac{ \hat c_{n-1} \Ga^{(n-1)}_m }{ \hat c_1 \Ga^{(1)}_m } & &
\\
\phantom{
\frac{ \hat c_0 \Ga^{(0)}_m }{ \hat c_n \Ga^{(n)}_m }
}
 & & \ddots & 
\\
& & & \frac{ \hat c_0 \Ga^{(0)}_m }{ \hat c_n \Ga^{(n)}_m } 
\ep
\]
where we have written $\Ga_m=\diag( \Ga^{(0)}_m,\dots, \Ga^{(n)}_m)$.
In view of this, we make the following definition:

\begin{definition}\label{ei}
For $i=0,1,\dots,n$, let $e_i=
-(\hat c_{n-i} \Ga^{(n-i)}_m)/(\hat c_i  \Ga^{(i)}_m) $.
\end{definition}

With this notation, Theorem \ref{explicitE1} becomes:

\begin{theorem}\label{explicitE1withe}
$
E_1
(E_1^{\id})^{-1}
=
(V_m P_m^T)^{-1} 
\!
\bp
e_0 & & & \\
 & e_1 & & \\
 & & \ddots & \\
 & & & e_n
\ep
\!
(V_m P_m^T).
$
\qed
\end{theorem}

Using the values of the $\Ga^{(i)}_m=(-1)^i C_0^{[i]}$ which were computed in section \ref{6.3}, we
find explicit expressions for the $e_i$:

\begin{corollary}\label{eiofost} For $i=0,1,\dots,n$, we have $e_i =$
\[
\textstyle
\frac{ \hat c_{n-i} } { \hat c_i }
\frac{  (n+1)^{m_{n-i} } } {  (n+1)^{m_i } }
\frac
{
\Gad\left(
\frac{
m^\pr_{n-i} - m^\pr_{n-i+1} 
}{n+1}
\right)
}
{
\Gad\left(
\frac{
m^\pr_{i} - m^\pr_{i+1} 
}{n+1}
\right)
}
\frac
{
\Gad\left(
\frac{
m^\pr_{n-i} - m^\pr_{n-i+2} 
}{n+1}
\right)
}
{
\Gad\left(
\frac{
m^\pr_{i} - m^\pr_{i+2} 
}{n+1}
\right)
}
\cdots
\frac
{
\Gad\left(
\frac{
m^\pr_{n-i} - m^\pr_{n-i+n} 
}{n+1}
\right)
}
{
\Gad\left(
\frac{
m^\pr_{i} - m^\pr_{i+n} 
}{n+1}
\right)
}.
\qed
\]
\end{corollary}

Note that the minus sign in 
Definition \ref{ei} is cancelled by the minus signs
in the $\Ga^{(n-i)}_m / \Ga^{(i)}_m$. 
We recall that $m_i+m_{n-i}=0$ and $\hat c_i \hat c_{n-i} = 1$ here.  As in
section \ref{6.3}, we extend to $i\in\Z$ 
the definitions of $m_i$ and $m^\pr_i=m_i-i$ (that is, we put $m_{i+n+1}=m_i$ and $m^\pr_{i+n+1}=m^\pr_i-(n+1)$).  

\begin{remark}  The (positive) constants $c,N$ in Definition \ref{Nmchat} may be chosen freely.  Let
us normalize them both to be $1$. Then for
$i=0,1,\dots,\tfrac{1}{2}(n+1)$
the formula for $e_i$ can be written very simply in terms of
the original data $c_i,k_i$ of the connection form $\hat\om$ as
\[
\textstyle
e_i= \frac{1}{c_{i+1}c_{i+2}\cdots c_{n-i}}
\ 
\frac
{
\Ga(\al_{n-i+1})
}
{
\Ga(\al_{i+1})
}
\frac
{
\Ga(\al_{n-i+1}+\al_{n-i+2}) 
}
{
\Ga(\al_{i+1}+\al_{i+2})
}
\cdots
\frac
{
\Ga(\al_{n-i+1}+\cdots+ \al_{n-i+n})
}
{
\Ga(\al_{i+1}+\cdots+ \al_{i+n})
}
\]
where $\al_i=k_i+1$ (and we extend this to $i\in\Z$ by $\al_{i+n+1}=\al_i$).  
\qed
\end{remark}

\begin{remark} Substituting $D_1t^m=
\kappa P_m^{-T} V_m^{-1} \Ga_m^{-1} \hat c^{-1}$
into (\ref{D-cyc}) gives a diagonalization of
$\Qz_1\Qz_{1+\scriptstyle\frac1{n+1}}\Pi$:
\[
\Qz_1\Qz_{1+\scriptstyle\frac1{n+1}}\Pi = D_1 \om^{m^\pr} D_1^{-1} \\
= 
(V_m P_m^T)^{-1}  \om^{m^\pr} V_m P_m^T.
\]
Thus,  both $E_1(E_1^{\id})^{-1}$ and $\Qz_1\Qz_{1+\scriptstyle\frac1{n+1}}\Pi$ are
diagonalized by $V_m P_m^T$.  This is a fact of independent interest, so it
may be worthwhile making some further comments here.

We could have deduced the diagonalization of $\Qz_1\Qz_{1+\scriptstyle\frac1{n+1}}\Pi$
(without calculating $D_1$) from formula (\ref{neutral}) in the proof of Proposition \ref{compare}, because
that formula says that 
$P_m^T \Qz_1\Qz_{1+\scriptstyle\frac1{n+1}}\Pi  P_m^{-T}$ is in rational normal form, and it
is well known that the rational normal form is diagonalized by a Vandermonde matrix ($V_m$ in
the present situation).

The simultaneous diagonalizability could also have been
deduced a priori, from the cyclic symmetry (\ref{E-cyc}):
we have
\begin{align*}
&(a) \ E_1 = \Qz_1\Qz_{1+\scriptstyle\frac1{n+1}}\Pi \  E_1\    
\Qi_1\Qi_{1+\scriptstyle\frac1{n+1}}\Pi 
\\
&(b) \ E_1^{\id} = \Qz_1\Qz_{1+\scriptstyle\frac1{n+1}}\Pi \  E_1^{\id}\    
\Qi_1\Qi_{1+\scriptstyle\frac1{n+1}}\Pi 
\end{align*}
then multiplication of (a) by the inverse of (b) gives
\[
E_1 (E_1^{\id})^{-1}= 
 \Qz_1\Qz_{1+\scriptstyle\frac1{n+1}}\Pi \
 E_1 (E_1^{\id})^{-1} 
(\Qz_1\Qz_{1+\scriptstyle\frac1{n+1}}\Pi)^{-1},
\]
i.e.\ $E_1 (E_1^{\id})^{-1}$ commutes with $ \Qz_1\Qz_{1+\scriptstyle\frac1{n+1}}\Pi$.

Nevertheless, the explicit formula for the eigenvalues of $E_1 (E_1^{\id})^{-1}$ 
in Corollary \ref{eiofost} depends
on the calculation of $D_1$ in this section.
\qed
\end{remark}

To conclude, we can say that
the eigenvalues $e_0,\dots,e_n$ of $E_1(E_1^{\id})^{-1}$
and the eigenvalues
$\om^{m^\pr_0},\dots,\om^{m^\pr_n}$ of
$\Mz=\Qz_1\Qz_{1+\scriptstyle\frac1{n+1}}\Pi$
represent the monodromy data of the o.d.e.\
(\ref{psizeta}) corresponding to the solutions of the tt*-Toda equations (\ref{ost})
which were
constructed in section \ref{2.3}. 
Similarly, the tilde versions $\tE_1 (\tE_1^{\id})^{-1}$ and $\tMz=\tQz_1\tQz_{1+\scriptstyle\frac1{n+1}}\hat\Pi$ commute. Their eigenvalues
are $e_0,\dots,e_n$ and
$\om^{m^\pr_0-\frac12},\dots,\om^{m^\pr_n-\frac12}$.

\section{Riemann-Hilbert problem}\label{7}

In sections \ref{6.2} and \ref{6.3} we computed the monodromy data associated to certain local solutions of (\ref{ost}) near $t=0$, namely those arising as in Proposition \ref{localsol} from the data $c_0 z^{k_0},\dots,c_n z^{k_n}$
(with $c_i>0,k_i>-1$). The method used the Iwasawa factorization, which can be interpreted as solving a Riemann-Hilbert problem.

Our aim in this section is to go in the opposite direction: to reconstruct solutions from monodromy data.  This is a different type of Riemann-Hilbert problem, which we shall solve under certain conditions; in particular we shall produce local solutions of (\ref{ost}) near $t=\infty$, some of which will be global solutions. How these are related to the local solutions in Proposition \ref{localsol} will be discussed in section \ref{8}. 

\subsection{Preparation}\label{7.1}\ 

As monodromy data we take hypothetical Stokes factors $\Qz_k,\Qi_k$ and hypothetical  connection matrices $E_k$, motivated by the fact that {\em any} local solution of (\ref{ost}) (not necessarily at $t=0$) must give rise to such data as in sections \ref{stokesforhatal-zero}, \ref{stokesforhatal-infinity}, \ref{connforhatal}. Moreover, the shape of the Stokes factors must be as described in section \ref{6.2}, excluding only the final statement of Theorem \ref{siofost} which expresses $s_1,\dots,s_{\frac12(n+1) }$
in terms of $m_0,\dots,m_n$.  

Let us choose hypothetical \ll Stokes parameters\rr $s_1,\dots,s_{\frac12(n+1) }$
as in Definition \ref{si},
hence hypothetical Stokes factors $\Qz_k,\Qi_k$. 
We do {\em not} assume that $s_1,\dots,s_{\frac12(n+1) }$ arise from the solutions in Proposition \ref{localsol}; we just take $s_1,\dots,s_{\frac12(n+1) }$ to be arbitrary real numbers. 
We choose also hypothetical connection matrices $E_k$.   
We shall specify $E_k$ more precisely in a moment.

From this data, our aim is to construct functions $\Psiz_k,\Psii_k$ which are holomorphic on the appropriate sectors and which satisfy the \ll jump conditions\rr 
\[
\Psiz_{k+\scriptstyle\frac1{n+1}}=\Psiz_k \Qz_k,
\quad
\Psii_{k+\scriptstyle\frac1{n+1}}=\Psii_k \Qi_k,
\quad
\Psii_k=\Psiz_k E_k.
\]
It will be convenient to replace the connection matrices $E_k$ by new matrices $Z_k$, defined by
\begin{equation}\label{Zdef}
\Psii_k=\Psiz_{ {\scriptstyle\frac{2n+1}{n+1} - k}  }    Z_{ {\scriptstyle\frac{2n+1}{n+1} - k}  },
\end{equation}
as these matrices relate \ll opposite\rr sectors (cf.\ Figure \ref{RH1} below).
From the definition it follows that
\begin{equation}\label{Zfromdef}
Z_{k+\scriptstyle\frac {1}{n+1}} =
\left( \Qz_k \right)^{-1}
Z_k
\left( \Qi_{   {\scriptstyle\frac{2n}{n+1} - k}  } \right)^{-1}.
\end{equation}
We recall the analogous relation (\ref{Efromdef}) for the $E_k$:
\begin{equation}\label{Efromdef2}
E_{k+\scriptstyle\frac1{n+1}} = \left(\Qz_{k}\right)^{-1}  E_{k}
\ \Qi_{k}.
\end{equation}
For $k=1$ we have the relation
\begin{equation}\label{RHinitial}
E_1=Z_1 \Qi_{\scriptstyle\frac {n}{n+1}},
\end{equation}
as
$
\Psii_1(\ze)=\Psiz_1(\ze) E_1 = 
\Psii_{   \scriptstyle\frac{n}{n+1}  }(\ze) Z_1^{-1} E_1 = 
\Psii_1(\ze) \left( \Qi_{ \scriptstyle\frac{n}{n+1} }\right)^{-1} Z_1^{-1} E_1.
$
Formulae (\ref{Zfromdef}), (\ref{Efromdef2}), (\ref{RHinitial}) show that the $Z_k$ are equivalent to the $E_k$.

We aim to choose specific 
$E_k$ in such a way that all the $Z_k$ are the same, independent of $k$.  One such choice of $E_k$ is given by:

\begin{lemma}\label{globalE}  Let $E_1=\tfrac1{n+1} C \Qi_{\scriptstyle\frac {n}{n+1}}$.
Then $Z_k=\tfrac1{n+1} C$ for all $k$.
\end{lemma}

\begin{proof} Using the symmetries, we shall prove first that
\begin{equation}\label{RHcond}
\Qi_k= \left( \tfrac1{n+1} C \right)^{-1} 
\left(
\Qz_{\scriptstyle\frac {2n}{n+1} -k}
\right)^{-1}
\tfrac1{n+1} C
\end{equation}
for all $k$.  

To begin,  we note that
the $\th$-reality condition (section 3 of \cite{GH1}) gives
\begin{equation}\label{RHreality}
\Qz_k=  C \
\left( \, \overline{\Qz_{ {\scriptstyle\frac{2n}{n+1} - k}  }}\, \right)^{-1} \  C.
\end{equation}
On the other hand,  (\ref{QzandQi}) gives:
\[
\text{(a)}\quad  \Qi_k= d_{n+1}^{-1} \Qz_k d_{n+1}.
\]
From (\ref{Qintermsofsij}) and Proposition \ref{qij} we know that all
$\tQz_k$ are real. By (\ref{tQzandQz}), this means:
\[
\text{(b)}\quad  \overline{ \Qz_k}= d_{n+1}^{-1} \Qz_k d_{n+1}.
\]
Substituting (a) and (b) into (\ref{RHreality}) we obtain 
(\ref{RHcond}).

To prove the lemma, let us put
$E_1=\tfrac1{n+1} C \Qi_{\scriptstyle\frac {n}{n+1}}$.
By (\ref{RHinitial}), this means $Z_1=\tfrac1{n+1} C$.  
Proceeding by induction, if we assume that $Z_k=\tfrac1{n+1} C$,
then (\ref{Zfromdef}) and (\ref{RHcond})
give $Z_{ k+ \scriptstyle\frac {1}{n+1}}=\tfrac1{n+1} C$.
This completes the proof.
\end{proof}

\begin{remark} Evidently the scalar $\tfrac 1{n+1}$ does not play any role in the above proof; if
we set 
$E_1=  a\, C \Qi_{\scriptstyle\frac {n}{n+1}}$, 
then we obtain $Z_k = a\, C$ for all $k$.  
The choice $a=\tfrac 1{n+1}$ comes from the fact (from $\Psii_k=\Psiz_k E_k$)
that 
$\det E_1= \det \Om^{-2} = \det \frac1{n+1} C$, when
$E_1$ arises from a solution of (\ref{ost}).  (The choice $a=-\tfrac 1{n+1}$ also
gives the same determinant, but it turns out not to give a
(real) solution of (\ref{ost}) --- cf.\ section 4 of \cite{GIL2}.)
\qed
\end{remark}

\subsection{The contour}\label{7.2}\

To formulate a Riemann-Hilbert problem (in the manner of \cite{FIKN06}, chapter 3) it is necessary to specify an oriented contour with  jump matrices in the punctured $\ze$-plane $\C^\ast$ (not $\tilde\C^\ast$). 
For this purpose we choose to parametrize angles in the $\ze$-plane
using the interval $[-\frac{\pi}{n+1},2\pi-\frac{\pi}{n+1})$, i.e.\
$[\thz_1,2\pi+\thz_1)$.  
\begin{figure}[h]
\begin{center}
\includegraphics[angle=270,origin=c,scale=0.4, trim= 0 200 0 150]{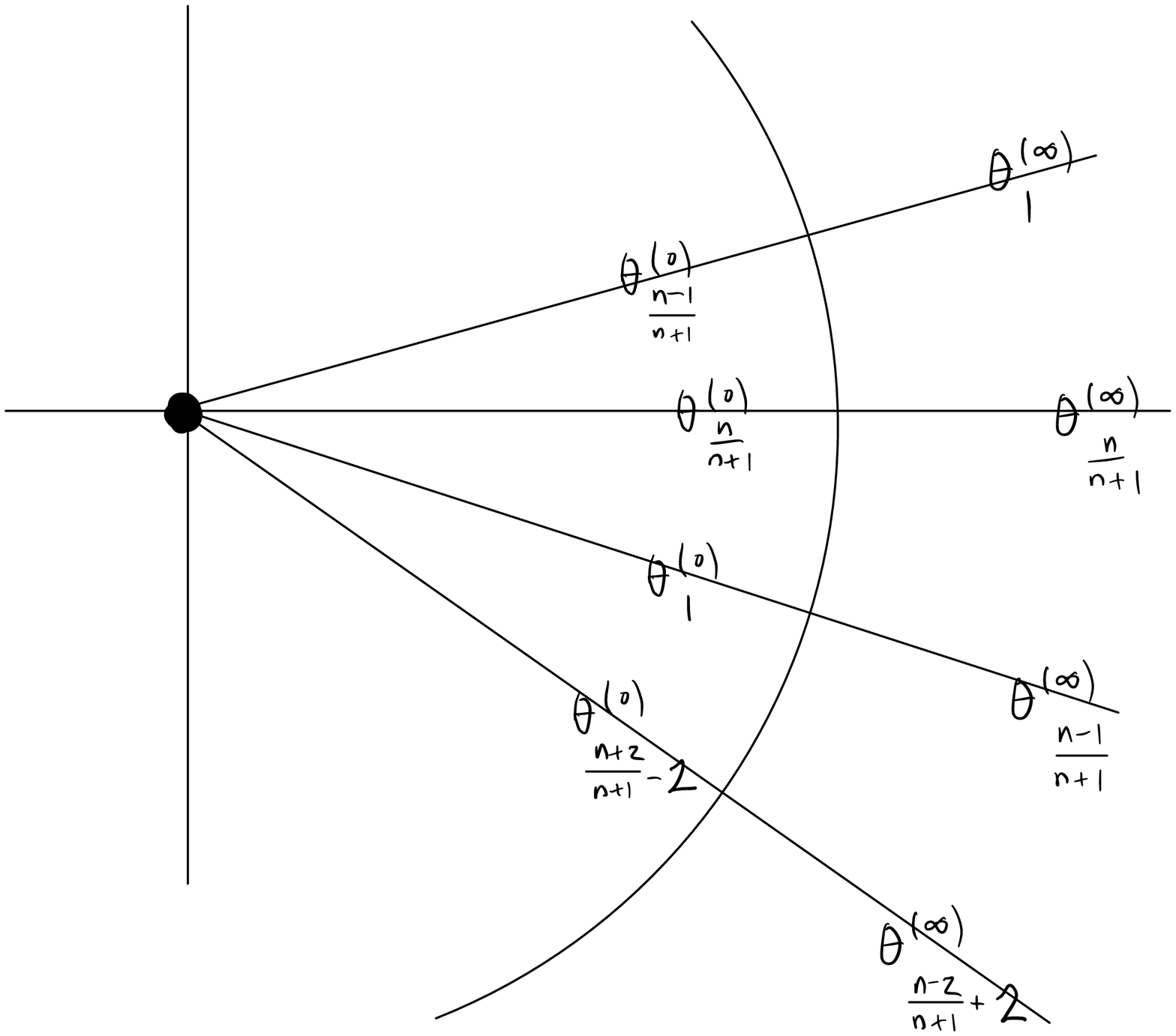}
\end{center}
\caption{Contour.}\label{RHcon}
\end{figure}
For the contour we take the union of the $2n+2$
singular directions in our \ll reference sector\rr
$[-\frac{\pi}{n+1},2\pi-\frac{\pi}{n+1})$, together with the circle $\vert \ze\vert = x^2$ (where $x\in(0,\infty)$ is fixed), 
as in Figure \ref{RHcon}.
These singular directions
are:
\begin{gather*}
 -\tfrac \pi{n+1} =  \thz_1, \thz_{\scriptstyle\frac n{n+1}},\dots,
\thz_{-\scriptstyle\frac n{n+1}}=2\pi - 2\tfrac \pi{n+1}
\left(  < 2\pi - \tfrac \pi{n+1} \right)
\\
-\tfrac \pi{n+1} = \thi_{\scriptstyle\frac{n-1}{n+1}}, \thi_{\scriptstyle\frac{n}{n+1}},\dots,
\thi_{2+\scriptstyle\frac{n-2}{n+1}}=2\pi - 2\tfrac \pi{n+1}
\left(  < 2\pi - \tfrac \pi{n+1} \right)
\end{gather*}
To emphasize that we are working in $\C^\ast$, let us write
$\fthz_k$, for these particular values of $k$.  For arbitrary $k$ 
we define $\fthz_k = \fthz_{k^\pr}$ when $k^\pr \cong k$ mod $2\Z$ 
and $-\frac n{n+1} \le k^\pr \le 1$. It follows that
$\fthz_{k+2}=\fthz_k$, in contrast to the relation 
$\thz_{k+2}=\thz_k - 2\pi$ which holds in $\tilde\C^\ast$.  For 
$\fthi_k$ we make similar definitions.
On the closed sector of $\C^\ast$ bounded by
$\fthz_k,\fthz_{k-\scriptstyle\frac1{n+1}}$ (inside the circle) we define a holomorphic function
$\fPsiz_k$ by
\begin{equation}\label{flatpsi}
\fPsiz_k(\ze)=\Psiz_k(\ze^\pr)
\end{equation}
where $\ze^\pr\in\Omz_k$ covers $\ze\in\C^\ast$.  
Similarly we define $\fPsii_k$ on the closed sector bounded by
$\fthi_{k-\scriptstyle\frac1{n+1}},\fthi_k$ (outside the circle).
Whereas the functions $\Psiz_k, \Psii_k$ are defined on the universal cover $\tilde\C^\ast$, 
the functions $\fPsiz_k, \fPsii_k$ are defined only on the sectors specified.
These are illustrated in Figure \ref{RH1}. We refer to this diagram as the
\ll Riemann-Hilbert diagram\rrr.  The jumps on the circle are given by the connection matrices $Z_\ast$
and the jumps on the rays are given by the Stokes factors $Q_\ast$.

\begin{figure}[h]
\begin{center}
\includegraphics[angle=270,origin=c,scale=0.4, trim= 0 200 0 150]{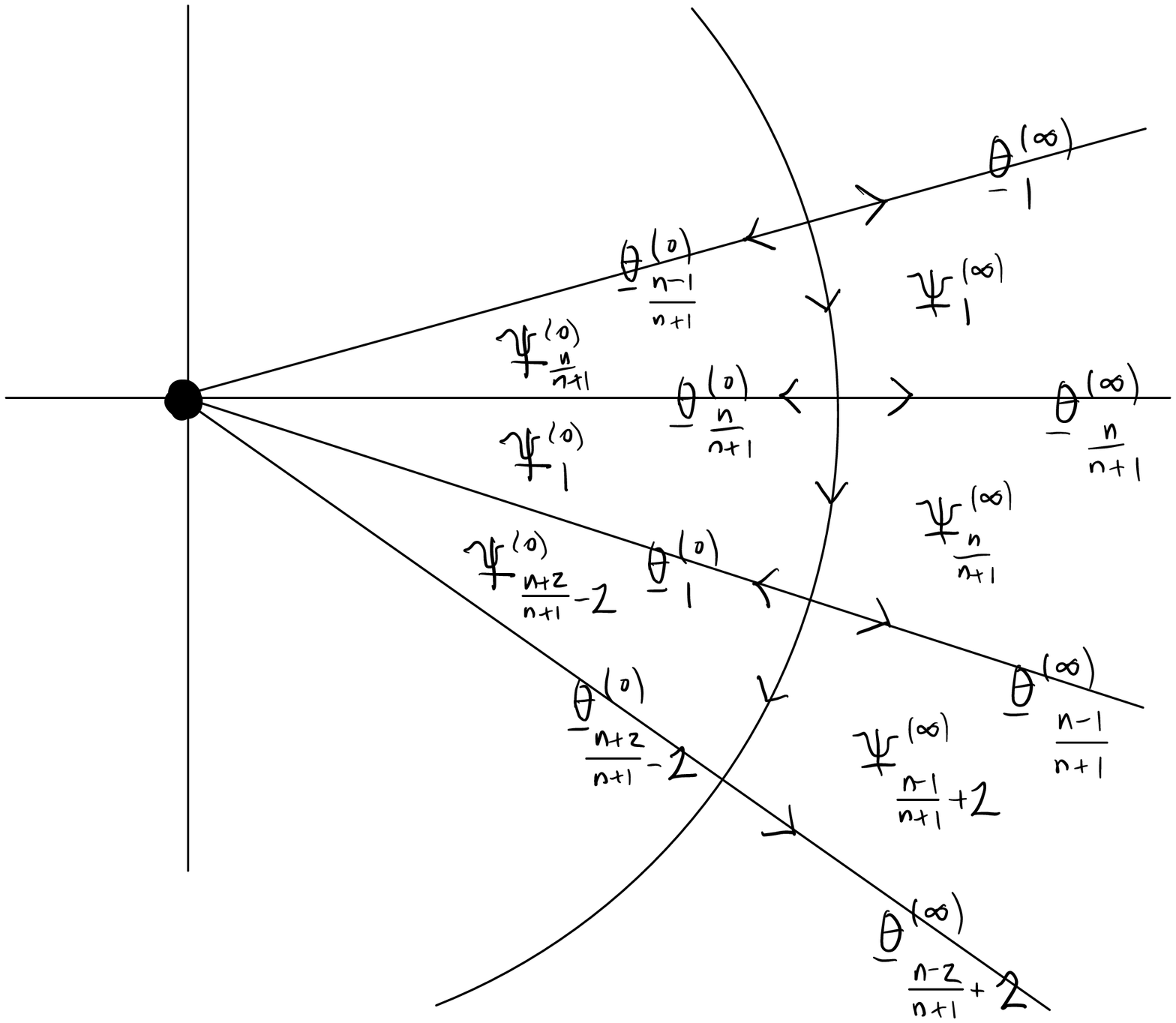}
\end{center}
\caption{Riemann-Hilbert diagram.}\label{RH1}
\end{figure}
The contours are oriented so that
\begin{equation}\label{jumprule}
\fPsi_{  \text{left}  } = \fPsi_{ \text{right} }\ \times \ \text{jump matrix}.
\end{equation}
The problem of recovering the functions from 
the jumps on the contour will be referred to as \ll the Riemann-Hilbert problem\rrr.  
The Riemann-Hilbert diagram  in Figure \ref{RH1} constructed 
from the functions $\fPsiz_k,\fPsii_k$
is the model for this. Tautologically, it is solvable (i.e.\ the functions $\fPsiz_k,\fPsii_k$ exist)
if the monodromy data is exactly the data that we have calculated in
section \ref{6}.  We shall now consider other monodromy data and attempt to find
corresponding functions $\fPsiz_k,\fPsii_k$. For this purpose it is convenient to
modify the Riemann-Hilbert diagram.

If we replace $\Psiz_k$ by $\Psiz_k \tfrac1{n+1} C$ 
then all jumps on the circle will be the identity.  The Riemann-Hilbert diagram in Figure \ref{RH1}
can then be replaced by 
the simplified Riemann-Hilbert diagram in Figure \ref{RH2}.
\begin{figure}[h]
\begin{center}
\includegraphics[angle=270,origin=c,scale=0.4, trim= 0 200 0 150]{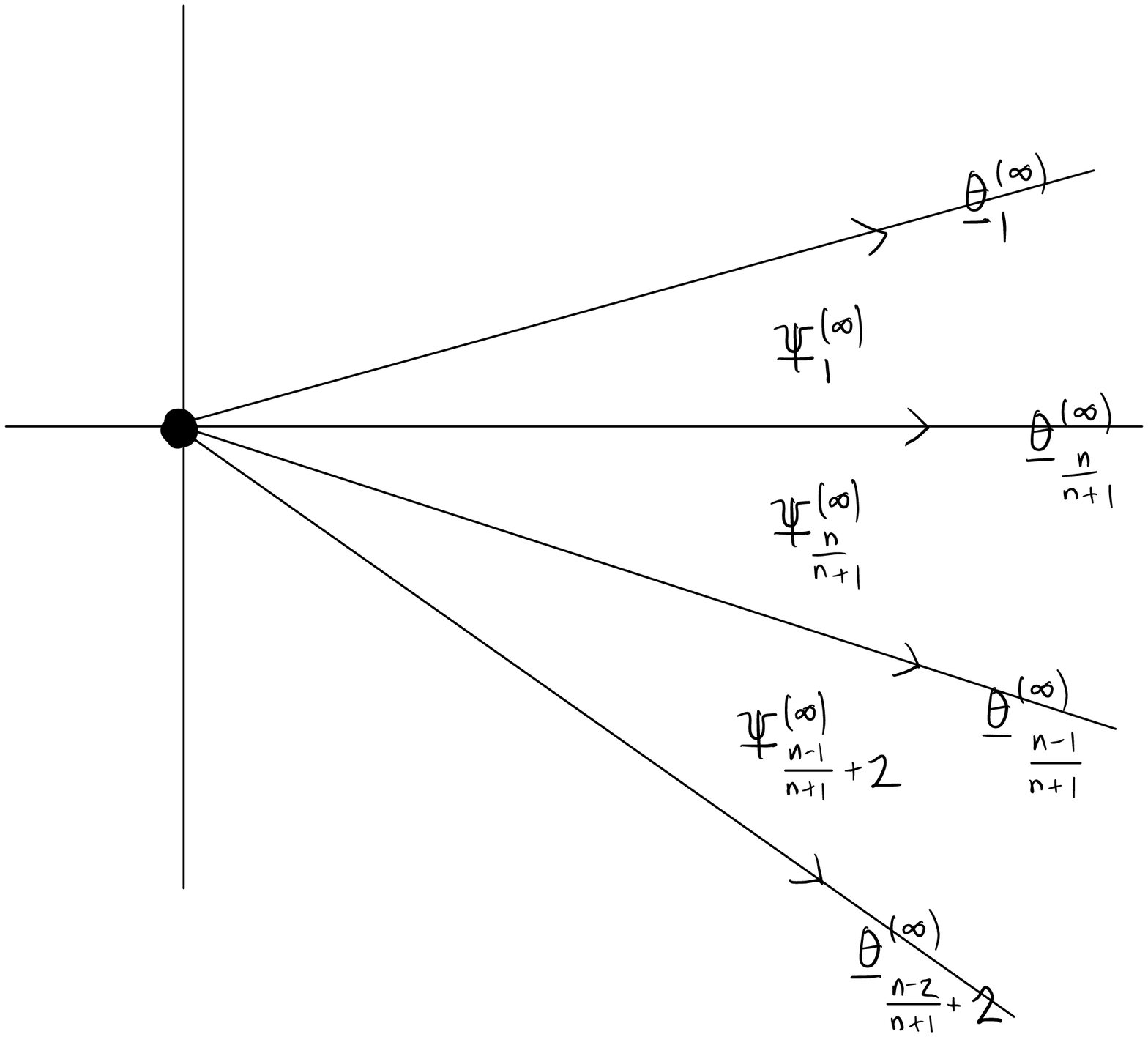}
\end{center}
\caption{Simplified Riemann-Hilbert diagram.}\label{RH2}
\end{figure}

A further modification will be useful.
Bearing in mind the asymptotics 
\begin{align*}
\Psiz_k(\ze)&\sim
e^{-w}\, \Omega \left( I + O(\ze) \right) e^{\frac1\zeta  d_{n+1}},
\quad
\ze\to 0
\\
\Psii_k(\ze)&\sim
e^w \Om^{-1} \left( I + O(1/\ze) \right) e^{x^2 \zeta  d_{n+1}},
\quad
\ze\to \infty
\end{align*} 
(from sections \ref{stokesforhatal-zero} 
and \ref{stokesforhatal-infinity}), 
we shall replace $\fPsii_k$ by $Y_k$, defined as follows:

\begin{definition}\label{Yk} Let
$Y_k(\ze)= (e^w \Om^{-1})^{-1} \fPsii_k(\ze) e(\ze)^{-1}$
where $e(\ze)$ is defined by
$e(\ze)= e^{  \frac1\ze d_{n+1}^{-1}  }
e^{  x^2\ze d_{n+1}  }$.
\end{definition} 

The form of $e(\ze)$ is motivated by:

\begin{proposition}\label{Ylimits}  If $Y_k$ arises from any (local) solution $w$ of (\ref{ost}), then

(1) $\lim_{\ze\to 0} Y_k(\ze) = \Om e^{-2w} \Om^{-1}$

(2) $\lim_{\ze\to \infty} Y_k(\ze) = I$

\no with exponential convergence in both cases.
\end{proposition}

\begin{proof} As $\ze\to\infty$, we have 
$Y_k(\ze)= (e^w \Om^{-1})^{-1} \fPsii_k(\ze) e(\ze)^{-1}
= (I + O(1/\ze)) e^{x^2 \zeta  d_{n+1}} e(\ze)^{-1}$,
from which (2) is immediate.  

Since
$\Psii_k(\ze)=\Psiz_{ {\scriptstyle\frac{2n+1}{n+1} - k}  }(\ze)    \tfrac1{n+1} C$
by  (\ref{Zdef}), and 
$\Psiz_k(\ze)\sim   
e^{-w}\, \Omega ( I + O(\ze) ) e^{\frac1\zeta  d_{n+1}}$
as $\ze\to 0$, we obtain
$\lim_{\ze\to 0} Y_k(\ze) =\Om e^{-2w} \Om \tfrac1{n+1} C$.
Here we make use of the fact that $d_{n+1}C d_{n+1}=C$. 
As $\frac1{n+1}C=\Om^{-2}$, we obtain (1).   
\end{proof}

The modified Riemann-Hilbert diagram (with $Y_k$ instead of $\fPsii_k$) will be the basis of our Riemann-Hilbert problem:  to reconstruct functions $Y_k$ starting from the monodromy data. By
(1) of Proposition \ref{Ylimits}, this will produce a solution $w$ of the tt*-Toda equation.
Instead of the jumps $\Qi_k$ in Figure \ref{RH2}, however, the modification gives new jumps $G_k(\ze)$:

\begin{definition} For $\ze$ in the ray $\fthi_k$, 
$G_k(\ze)=
e(\ze) \, \Qi_k e(\ze)^{-1}$.
\end{definition}

The tilde version of $Y_k$ is
$
\tY_k(\ze)= (e^w \Om^{-1} d_{n+1}^{-\frac12})^{-1} \ftPsii_k(\ze) e(\ze)^{-1}.
$
We obtain
\begin{equation}\label{Gtilde} 
\tG_k(\ze)=
e(\ze) \, \tQi_k e(\ze)^{-1}.
\end{equation}
The tilde versions of (1) and (2) in Proposition \ref{Ylimits} are:
\begin{align}\label{tildeYlimitsz}
&\lim_{\ze\to 0} \tY_k(\ze) =  
d_{n+1}^{\frac12} \Om e^{-2w } \Om^{-1} d_{n+1}^{-\frac12}
\\
\label{tildeYlimitsi}
&\lim_{\ze\to \infty} \tY_k(\ze) = I.
\end{align}

\subsection{The Riemann-Hilbert problem}\label{7.3}\ 

\begin{figure}[h]
\begin{center}
\includegraphics[angle=270,origin=c,scale=0.4, trim= 0 200 0 150]{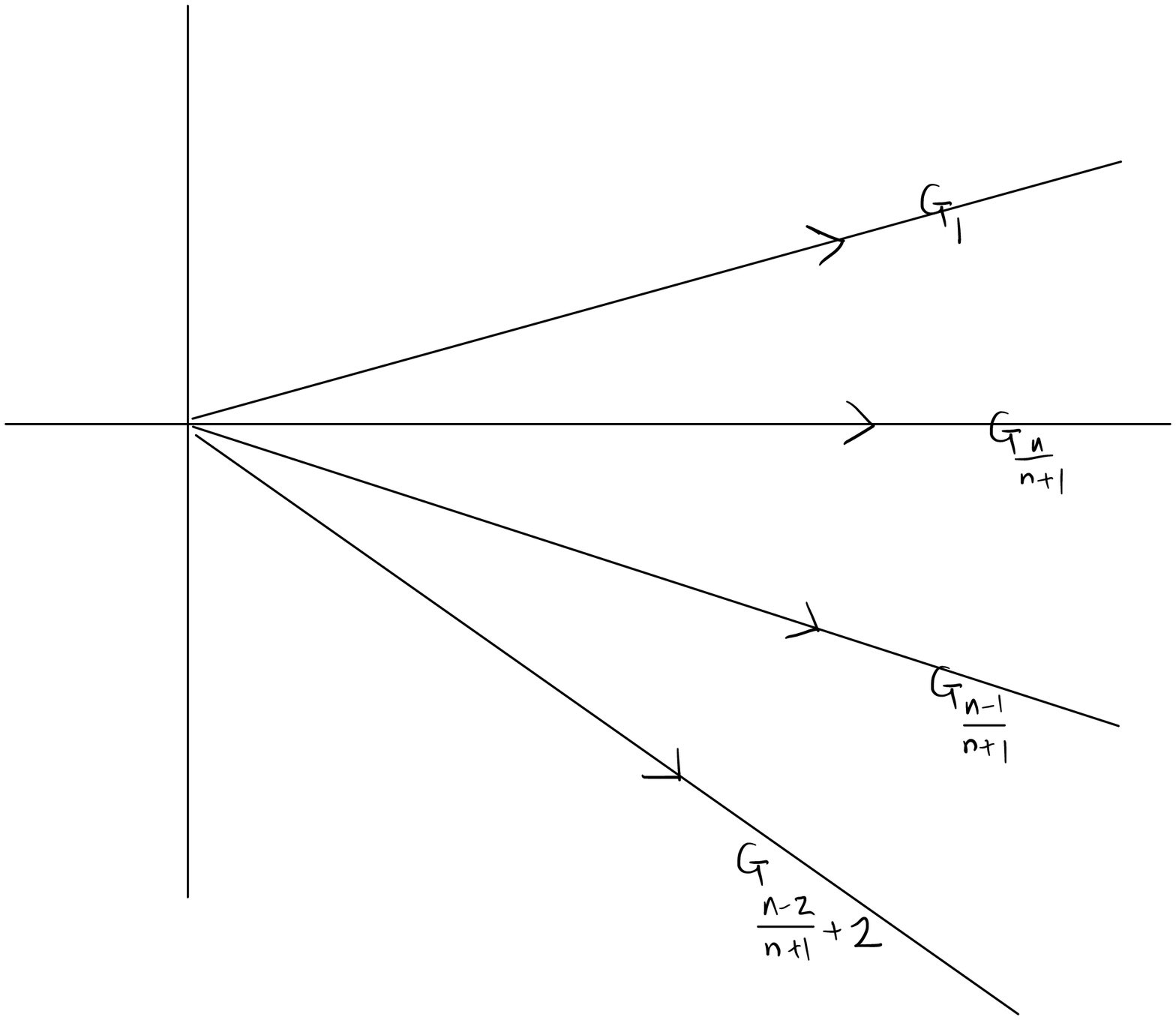}
\end{center}
\caption{The Riemann-Hilbert problem.}\label{RH3}
\end{figure}
In this section we take as our starting point the functions $G_k$ (or $\tG_k$) on the contour, as in Figure \ref{RH3},
and seek corresponding functions $Y_k$ (or $\tY_k$).  This is the Riemann-Hilbert problem.
The function $G_k$ was defined in terms of the 
hypothetical Stokes factors $\Qi_k$.  
We emphasize that $G_k$ depends on the variable $x=\vert t\vert$ 
of (\ref{ost}) as well as on $\ze$; thus we have a Riemann-Hilbert problem
for each $x$.

We shall show that the
Riemann-Hilbert problem is solvable for sufficiently large $x$.
For this we use the formulation of \cite{FIKN06} (pages 102/3).

Our contour is
\[
\Ga=
\fthi_{\scriptstyle\frac{n-1}{n+1}} \cup \fthi_{\scriptstyle\frac{n}{n+1}} \cup \cdots \cup
\fthi_{2+\scriptstyle\frac{n-2}{n+1}}.
\]
We shall denote by $G, \Qi$ the matrix-valued functions on $\Ga$
defined respectively by
\[
\Ga\vert_{\fthi_k} = G_k,
\quad
\Qi\vert_{\fthi_k} = \Qi_k,
\]
and similarly for $\tG, \tQi$.

The formulation of \cite{FIKN06} requires that (for fixed $x$)
$G(\ze)$ approaches $I$ rapidly (both in $L^2$ norm and $L^\infty$ norm)
as $\ze\to\infty$ along any infinite component of $\Ga$, and we shall verify that $G(\ze)$ has this property this next.
From the formula (\ref{Qintermsofsij}) for $\tQz_k (=\tQi_k)$ we have
\[
\tQi= I + \sum_{i\ne j} s_{i,j}  \, e_{i,j},
\]
so let us work with $\tG=d_{n+1}^\frac12 G d_{n+1}^{-\frac12}=e(\ze) \tQi e(\ze)^{-1}$
(instead of $G$).
To calculate this explicitly, we need to calculate $e(\ze) \, e_{i,j} \, e(\ze)^{-1}$:

\begin{lemma}\label{Gonthetak}  For $\ze=re^{\i\fthi_k}$, i.e.\ for $\ze$ in the $\fthi_k$ ray,
we have
$e(\ze) \, e_{i,j} \, e(\ze)^{-1} =
e^{ (-rx^2 - \frac1r )\vert \om^j - \om^i\vert } \, e_{i,j}$.
\end{lemma}

\begin{proof} We have
$e^{x^2\ze d_{n+1} } \ e_{i,j}\  e^{-x^2\ze d_{n+1} } 
= e^{ x^2\ze(\om^i - \om^j)  } e_{i,j}$. 
On the ray $\ze=r e^{ \i \fthi_k }$ (with $r>0$) we have
\begin{align*}
x^2\ze(\om^i - \om^j) &= -rx^2 (\om^j - \om^i) e^{ \i \fthi_k }
\\
&= -rx^2 (\om^j - \om^i) e^{ -\i \arg(\om^j -\om^i) } 
\\
&= -rx^2 \vert \om^j - \om^i \vert,
\end{align*}
as $\fthi_k=-\thz_k=- \arg(\om^j -\om^i)$.
Thus, 
$e^{x^2\ze d_{n+1} } \ e_{i,j}\  e^{-x^2\ze d_{n+1} } =
e^{ -rx^2 \vert \om^j - \om^i\vert } \, e_{i,j}$.

Similarly,
$e^{\frac 1{\ze} d_{n+1}^{-1} } \ e_{i,j}\  e^{-\frac 1{\ze} d_{n+1}^{-1} } = 
e^{\frac 1{\ze} (\om^{-i} - \om^{-j}) } e_{i,j}$, and
\begin{align*}
\overline{
\tfrac1\ze (\om^{-i} - \om^{-j} )
}
&=\tfrac 1r (\om^{i} - \om^{j}) e^{ \i \fthi_k }
\\
&=-\tfrac 1r (\om^{j} - \om^{i}) e^{ -\i \arg(\om^j -\om^i) }
\\
&= -\tfrac1r  \vert \om^j - \om^i \vert.
\end{align*}
Thus
$e^{\frac 1{\ze} d_{n+1}^{-1} } \ e_{i,j}\  e^{-\frac 1{\ze} d_{n+1}^{-1} } =
e^{ - \frac1r \vert \om^j - \om^i\vert } \, e_{i,j}$.
The stated result follows.
\end{proof}

The following notation will be useful.

\begin{definition} For
$p=1,\dots,n$, let
$L_p= \vert \om^p - \om^0\vert  = 2\sin \tfrac {p}{n+1} \pi$.
\end{definition}

Using this, our explicit formula for $\tG$ is:
\[
\tG(\ze)= I + \sum_{i\ne j} 
e^{ (-rx^2 - \frac1r )L_{\vert j-i\vert} } \, 
s_{i,j} \, e_{i,j}.
\]
We can re-parametrize the rays by putting $l=rx$; this gives
the more symmetrical expression
\begin{equation}\label{explicitG}
\tG(\ze)= I + \sum_{i\ne j} 
e^{ -x(l+ \frac1l) L_{\vert j-i\vert} } \, 
s_{i,j} \, e_{i,j}.
\end{equation}
Evidently $\tG(\ze)$ approaches $I$ exponentially as $r\to\infty$ (or $l\to\infty$).  In fact we
have
\begin{equation}\label{Gestimate}
\vert (\tG (\ze) -I)_{ij}\vert \le  A_{\vert j-i\vert} e^{ -2L_{\vert j-i\vert} x }
\end{equation}
for some constants $A_{\vert j-i\vert}$ when $x$ is sufficiently large.

This shows that $G(\ze)$ approaches $I$ (both in $L^2$ norm and $L^\infty$ norm)
sufficiently rapidly as to satisfy the solvability criterion of \cite{FIKN06} (Theorem 8.1):

\begin{theorem}\label{Yatinfinity}
 The Riemann-Hilbert problem of Figure \ref{RH3} is uniquely solvable when $x$ is sufficiently large. Concretely this means that there is some $R>0$, depending on the given monodromy data, such that the function $\tY$ is defined for $x\in (R,\infty)$ and has jumps on the contour $\Ga$ given by $\tG$.  

Furthermore the solution $\tY$ satisfies
\begin{equation}\label{asymptYatinfinity}
\displaystyle \tY(0) = I  + \tfrac{1}{2\pi\ii}\int_{\Ga}  
\frac{\tG(\ze) - I}{\ze} \  d\ze
+ (O_{i,j})
\end{equation}
as $x\to\infty$, where
$(O_{i,j})$ is the matrix with $(i,j)$ entry $O(e^{ - 4 L_{\vert j-i \vert} x })$.
\qed
\end{theorem}
 
The solution $\tY$ (i.e.\ the collection of piecewise holomorphic functions $\tY_k$, with jumps $\tilde G_k$) 
produces functions $\ftPsii_k=(e^w \Om^{-1} d_{n+1}^{-\frac12}) \tY_k(\ze)  e(\ze)$ with jumps $\tQi_k$. 

Moreover, it can be shown these functions satisfy the system 
(\ref{ode-hatal}), hence the corresponding functions $w_i:(R,\infty)\to \R$ 
(defined by formula (1) of Proposition \ref{Ylimits})
satisfy the
tt*-Toda equations (\ref{ost}).  
This argument is the same as that given in section 3.4 of \cite{GIL2}.  

We have now achieved our goal of producing local solutions of (\ref{ost}) near $t=\infty$:

\begin{theorem}\label{atinfinity} Let $s_1,\dots,s_{\frac12(n+1) }$ be real numbers, 
and let the matrices $\tQi_k$ be defined in terms of $s_1,\dots,s_{\frac12(n+1) }$ as
in section \ref{6}.  Then there is a unique solution $w$ of (\ref{ost}) on an interval
$(R,\infty)$, where $R$ depends on  $s_1,\dots,s_{\frac12(n+1) }$, such
that the associated monodromy data is given by the $\Qi_k$ and
$E_1^{\id}=\tfrac1{n+1} C \Qi_{\scriptstyle\frac {n}{n+1}}$.
\qed
\end{theorem}

In the next section we shall investigate the relation between these solutions and the
local solutions of near $t=0$ which were constructed in section \ref{2.3}.

To end this section, we note the important consequence 
that Theorem \ref{Yatinfinity} gives information on
the asymptotics of the solutions at $t=\infty$.
Before stating this, we recall that 
$w_i+w_{n-i}=0$, and $n+1$ is even, so it suffices to specify the behaviour of
$w_0,\dots,w_{\frac{n-1}2}$.  In the course of the proof, certain linear combinations of the
$w_i$ arise naturally, and we state the result using these.

\begin{theorem}\label{asymptatinfinity}  Let $w$ be 
the solution of (\ref{ost}) 
constructed above from $s_1,\dots,s_{\frac12(n+1) }$. Then,
for $1\le p\le \frac{n+1}2$, we have
\begin{align*}
w_0 \sin p \tfrac{\pi}{n+1} +
w_1 &\sin 3p \tfrac{\pi}{n+1} +
\cdots +
w_{\frac{n-1}2} \sin np \tfrac{\pi}{n+1} 
\\
&=
-\tfrac{n+1}8
\
s_p  
\
(\pi L_p x)^{-\frac12} e^{-2L_p x}
+ O(x^{-\frac32} e^{-2L_px })
\end{align*}
as $x\to\infty$.  
\end{theorem}

\begin{proof}  Using (\ref{explicitG}),
the right hand side of 
formula (\ref{asymptYatinfinity}) becomes
\[
I+
\tfrac{1}{2\pi\ii}
\sum_{i\ne j}  s_{i,j} 
\int_{0}^\infty  
e^{-x(l+\frac1l )  L_{\vert j-i\vert}
}  \tfrac{dl}l\ 
e_{i,j}
+ (O_{i,j}).
\]
With the notation of section \ref{6} this can be written more compactly as
\[
I
+\tfrac{1}{2\pi\ii}
\sum_{p=1}^n  s_p 
\,
\int_{0}^\infty  
e^{-x(l+\frac1l )  L_p
}  \tfrac{dl}l
\ 
\hat\Pi  \, \vphantom{\hat\Pi}^p
+ 
\sum_{p=1}^n  
\,
O(e^{ -4 L_p x }) \hat\Pi  \, \vphantom{\hat\Pi}^p,
\]
where, as in formula (\ref{pihat}),
\[
\hat\Pi 
=
\left(
\begin{smallmatrix}
  &  1  & & &\\
  &   & \ \ddots & &\\
  & & &  & 1\\
-1    & & &
\end{smallmatrix}
\right)
= \om^{-\frac12} d_{n+1}^{-\frac12} \Pi \, d_{n+1}^{\frac12}.
\]
From formula (\ref{tildeYlimitsz}),
the left hand side of (\ref{asymptYatinfinity}) can be replaced by
$
d_{n+1}^{\frac12} \Om e^{-2w } \Om^{-1} d_{n+1}^{-\frac12}.
$
We obtain
\[
d_{n+1}^{\frac12} \Om e^{-2w } \Om^{-1} d_{n+1}^{-\frac12}
=
I
+\tfrac{1}{2\pi\ii}
\sum_{p=1}^n  s_p 
\,
\int_{0}^\infty  
e^{-x(l+\frac1l )  L_p
}  \tfrac{dl}l
\ 
\hat\Pi  \, \vphantom{\hat\Pi}^p
+ 
\sum_{p=1}^n  
\,
O_p \hat\Pi  \, \vphantom{\hat\Pi}^p,
\]
where $O_p= O(e^{ -4 L_p x })$. 
Now, Laplace's method gives  $\frac1{2\pi}\int_0^\infty  e^{-x(\frac1l+l)}  \frac{dl}l
=   \tfrac12 (\pi x)^{-\frac12} \, e^{-2x} + O(x^{-\frac32} e^{-2x})$ as $x\to\infty$. 
Taking logarithms, we obtain
\[
d_{n+1}^{\frac12} \Om (-2w) \Om^{-1} d_{n+1}^{-\frac12}
=
\tfrac 1 \ii
\sum_{p=1}^n  s_p 
\,
\tfrac12 (\pi L_p x)^{-\frac12} e^{-2L_p x}
\ 
\hat\Pi  \, \vphantom{\hat\Pi}^p
+
\sum_{p=1}^n  
\,
O_p^{\text{Lap}} \hat\Pi  \, \vphantom{\hat\Pi}^p,
\]
where
$O_p^{\text{Lap}}=
O(x^{-\frac32} e^{ -2 L_p x })$.
We claim next that
\begin{equation}\label{hatwexpansion}
d_{n+1}^{\frac12} \Om (-2w) \Om^{-1} d_{n+1}^{-\frac12}
=
\hat w_0 + \hat w_1 \hat\Pi + \cdots + \hat w_n \hat\Pi  \, \vphantom{\hat\Pi}^n,
\end{equation}
where the $\hat w_p$ are given by 
\[
\hat w_p=
\tfrac4{n+1} \ii
\left[
w_0 \sin p \tfrac{\pi}{n+1} +
w_1 \sin 3p \tfrac{\pi}{n+1} +
\cdots +
w_{\frac{n-1}2} \sin np \tfrac{\pi}{n+1} 
\right].
\]
Comparison with the previous formula will then give
\begin{equation}\label{hatwasymp}
\hat w_p =
-\ii s_p
\,
\tfrac12 (\pi L_p x)^{-\frac12} e^{-2L_p x}
+  O_p^{\text{Lap}},
\end{equation}
which is the statement of the theorem.

To prove (\ref{hatwexpansion}), we note first that
$\Om (-2w) \Om^{-1}$ must be a linear combination of
$I,\Pi,\dots,\Pi^n$, because $\Om d_{n+1} \Om^{-1}=\Pi$ 
and any diagonal matrix must be 
a linear combination of
$I,d_{n+1},\dots,d_{n+1}^n$.

Let us write $-2w= v_0+ v_1 d_{n+1} + \cdots + v_n d_{n+1}^n$. 
As
$I,d_{n+1},\dots,d_{n+1}^n$
are the columns of $\Om$, this is equivalent to
\[
\bp
-2w_0
\\
-2w_1
\\
\vdots
\\
-2w_n
\ep
=
\Om
\bp
v_0
\\
v_1
\\
\vdots
\\
v_n
\ep.
\]
The complex conjugate of the formula
$\hat\Pi = \om^{-\frac12} d_{n+1}^{-\frac12} \Pi d_{n+1}^{\frac12}$ is
$\hat\Pi = \om^{\frac12} d_{n+1}^{\frac12} \Pi d_{n+1}^{-\frac12}$,
and we know that $\Pi=\Om d_{n+1} \Om^{-1}$, so
$d_{n+1}^{\frac12} \Om d_{n+1} \Om^{-1} d_{n+1}^{-\frac12}= 
\om^{-\frac12} \hat\Pi$.  
From $-2w= v_0+ v_1 d_{n+1} + \cdots + v_n d_{n+1}^n$
we deduce that
\begin{align*}
d_{n+1}^{\frac12} \Om (-2w) \Om^{-1} d_{n+1}^{-\frac12}
&= 
v_0+  \om^{-\frac12} v_1 \hat\Pi + \cdots + (\om^{-\frac12})^n v_n 
\hat\Pi  \, \vphantom{\hat\Pi}^n
\\
&=\hat w_0 + \hat w_1 \hat\Pi + \cdots + \hat w_n 
\hat\Pi  \, \vphantom{\hat\Pi}^n,
\end{align*}
where
\begin{equation}\label{hatw}
\bp
\hat w_0
\\
\hat w_1
\\
\vdots
\\
\hat w_n
\ep
=
d_{n+1}^{-\frac12}
\bp
v_0
\\
v_1
\\
\vdots
\\
v_n
\ep
=
d_{n+1}^{-\frac12}
\Om^{-1}
\bp
-2w_0
\\
-2w_1
\\
\vdots
\\
-2w_n
\ep.
\end{equation}

Let us compute this more explicitly. From (\ref{matrixC}) we have $\Om^{-1}= \frac1{n+1} \bar\Om$.  Using this in (\ref{hatw}), we see that $\hat w_0=0$ and,
for $1\le p \le \frac{n+1}2$,  $\hat w_p=$
\[
-\tfrac2{n+1}
\left[
(\om^{-\frac12 p}  -  \om^{\frac12 p}) w_0
+
(\om^{-\frac32 p}  -  \om^{\frac32 p}) w_1
+
\cdots
+
(\om^{-\frac{n}2 p}  -  \om^{\frac{n}2 p}) w_{\frac{n-1}2}
\right], 
\]
and this is 
$\frac4{n+1} \ii
(
w_0 \sin p \tfrac{\pi}{n+1} +
w_1 \sin 3p \tfrac{\pi}{n+1} +
\cdots +
w_{\frac{n-1}2} \sin np \tfrac{\pi}{n+1} 
)$,
as required.
\end{proof}

\section{Global solutions}\label{8}

Now we shall relate the local solutions at $t=0$ (section \ref{2}) to the 
local solutions at $t=\infty$ (section \ref{7}). We begin by summarizing both
constructions.

\subsection{Monodromy data for local solutions at $t=0$: summary}\label{8.1}\ 

In section \ref{2.3},  solutions $w=w(\vert t\vert)$ of (\ref{ost}) on intervals of the form $(0,\eps)$ were constructed from data
$(k_0,\dots,k_n)$,
$(c_0,\dots,c_n)$
with $c_i>0$ and $k_i>-1$. 
These solutions have the asymptotics
\[
w_i\sim -m_i \log\vert t\vert - \log \hat c_i
\]
as $t\to 0$, where 
\begin{align*}
m&=(m_0,\dots,m_n)
\\
\hat c&=(\hat c_0,\dots,\hat c_n)
\end{align*}
are equivalent to $(k_0,\dots,k_n)$,
$(c_0,\dots,c_n)$ through the formulae in Definition \ref{Nmchat}.
We have $\hat c_i>0$ and $m_{i-1}-m_i> -1$.  

\begin{definition}
We denote by $w^0_{m,\hat c}$ the solution corresponding to the pair $(m,\hat c)$.
The domain $(0,\eps)$ of $w^0_{m,\hat c}$ depends on $m$ and $\hat c$.
\end{definition}

Let
\[
\cA=\{ m\in \R^{n+1} \st m_{i-1}-m_i\ge -1, m_i+m_{n-i}=0 \}.
\]
This is a compact convex region of $\R^{n+1}$, of dimension $\tfrac12(n+1)$.
(We are continuing to
assume that $n+1$ is even in this section.)
As noted at the end of section \ref{2}, the condition $m\in\cA$  is a necessary and sufficient condition for the existence of local (radial) solutions near $t=0$ satisfying $w_i\sim -m_i \log|t|$.
Our solutions $w^0_{m,\hat c}$ are of this type, and in fact $m\in\mathring\cA$ (the interior of 
$\cA$, given by $m_{i-1}-m_i> -1$) as we are considering the generic case. 

The monodromy data of $w^0_{m,\hat c}$ consists of Stokes factors and connection matrices.  The ingredients of this monodromy data are real numbers $s_i,e_i$ (see 
Definition \ref{si} and Definition \ref{ei}). Writing
\begin{align*}
s&=(s_1,\dots,s_n)
\\
e&=(e_0,\dots,e_n)
\end{align*}
we can say that the \ll monodromy map\rr is the map
\[
\mu:(m,\hat c)\mapsto (s,e) = (\mu_1(m),\mu_2(m,\hat c))
\]
because the Stokes factors are given in Theorem \ref{siofost} 
in terms of $s$, and 
the connection matrices are given in 
Theorem \ref{explicitE1withe} and
Corollary \ref{eiofost} 
in terms of $s$ and $e$.

The map $\mu$ is injective. This follows from the injectivity of its first component $\mu_1$:

\begin{lemma}\label{bij}
The map $\mu_1: \mathring\cA \to \R^{n+1}$,
$m\mapsto s$, is injective. 
\end{lemma}

\begin{proof}
We have to prove that $s$ determines $m$. First, $s$ determines
the unordered set
$\{\om^{m_0+\frac n2}, \om^{m_1-1+\frac n2}, \dots, \om^{m_n-n+\frac n2} \}$
by Proposition \ref{charpoly} and
Theorem \ref{siofost}.
Multiplying by $\om^{-\frac n2}$ we obtain the unordered set
$\{ \om^{m_0^\pr}, \om^{m_1^\pr}, \dots, \om^{m_n^\pr} \}$.
By (\ref{mdashineq}) we have 
$m^\pr_i<m^\pr_{i-1}$ (for $i\in\Z$), and by 
(\ref{bdashineq}) we have $\vert m_i^\pr - m_j^\pr \vert < n+1$ 
(for $0\le i,j\le n$).  As
$\om=e^{2\pi\i/(n+1)}$,  these facts show that $s$ determines 
$m_n^\pr < \cdots < m_0^\pr$, hence $m$.
\end{proof}

Hence $\mu$ is bijective to its image. We conclude that, for the solutions $w^0_{m,\hat c}$, the monodromy data $s,e$ is equivalent to the data $m,\hat c$.

\subsection{Monodromy data for local solutions at $t=\infty$: summary}\label{8.2}\ 

In section \ref{7.3},  solutions $w=w(\vert t\vert)$ of (\ref{ost}) 
on intervals of the form $(R,\infty)$ were constructed from monodromy data 
$s \in \R^n$,
$e=(1,\dots,1)$.  (The real number $R$ depends on $s$.)
The corresponding
Stokes factors have the same shape as in Theorem \ref{siofost} (without
assuming that $s_i$ is the $i$-th symmetric function of powers of $\om$).
The connection matrix is  $E_1^{\id}=\tfrac1{n+1} C \Qi_{\frac n{n+1}}$.

\begin{definition}
We denote by $w^\infty_{s}$ the solution corresponding to $s\in\R^n$.
\end{definition}

\begin{lemma}\label{111}  Let $m\in \mathring\cA$.  
Let $s=\mu_1(m)$.  
Then there exists exactly one $\hat c$ with the
property $\mu(m,\hat c)=(s,(1,\dots,1))$.  We denote this by $\hat c^{\id}$.
\end{lemma}

\begin{proof}
By Corollary \ref{eiofost} the condition $\mu_2(m,\hat c)=(1,\dots,1)$
means
\[
\textstyle
(\hat c_i)^2=
\scriptstyle{(n+1)}^{m_{n-i} - m_i}
\frac
{
\Gad\left(
\frac{
m^\pr_{n-i} - m^\pr_{n-i+1} 
}{n+1}
\right)
}
{
\Gad\left(
\frac{
m^\pr_{i} - m^\pr_{i+1} 
}{n+1}
\right)
}
\frac
{
\Gad\left(
\frac{
m^\pr_{n-i} - m^\pr_{n-i+2} 
}{n+1}
\right)
}
{
\Gad\left(
\frac{
m^\pr_{i} - m^\pr_{i+2} 
}{n+1}
\right)
}
\cdots
\frac
{
\Gad\left(
\frac{
m^\pr_{n-i} - m^\pr_{n-i+n} 
}{n+1}
\right)
}
{
\Gad\left(
\frac{
m^\pr_{i} - m^\pr_{i+n} 
}{n+1}
\right)
}.
\]
This determines $\hat c_i$ as all $\hat c_i$ are positive.
\end{proof}

It follows that the monodromy data of $w^\infty_{s}$ coincides with 
the monodromy data of
$w^0_{m,\hat c^{\id}}$ if $m$ corresponds to $s$, i.e.\ $s=\mu_1(m)$.
We are going to prove that, when  $m\in \mathring\cA$, there exist global solutions which
coincide near $t=0$ with $w^0_{m,\hat c^{\id}}$ and near $t=\infty$ with $w^\infty_{s}$.

\subsection{Global solutions and their monodromy data}\label{8.3}\ 

When $m=0$ we have $s=0$, and
this data corresponds to the trivial solution $w=0$ of (\ref{ost}), which (obviously)
is a globally smooth solution.  
The following result, which makes use of most of
the calculations in this article, extends this to an open neighbourhood of the trivial solution:

\begin{theorem}\label{VanLem} 
There exists an open neighbourhood $V$ of $0$ in $\mu_1(\mathring \cA)$ such that,
if $s\in V$, then the local solution $w^\infty_{s}$ at $t=\infty$ is in fact smooth on the interval $0<\vert t\vert < \infty$, i.e.\  is a globally smooth solution.
\end{theorem}

\begin{proof} The strategy of the proof is the same as that in section 5 of \cite{GIL2}, using
the \ll Vanishing Lemma\rrr.  Namely, we shall show that 
the \ll homogeneous\rr Riemann-Hilbert problem, 
in which the condition  $Y\vert_{\ze=\infty}=I$ is replaced by the condition $Y\vert_{\ze=\infty}=0$,
has only the trivial solution $Y\equiv 0$.  It follows from this (Corollary 3.2 of \cite{FIKN06}) that the original problem is solvable.  

We shall work with a modified (but equivalent) Riemann-Hilbert problem on the contour consisting of
the two rays with arguments $\scriptstyle \frac{\pi}{2(n+1)}$,  $\pi+  \scriptstyle   \frac{\pi}{2(n+1)}$ (Figure \ref{RH4}).  
This contour divides $\C$ into
two half-planes. We denote the upper region by $\C_+$ and the lower region by $\C_-$.   Explicitly, 
\[
\C_+=\{ \zeta\in\C^\ast \st  \tfrac{\pi}{2(n+1)} <\text{arg}\zeta < \pi+\tfrac{\pi}{2(n+1)} \}.
\]
Our Riemann-Hilbert problem is motivated (in analogy with section \ref{7.2}) by considering first the holomorphic functions
$\fPsii_{\scriptstyle 3/2}$ on $\C_+$ and $\fPsii_{5/2}$ on $\C_-$.  By 
$\fPsii_{3/2},\fPsii_{5/2}$ 
we mean the functions defined on $\bar\C_+,\bar\C_-$
by the procedure of section \ref{7.2}, i.e.\ they are obtained
from $\Psii_{3/2}$, $\Psii_{5/2}$ by projecting to the reference sector
$[-\frac{\pi}{n+1},2\pi-\frac{\pi}{n+1})$.

\begin{figure}[h]
\begin{center}
\includegraphics[angle=270,origin=c,scale=0.4, trim= 70 200 100 150]{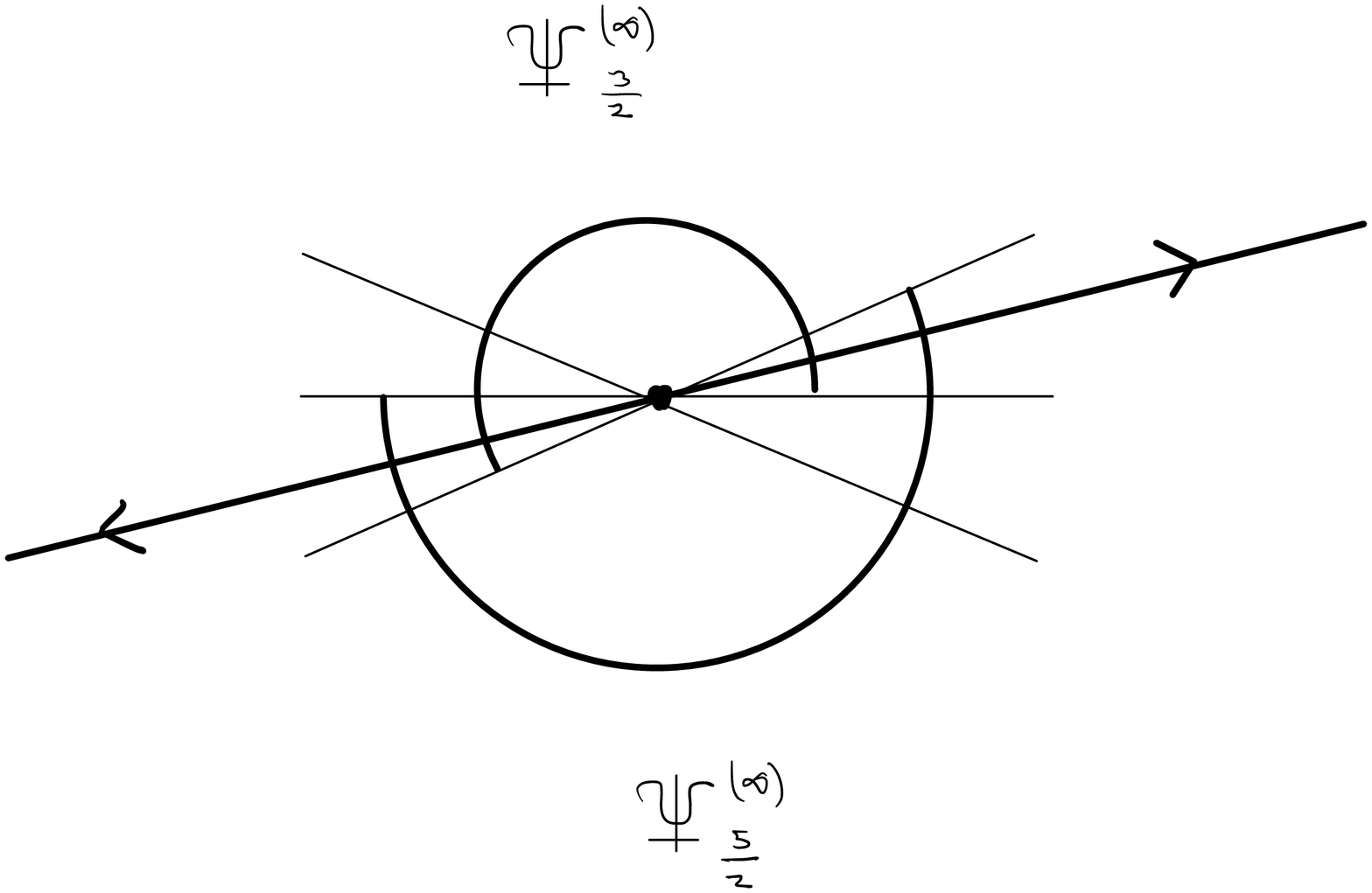}
\end{center}
\caption{$\C_+$ (above heavy line) and $\C_-$ (below heavy line).}\label{RH4}
\end{figure}

Let us recall that the sectors 
$\Omi_{3/2}$, $\Omi_{5/2}$ in $\tilde\C^\ast$
(which were used to define
$\Psii_{3/2}$, $\Psii_{5/2}$)
were defined, respectively, by the conditions 
$0<\text{arg}\zeta < \pi+\scriptstyle\frac{\pi}{n+1}$,
$\pi <\text{arg}\zeta < 2\pi+\scriptstyle\frac{\pi}{n+1}$.
These are indicated by the heavy circular segments in Figure \ref{RH4}.
It follows that $\fPsii_{3/2}(\ze)=\Psii_{3/2}(\ze)$ 
for $\ze\in\bar\C_+$. On the other hand $\fPsii_{5/2}(\ze)=\Psii_{5/2}(\ze)$
only when $\pi+   { \scriptstyle \frac{\pi}{2(n+1)}  } \le \arg\ze < 
2\pi- \scriptstyle\frac{\pi}{n+1})$;  for $ {-\scriptstyle\frac{\pi}{n+1} }  \le \arg\ze \le
\scriptstyle \frac{\pi}{2(n+1)} $ we have $\fPsii_{5/2}(\ze)=\Psii_{5/2}( e^{2\pi\i} \ze)$.

We need the jump matrices on the rays which separate $\C_+$ from $\C_-$.
For $\ze$ on the $\pi+\scriptstyle\frac{\pi}{2(n+1)}$ ray we have 
$\fPsii_{\scriptstyle\frac32}(\ze)=\Psii_{\scriptstyle\frac32}(\ze)$, 
$\fPsii_{\scriptstyle\frac52}(\ze)=\Psii_{\scriptstyle\frac52}(\ze)$,
so the jump matrix is 
$\Si_{3/2}$.
On the $\frac{\pi}{2(n+1)}$ ray we have 
$\fPsii_{3/2}(\ze)=\Psii_{3/2}(\ze)$,
but
$\fPsii_{5/2}(\ze)=\Psii_{5/2}(e^{2\pi\i}\ze)$,
which is equal to $\Psii_{1/2}(\ze)$ by formula 
(\ref{con-infinity}).  It follows that the jump matrix here is $\Si_{1/2}$.

Next, we replace the
$\fPsii_{k}$ by
$Y_k(\ze)= (e^w \Om^{-1})^{-1} \fPsii_k(\ze) e(\ze)^{-1}$
as in Definition \ref{Yk},
where 
$e(\ze)= e^{  \frac1\ze d_{n+1}^{-1}  }
e^{  x^2\ze d_{n+1}  }$.  Then the jump matrices become $e\Si_{3/2}e^{-1}$, 
$e\Si_{5/2}e^{-1}$.

Motivated by this, we can now formulate our Riemann-Hilbert problem:  to reconstruct (holomorphic) functions
$Y_{3/2},Y_{5/2}$ on $\bar\C_+,\bar\C_-$ such that the jump matrix on the 
$\scriptstyle \frac{\pi}{2(n+1)}$ ray is $e\Si_{3/2}e^{-1}$, and
the jump matrix on the  $\pi+ \scriptstyle   \frac{\pi}{2(n+1)}$ ray is $e\Si_{5/2}e^{-1}$.
 
In order to apply the theory of \cite{FIKN06} we must show that these jump matrices approach $I$ exponentially as $\ze\to\infty$ along either ray.  For this we need the following modification of Lemma 
\ref{Gonthetak}, which can be proved in exactly the same way:

\begin{lemma} For $\ze=re^{\i(\fthi_k + \al)}$, i.e.\ for $\ze$ in the $\fthi_k + \al$ ray
($\al\in\R$),
$e(\ze) \, e_{i,j} \, e(\ze)^{-1} =
e^{ (-rx^2 - \frac1r )\vert \om^j - \om^i\vert \cos\al
+ \i  (-rx^2 + \frac1r )\vert \om^j - \om^i\vert \sin\al  }
\, e_{i,j}$. 
\qed
\end{lemma}

Let us consider now the behaviour of $e\Si_{3/2}e^{-1}$ on the 
$\pi+\scriptstyle\frac{\pi}{2(n+1)}$ ray.  We have
\[
\Si_{\scriptstyle\frac32}=
\Qi_{\scriptstyle\frac32} \dots
\Qi_{ 1 +\scriptstyle\frac n{n+1}}
\Qi_{2}
\dots
\Qi_{ {\scriptstyle\frac{3}{2}} +\scriptstyle\frac{n}{n+1}}
\]
so it suffices show that
$e\Qi_{k}e^{-1}$ approaches $I$ exponentially 
on the $\pi+\scriptstyle\frac{\pi}{2(n+1)}$ ray
as $\ze\to\infty$,
for each $k\in \{ \frac{3}{2}, \frac{3}{2}+\frac1{n+1},\dots, \frac{3}{2}+\frac n{n+1} \}$.

Adjacent $\fthi_{k}$ rays are separated by $\scriptstyle\frac{\pi}{n+1}$,
and the $\pi+\scriptstyle\frac{\pi}{2(n+1)}$ ray lies strictly between the 
$\fthi_{ 1 +\scriptstyle\frac n{n+1}}$ and $\fthi_{2}$ rays. It follows that, if we write
$\pi+{\scriptstyle\frac{\pi}{2(n+1)} }=\fthi_k+\al_k$, then we have
$-\frac\pi2 < \al_k < \frac\pi2$
for each $k\in \{ \frac{3}{2}, \frac{3}{2}+\frac1{n+1},\dots, \frac{3}{2}+\frac n{n+1} \}$.
As all $\cos \al_k>0$ here, the lemma shows that all $e\Qi_{k}e^{-1}$ approach $I$ exponentially as $\ze\to\infty$.

A similar argument applies to the jump matrix on the $\scriptstyle\frac{\pi}{2(n+1)}$ ray.

As in the case of the Riemann-Hilbert problem in section \ref{7.3}, it follows that the Riemann-Hilbert problem
(with $Y\vert_{\ze=\infty}=I$)  is solvable when
$x$ is sufficiently large. 
The Vanishing Lemma argument will allow us to establish solvability, not just for large $x$, but for all $x\in(0,\infty)$. 

For this we consider the homogeneous Riemann-Hilbert problem, in which we require
$Y\vert_{\ze=\infty}=0$ (in contrast to the Riemann-Hilbert problem based on the functions $\fPsii$,
for which we had
$Y\vert_{\ze=\infty}=I$).

For brevity let us denote the $\scriptstyle\frac{\pi}{2(n+1)}$ ray by $\Ga_{1/2}$ and the
$\pi+\scriptstyle\frac{\pi}{2(n+1)}$ ray by $\Ga_{3/2}$, and the jump matrices
on these rays by $G_{1/2}$, $G_{3/2}$.  If $Y^0_{5/2}$, $Y^0_{3/2}$ solve
the homogeneous Riemann-Hilbert problem, we aim to find sufficient conditions which ensure
that $Y^0_{5/2}=0$, $Y^0_{3/2}=0$.

As in Proposition 5.1 of \cite{GIL2}, Cauchy's Theorem implies that

(a) $\displaystyle \int_{\Ga} Y^0_{5/2}(\ze) \ 
\overline{
Y^0_{3/2}(\bar\ze e^{2\pi\i/2n+2})
}
^{\, T}\  d\ze=0$

(b) $\displaystyle \int_{\Ga} Y^0_{3/2}(\ze) \ 
\overline{
Y^0_{5/2}(\bar\ze e^{2\pi\i/2n+2})
}
^{\, T}\  d\ze=0$

\no where $\Ga=\Ga_{3/2}\cup\Ga_{1/2}$. Let us add (a) and (b) together, and
use the fact that $Y^0_{5/2}=Y^0_{3/2}G_{3/2}$ on $\Ga_{3/2}$,
and  $Y^0_{3/2}=Y^0_{5/2}G_{1/2}$ on $\Ga_{1/2}$. From the resulting equation
we deduce that, if
$G_{3/2}(\ze)+\overline{
G_{3/2}(\bar\ze e^{2\pi\i/2n+2})
}
^{\, T}>0$ on $\Ga_{3/2}$, 
and
$G_{1/2}(\ze)+\overline{
G_{1/2}(\bar\ze e^{2\pi\i/2n+2})
}
^{\, T}>0$ on $\Ga_{1/2}$, then $Y^0_{5/2}$ and $Y^0_{3/2}$ must both
be identically zero. 

For any $s$, and for $x$ sufficiently large this criterion is satisfied, as $G_{1/2}$ and $G_{3/2}$
are then close to the identity, so we recover the fact stated earlier that the original
Riemann-Hilbert problem is solvable near $x=\infty$. For $s=0$ the criterion is satisfied (for all $x$), as then $G_{1/2}=I=G_{3/2}$. Let us now consider the behaviour of 
$H_{3/2}  = G_{3/2}(\ze)+\overline{
G_{3/2}(\bar\ze e^{2\pi\i/2n+2})
}
^{\, T}$ 
and
$H_{1/2} = G_{1/2}(\ze)+\overline{
G_{1/2}(\bar\ze e^{2\pi\i/2n+2})
}
^{\, T}$
when $s$ is close to zero.  A Hermitian matrix is positive definite if and only if 
its minors are positive, and (in our situation, by the formula for $G$ and the lemma above) the
minors are polynomials in variables
$\la_{i,j}s_{i,j}$ where $\vert\la_{i,j}\vert<1$ for all $i,j$.  

If all $\la_{i,j}s_{i,j}$ are replaced by $1$ (i.e.\ $e$ is replaced by $I$), then
we obtain the matrices $\Si_{3/2}+\overline{\Si_{3/2}}^T$, $\Si_{5/2}+\overline{\Si_{5/2}}^T$.
These are positive definite when $s=0$, hence for all $s$ in an open neighbourhood $V$ of $s=0$.
As $\vert\la_{i,j}\vert<1$, 
the original matrices $H_{3/2}$ and $H_{1/2}$ are also positive definite
for all $s\in V$.  This completes the proof of Theorem \ref{VanLem}.
\end{proof}

\begin{remark} We chose the rays with arguments $\scriptstyle \frac{\pi}{2(n+1)}$,  $\pi+  \scriptstyle   \frac{\pi}{2(n+1)}$ in the above proof for compatibility with the proof in \cite{GIL2} in the case $n=3$.  Examination of the proof given here shows that in fact any two (collinear) rays could have been used, providing that they do not coincide with singular directions.  
\qed
\end{remark}

\begin{corollary}\label{VanLemCor} 
For $s\in V$ we have $w^\infty_{s}\vert_{(0,\eps)}=w^0_{m,\hat c^{\id}}$.
\end{corollary}

\begin{proof} For any $x\in(0,\eps)$, both $w^\infty_{s}\vert_{(0,\eps)}$ and 
$w^0_{m,\hat c^{\id}}$ are solutions to the same Riemann-Hilbert problem. Hence
they coincide. 
\end{proof}

In the case $n=1$ it is easy to verify that
$V= \mu_1(\mathring \cA)$.
In \cite{GIL2}, in the case $n=3$, the precise region $V\sub \mu_1(\mathring \cA)$ of positivity was calculated, and is a proper subset of $\mu_1(\mathring \cA)$.
For the purposes of this article it will be sufficient to know that $V$ is a non-empty 
open set, as we shall now appeal to p.d.e.\ theory to deduce that every point of 
$\mu_1(\mathring \cA)$ corresponds to a global solution.

\begin{theorem}\label{pde}   
There is a bijection between  solutions $w$ of the $tt^*$-Toda equations (\ref{ost})
on $\C^\ast$ such that
\[
\lim_{t\to 0} \frac{   w(\vert t\vert)   }{  \log \vert t\vert  } = -m,
\quad
\lim_{t\to\infty} w(\vert t\vert) = 0
\]
and points $m\in\cA$.
\end{theorem}

The proof will be given in an appendix (section \ref{10}).

Restricting attention to $\mathring \cA$, it follows from the construction of section \ref{2} that the solution corresponding to $m$ in
Theorem  \ref{pde} must have Stokes data $s=\mu_1(m)$, but we do not yet know the
connection matrix data.  Theorem \ref{VanLem} allows us to obtain this:

\begin{corollary} Let $m\in\mathring \cA$, and let 
$w_m^{\text{\em pde}}$ be the corresponding solution given by Theorem \ref{pde}.  Then $w_m^{\text{\em pde}}$ must have $e=(1,\dots,1)$.
\end{corollary}

\begin{proof} For $s\in V$ we know that $w^0_{m,\hat c^{\id}}$ is a
global solution of the type in Theorem \ref{pde}, where $s=\mu_1(m)$. By the bijectiveness property it
coincides with $w_m^{\text{pde}}$. Hence $w_m^{\text{pde}}$ has $e=(1,\dots,1)$ for all $m\in V$. But $e$ depends analytically on $m$, so we must have $e=(1,\dots,1)$ for all $m\in\mathring \cA$.
\end{proof}

Strictly speaking, the corollary (and its proof) should have been phrased in terms of
the connection matrix $E_1$ rather than $e$, because $E_1$ exists for arbitrary solutions of (\ref{ost}), whereas we have defined $e$ only for solutions of the form $w_{m,\hat c}$. 
However the same argument shows that $E_1=E_1^{\id}$ for $w_m^{\text{pde}}$.

We conclude that the solutions $w_m^{\text{pde}}$ with $m\in\mathring \cA$
are precisely our solutions $w^0_{m,\hat c^{\id}}$, hence the latter are indeed global
solutions.

\begin{remark}\label{pdeboundary}
The \ll boundary conditions\rr at $t=0$ and $t=\infty$ in Theorem \ref{pde} are automatically satisfied if $w$ is a radial solution of (\ref{ost}). We shall give a proof of this fact in section \ref{10.2}.
\qed
\end{remark}

\subsection{Summary of results}\label{8.4}\ 

From the point of view of the family of local solutions  $w^0_{m,\hat c}$ constructed in
section \ref{2}, we have obtained the following characterizations of global solutions:

\begin{corollary}\label{global} Let $m\in\mathring \cA$.
The following conditions are equivalent:

(1) $w^0_{m,\hat c}$ is a globally smooth solution on $(0,\infty)$,

(2) $\hat c=\hat c^{\id}$.

(3) The connection matrix $E_1$ satisfies $E_1=E_1^{\id}$, i.e.\ $e=(1,\dots,1)$.

(4) The connection matrix $D_1$ satisfies 
$\Om D_1 t^m  =   \De  (\overline{ \Om D_1 t^m }) \De $, i.e.\ $\Om D_1 t^m$
belongs to the group
$\SL^\De_{n+1}\R$.
\end{corollary}

\begin{proof}
The equivalence of (1),(2),(3) follows immediately from the discussion above.
The equivalence of (3) and (4) is a consequence of formula (\ref{E1intermsofD1})
in Theorem \ref{EintermsofD}, i.e.\
\[
E_1=\tfrac1{n+1} D_1 t^m\, \De\, (\overline{ D_1 t^m })^{-1} 
(\bar Q^{(0)}_{\frac{n}{n+1}})^{-1}
d_{n+1}^{-1}\,  C.
\]
Namely, as in the proof of 
Theorem \ref{explicitE1}, this is equivalent to 
$E_1=D_1 t^m\, \De\, (\overline{ D_1 t^m })^{-1} 
d_{n+1}^{-1}\, E_1^{\id}$, hence (3) is equivalent to
\[
D_1 t^m\, \De\, (\overline{ D_1 t^m })^{-1} 
d_{n+1}^{-1} = I.
\]
Multiplying by $\Om$ we obtain 
$\Om D_1 t^m\, \De = \Om  d_{n+1}  (\overline{ D_1 t^m })$.
As $\Om  d_{n+1} \bar\Om^{-1}=\De$ (cf.\ the formulae in
Appendix A of \cite{GIL2}), we obtain (4).
\end{proof}

The significance of $\Om D_k t^m$ is that it is the connection
matrix which relates the \ll undiagonalized form\rr $\Phiz_k \Om^{-1}$
of $\Phiz_k$ (see section \ref{stokesforhatom-zero}) to
$\Phii$ (see section \ref{stokesforhatom-infinity}), namely we have the relation
\[
\Phii = (  \Phiz_k \Om^{-1} ) ( \Om D_k t^m ).
\]
The Stokes factors based on $\Phiz_k \Om^{-1}$ are $\Om \Qz_k \Om^{-1}$. 
Using these, a further characterization can be obtained:

\begin{corollary}\label{global2}
$w^0_{m,\hat c}$ is a globally smooth solution on $(0,\infty)$ if and only if
all Stokes factors $\Om \Qz_k \Om^{-1}$ and all connection matrices
$\Om D_k t^m$ lie in
the group $\SL^\De_{n+1}\R$.
\end{corollary}

\begin{proof} From (4) of Corollary \ref{global}, the condition
$\Om D_1 t^m\in \SL^\De_{n+1}\R$ alone implies that 
$w^0_{m,\hat c}$ is a global solution.  Conversely, 
we claim first that $\Om \Qz_k \Om^{-1}\in \SL^\De_{n+1}\R$ for
all $m\in\mathring \cA$. 
This means
$\De \overline{\Om \Qz_k \Om^{-1}} \De = \Om \Qz_k \Om^{-1}$.  This
condition  is equivalent to 
$d_{n+1} \overline{\Qz_k} d_{n+1}^{-1}=\Qz_k$, as $\Om^{-1}\De\bar\Om=d_{n+1}$
(see the end of the proof of Corollary \ref{global}). 
By (\ref{tQzandQz}), this is equivalent to the reality of the Stokes factors $\tQz_k$, 
which we have already established in section \ref{6.2}. This completes the proof
of the claim.  In view of the formula (\ref{Dfromdef}), which expresses all $D_k$
in terms of $D_1$ and the $\Qz_k$, 
it follows that $\Om D_k t^m\in \SL^\De_{n+1}\R$ for all $k$.
This completes the proof.
\end{proof}

Corollary \ref{global2} is quite natural, as 
$\hat\al\left\vert_{\vert\la\vert=1}\right.$ is
an $\sl^\De_{n+1}\R$-valued connection form. However, we do not know a direct
proof of the above characterization which exploits this fact. 

Collecting together the asymptotic results at $t=0$ 
(Proposition \ref{localsol},  Corollary \ref{eiofost}) and at $t=\infty$
(Theorem \ref{asymptatinfinity}), we can now
give the asymptotic data of our global solutions:

\begin{corollary}\label{final} Let $m\in \mathring \cA$. Let $w$ denote the corresponding global solution. Then:

(1) The asymptotic data at $t=0$  of $w$ is given by 
\[
w_i= -m_i \log|t| - \tfrac12\log (\hat c^{\id}_i)^2 + o(1),
\] 
where $(\hat c^{\id}_i)^2=$
\[
\textstyle
{(n+1)}^{m_{n-i} - m_i}
\frac
{
\Gad\left(
\frac{
m^\pr_{n-i} - m^\pr_{n-i+1} 
}{n+1}
\right)
}
{
\Gad\left(
\frac{
m^\pr_{i} - m^\pr_{i+1} 
}{n+1}
\right)
}
\frac
{
\Gad\left(
\frac{
m^\pr_{n-i} - m^\pr_{n-i+2} 
}{n+1}
\right)
}
{
\Gad\left(
\frac{
m^\pr_{i} - m^\pr_{i+2} 
}{n+1}
\right)
}
\cdots
\frac
{
\Gad\left(
\frac{
m^\pr_{n-i} - m^\pr_{n-i+n} 
}{n+1}
\right)
}
{
\Gad\left(
\frac{
m^\pr_{i} - m^\pr_{i+n} 
}{n+1}
\right)
}.
\]

(2) The asymptotic data at $t=\infty$  of $w$ may be expressed as
\begin{align*}
w_0 \sin p \tfrac{\pi}{n+1} +
w_1 &\sin 3p \tfrac{\pi}{n+1} +
\cdots +
w_{\frac{n-1}2} \sin np \tfrac{\pi}{n+1} 
\\
&=
-\tfrac{n+1}8
\
s_p  
\
(\pi L_p x)^{-\frac12} e^{-2L_p x}
+ O(x^{-\frac32} e^{-2L_px })
\end{align*}
where $L_p= 2\sin \tfrac {p}{n+1} \pi$. Here $p$ ranges over $1,2,\dots,\frac{n+1}2$,
so that we have $\frac{n+1}2$ linear equations for 
$w_0,\dots,w_{\frac{n-1}2}$.
\qed
\end{corollary}

These formulae constitute a solution of the \ll connection problem\rrr, i.e.\ they give a precise relation between the asymptotics at zero and the asymptotics at infinity.  

\section{Appendix: modifications when $n+1$ is odd}\label{9}

Our main results are stated in the introduction for arbitrary $n\ge 1$, but for readability we have taken $n+1$ to be even in most of the calculations in sections \ref{4}-\ref{8}. 
Here we give the modifications needed when $n+1$ is odd (the \ll odd case\rrr).  We write $n+1=2d+1$, with $d\in\N$. 

\no{\em Section \ref{4}: }

The new features of the case $n+1=2d+1$ begin with the treatment of the Stokes data. As initial sector
at $\ze=0$ we take
\[
\Omz_1=\{ \zeta\in\C^\ast \ \vert \ -\tfrac{\pi}{2}-\tfrac{\pi}{2(n+1)}<\text{arg}\zeta <\tfrac\pi2 + \tfrac{\pi}{2(n+1)}\},
\]
and define 
\[
\Omz_{k+\scriptstyle\frac{1}{n+1}}= e^{-\frac\pi{n+1}\i} \Omz_k,
\quad
\thz_k=-\tfrac{1}{2(n+1)}\pi - (k-1)\pi
\]
for $k\in\tfrac1{n+1}\Z$.
As in the even case,  $\thz_k$ 
bisects $\Omz_k \cap \Omz_{k+\scriptstyle\frac{1}{n+1}}$.

The formal solution $\Psiz_f$ is defined as in the even case, but we define the tilde version 
by $\tPsiz_f(\ze)=\Psiz_f(\ze) (d_{n+1})^{-d}$.
The symmetries of the connection form $\hat\al$ are as stated in section
\ref{3.1}. The resulting symmetries of $\tPsiz_f$ and $\tPsiz_k$ are stated in \cite{GH1}; the symmetries of
$\Psiz_f$ and $\Psiz_k$ follow from this. 
At $\ze=\infty$ we define
$
\Omi_k
=\overline{\Omz_k},
$
as in the even case. 
We define $\tPsii_f(\ze)=\Psii_f(\ze) (d_{n+1})^{d}$.

The Stokes factors in the odd case are defined as in the even case, and they are indexed by
the singular directions $\thz_k$.  
Formulae (\ref{Psiz-mon})-(\ref{tQzandtQi}) continue to hold in the odd case,
except for (\ref{tQzandQz}) and (\ref{tQiandQi}), which must be replaced by
\[
\tQz_k= (d_{n+1})^{d} \Qz_k (d_{n+1})^{-d},\quad
\tQi_k= (d_{n+1})^{-d} \Qi_k (d_{n+1})^{d}.
\]
The cyclic symmetry formula (\ref{Qz-cyc}) 
implies that $(\Mz)^{n+1}= \Sz_1 \Sz_2$, where
$\Mz=\Qz_1\Qz_{1+\scriptstyle\frac1{n+1}}\Pi$, as in the even case.  However the
tilde version is simpler in the odd case: we have $(\tMz)^{n+1}= \tSz_1 \tSz_2$, where
$\tMz=\tQz_1\tQz_{1+\scriptstyle\frac1{n+1}}
\Pi=\om^d (d_{n+1})^d \Mz (d_{n+1})^{-d}$.  (Recall that,  in the even case,
$(\tMz)^{n+1}= -\tSz_1 \tSz_2$, where
$\tMz=\tQz_1\tQz_{1+\scriptstyle\frac1{n+1}}\hat\Pi
=
\om^{-\frac12} d_{n+1}^{-\frac12} \Mz d_{n+1}^{\frac12}$.)

For the connection form $\hat\om$, the Stokes sectors and Stokes factors at $\ze=0$ 
are the same as for the connection form $\hat\al$.

\no{\em Section \ref{5}: }

The connection matrices $E_k,D_k$ are defined as in the even case, but their tilde versions are
$\tilde E_k = (d_{n+1})^{d} E_k (d_{n+1})^{d}$,
$\tilde D_k = (d_{n+1})^{d} D_k$.
The formulae
(\ref{Efromdef})-(\ref{D-cyc})
continue to hold in the odd case.

In the odd case, the first statement of Theorem \ref{EintermsofD} becomes
$E_k=\tfrac1{n+1} D_kt^m\, \De\, \bar t^{-m}\bar D_{2-k}^{-1}\, d_{n+1}^{-1}\, C$,
and this gives immediately the following (simpler) version of (\ref{E1intermsofD1}):
\begin{equation*}
E_1=\tfrac1{n+1} D_1 t^m\, \De\, (\overline{ D_1 t^m })^{-1} 
d_{n+1}^{-1}\,  C.
\end{equation*}

\no{\em Section \ref{6}: }

In the even case, the shape of the Stokes factors $\tQz_k$ is specified by the set $\cRz_k$
in Proposition \ref{Sfactors2} (Proposition 3.4 of \cite{GH1}). In the odd case, it
is specified by the set $\cRz_k$ in Proposition 3.9 of \cite{GH1}.

Based on this, and the symmetries of the $\tQz_k$, we introduce the notation
\begin{equation*}
\tQz_k= I + \sum_{(i,j)\in\cRz_k} s_{i,j}  \, e_{i,j},
\end{equation*}
where $e_{i,j}= E_{i,j}$ if $0\le i<j\le n$ with $j-i$ even, or if $0\le j<i\le n$ with $i-j$ odd.
The remaining cases are determined by $e_{i,j}=-e_{j,i}$. These definitions
ensure that Proposition \ref{qij} continues to hold. 

We define $s_1,\dots,s_n$ as in Definition \ref{si}, and put $s_0=s_{n+1}=1$.  Then the
$s_i$ are all real and they satisfy $s_i=s_{n+1-i}$.  The essential Stokes
parameters are $s_1,\dots,s_d$. 
The analogue of Proposition \ref{charpoly} is that
the  (monic) characteristic polynomial of 
$\tMz$ is
\[
\textstyle
\sum_{i=0}^{n+1} (-1)^i s_i \mu^{n+1-i}.
\]
Then 
$s_i$ is the $i$-th symmetric function of
$\om^{m_0+\frac n2}, \om^{m_1-1+\frac n2}, \dots, \om^{m_n-n+\frac n2}$,
as in the even case.

\begin{example}  In analogy with Example \ref{2x2}, we present here the Stokes data
for $n+1=3$, the simplest odd case.
From Proposition 3.9 of \cite{GH1}
we have $\cRz_1=\{ (1,0) \}$, $\cRz_{\frac43}=\{ (1,2) \}$. 
There is only one Stokes parameter $s=s_1$, and
\[
\tQz_{1}=
\bp
1 & 0 & 0
\\
s & 1 &0
\\
0 & 0& 1
\ep,
\quad
\tQz_{\scriptstyle\frac43}=
\bp
1 &  0& 0
\\
0 & 1 & -s
 \\ 
 0& 0& 1
\ep.
\]
We have 
\[
\tMz= \tQz_1 \tQz_{\scriptstyle\frac43} \Pi =
\bp
0 & 1 &0
\\
-s & s & 1
\\
1 & 0& 0
\ep,
\]
with characteristic polynomial is $\mu^3 -s\mu^2 +s\mu-1$.
Theorem \ref{siofost} gives $s=1+ 2\cos \frac{2\pi}{3} (m_0+1)$.
We have $m_0=\frac{k_1-k_0}{k_0+2k_1+3}$ (from Definition \ref{Nmchat}). 
\qed
\end{example}

The computation of the connection matrices $D_k,E_k$ in the odd case can be carried out as
in the even case, but, to obtain the asymptotics of solutions near zero it is easier to observe that a solution $w=(w_0,\dots,w_{2d})$ with $n+1=2d+1$ gives rise to a
solution $w^\sharp=(w_0,\dots,w_{2d},-w_{2d},\dots,-w_0)$ with $n+1=4d+2$.  Then the
asymptotics of $w$ near zero follow immediately from the   
known asymptotics of $w^\sharp$ (Corollary \ref{final}).

\no{\em Section \ref{7}: }

The Riemann-Hilbert problem can be set up as in the even case, using the rays $\thz_k,\thi_k$
and the Stokes parameters $s_i$.  The analogue of formula (\ref{Zdef}) is
$\Psii_k=\Psiz_{ 2-k}Z_{ 2 - k}$.  The analogue of formula (\ref{RHcond}) is
\[
\Qi_k= \left( \tfrac1{n+1} C \right)^{-1} 
\left(
\Qz_{\scriptstyle\frac {2n+1}{n+1} -k}
\right)^{-1}
\tfrac1{n+1} C.  
\]
Then the analogue of Lemma \ref{globalEodd}, which is proved in the same way, is:

\begin{lemma}\label{globalEodd}  Let $E_1=\tfrac1{n+1} C$.
Then $Z_k=\tfrac1{n+1} C$ for all $k$.
\end{lemma}

The analogue of Theorem \ref{atinfinity}, which produces solutions near infinity, is:

\begin{theorem}\label{atinfinityodd} Let $s_1,\dots,s_d$ be real numbers, 
and let the matrices $\tQi_k$ be defined in terms of $s_1,\dots,s_d$ as
above.  Then there is a unique solution $w$ of (\ref{ost}) on an interval
$(R,\infty)$, where $R$ depends on  $s_1,\dots,s_d$, such
that the associated monodromy data is given by the 
$\Qi_k$ (i.e.\  $(d_{n+1})^d \tQi_k (d_{n+1})^{-d}$)
and $E_1^{\id}=\tfrac1{n+1} C$.
\end{theorem}

Again, to obtain the asymptotics of solutions near infinity, it is easiest to observe that a solution $w=(w_0,\dots,w_{2d})$ with $n+1=2d+1$ gives rise to a
solution $w^\sharp=(w_0,\dots,w_{2d},-w_{2d},\dots,-w_0)$ with $n+1=4d+2$.
From Theorem \ref{siofost} it follows that the Stokes parameters $s_i^\sharp$ for $w^\sharp$ are given in terms of 
the Stokes parameters $s_i$ for $w$ by:
\[
s_i^\sharp=
\begin{cases}
0 \quad\ \,  \text{if $i$ is odd}
\\
s_{\frac i2} \quad \text{if $i$ is even}
\end{cases}
\]
(i.e.\ from the expression for $s^\sharp_i$ as an elementary symmetric function, 
in terms of $m^\sharp=(m_0,\dots,m_{2d},-m_{2d},\dots,-m_0)$).
Using this, and the known asymptotics of $w^\sharp$ near infinity (Corollary \ref{final}), we obtain 
the asymptotics of $w$ as stated in the introduction.

\no{\em Section \ref{8}: }

As the notation  suggests, it is the value  $E_1=E_1^{\id}$ which is taken by global solutions.  The arguments concerning global solutions go through as in the even case.

\section{Appendix: p.d.e.\ theory}\label{10}

In this appendix, for $a,b>0$ and $n\ge 2$, we prove existence and uniqueness results for
solutions of the system
\begin{equation}\label{A.1}
\begin{cases}
\Delta u_1=e^{au_1}-e^{u_2-u_1}\doteqdot f_1(u_1,\dots,u_n),
\\
\hspace{2em}\cdots
\\
\Delta u_i=e^{u_i-u_{i-1}}-e^{u_{i+1}-u_i}\doteqdot f_i(u_1,\dots,u_n),
\\
\hspace{2em}\cdots
\\
\Delta u_n=e^{u_n-u_{n-1}}-e^{-bu_n}\doteqdot f_n(u_1,\dots,u_n),
\end{cases}
\end{equation}
where $u_i=u_i(x)$, $1\le i\le n$, and $x\in \mathbb{R}^2\setminus\{0\}$.
We impose the \ll boundary conditions\rr 
\begin{equation}\label{A.2}
    \lim\limits_{|x|\to0}\dfrac{u_i(x)}{\log|x|}=\gamma_i\mbox{ and }\lim\limits_{|x|\to +\infty}u_i(x)=0
\end{equation}
at $0$ and $\infty$, and we assume that
$\gamma_i$ satisfies
\begin{equation}\label{A.3}
-2\leq a\gamma_1,\quad -2\leq\gamma_{i+1}-\gamma_i  \ (1\leq i\leq n-1),
\quad b\,\gamma_n\leq2.
\end{equation}
It is a consequence of uniqueness that $u_i$ is necessarily radial. 

The tt*-Toda equations --- the system (\ref{ost}) --- are a special case of the system (\ref{A.1}). 
Namely, if $n+1$ in (\ref{ost}) is even, then the resulting equations for
$w_0,\dots,w_{\frac{n-1}{2}}$ have the form of (\ref{A.1}) with $a=b=2$;
if $n+1$ in (\ref{ost}) is odd, then the resulting equations for
$w_0,\dots,w_{\frac{n-2}{2}}$ have the form of (\ref{A.1}) with $a=2,b=1$.
Then the results from p.d.e.\ theory needed in section \ref{8} are
consequences of the results proved in this section:
Theorem \ref{Atheorem1} below implies Theorem \ref{pde}, and 
Theorem 
\ref{Atheorem2} below justifies the assertion in Remark \ref{pdeboundary} concerning the boundary conditions.
The case $n=2$ of (\ref{A.1}) was studied in detail by elementary methods in our previous articles \cite{GuLi14},\cite{GIL1},\cite{GIL2}. Here we shall give analogous results for general $n$, but using methods more appropriate for the general case.  

\subsection{Existence and uniqueness}\label{10.1}\ 

The purpose of this subsection is to prove Theorem \ref{Atheorem1} below, i.e.\ 
the existence and uniqueness of solutions of (\ref{A.1}) subject to
(\ref{A.2}) and (\ref{A.3}).

\begin{remark}
    Suppose $u(x)=(u_1(x),\dots,u_n(x))$ is a solution of (\ref{A.1}) with $\lim\limits_{|x|\to +\infty}u_i(x)=0$. Then, near $\infty,$ (\ref{A.1}) can be written as 
\begin{equation*}
\begin{cases}
\Delta u_1=(a+1)u_1-u_2+O(|u|^2),\\
\hspace{2em}\cdots\\
\Delta u_i=2u_i-u_{i-1}-u_{i+1}+O(|u|^2),\\
\hspace{2em}\cdots\\
\Delta u_n=(b+1)u_n-u_{n-1}+O(|u|^2).
\end{cases}
\end{equation*}
Since the matrix of coefficients on the right hand side 
is positive definite,  standard estimates (a straightforward application of the maximum principle) show that $u_i(x)=O(e^{-\epsilon|x|})$ for some $\epsilon>0,$ and then $|\nabla u_i(x)|=O(e^{-\epsilon|x|})$ for $|x|$ large.
For the tt*-Toda equations, the Riemann-Hilbert method allows us to make this more precise (Theorem \ref{asymptatinfinity}).
\qed
\end{remark}

First we establish the monotonicity of solutions of (\ref{A.1})-(\ref{A.2}) with respect to 
$\gamma=(\gamma_1,\dots,\gamma_n)$.
\begin{lemma}\label{Alemma1}
    Let $u$ and $v$ be two solutions of (\ref{A.1})-(\ref{A.2})
    with $\gamma^{(1)}=(\gamma_1^{(1)},\dots,\gamma_n^{(1)})$ and $\gamma^{(2)}=(\gamma_1^{(2)},\dots,\gamma_n^{(2)})$. If $\gamma_i^{(1)}<\gamma_i^{(2)},1\leq i\leq n,$ then $u(x)\geq v(x)$ for all $x\in\mathbb{R}^2\setminus\{0\}$.
\end{lemma}
\begin{proof}
    We shall apply the maximum principle. We remark that, in the following argument, it is sufficient to assume that $u(x)$ is a supersolution, i.e., $\Delta u_i(x)\leq f_i(u)$.
    
\no{\em Step 1:} $v_1\leq u_1$. 
If not, $v_1>u_1$ at some point in $\mathbb{R}^2\setminus\{0\}$. By the assumption $\gamma^{(2)}_1<\gamma_1^{(1)},$ the boundary conditions imply that $(v_1-u_1)(x_1)=\max_{x\in\mathbb{R}^2\setminus\{0\}}(v_1-u_1)(x)>0$ for some $x_1\in\mathbb{R}^2\setminus\{0\}$. The maximum principle gives 
$$
e^{av_1(x_1)}-e^{au_1(x_1)}-(e^{(v_2-v_1)(x_1)}-e^{(u_2-u_1)(x_1)})\leq 0.
$$
Since $e^{av_1(x_1)}-e^{au_1(x_1)}>0,$ we obtain $e^{(v_2-v_1)(x_1)}-e^{(u_2-u_1)(x_1)}>0,$ hence 
$$
(v_2-u_2)(x_1)>(v_1-u_1)(x_1)>0.
$$
The assumption $\gamma^{(2)}_2<\gamma_2^{(1)}$ again implies that 
\[
(v_2-u_2)(x_2)=\max_{x\in\mathbb{R}^2\setminus\{0\}}(v_2-u_2)(x)>0
\]
for some $x_2\in\mathbb{R}^2\setminus\{0\}$. 
Repeating this process for $1\leq i\leq n-1,$ we find that there exists $x_i\in\mathbb{R}^2\setminus\{0\}$ such that 
$$
(v_i-u_i)(x_i)=\max_{x\in\mathbb{R}^2\setminus\{0\}}(v_i-u_i)(x)>0
$$
and applying the maximum principle to the $i$-th equation at $x_i$ yields 
$$
(v_{i+1}-u_{i+1})(x_i)>(v_i-u_i)(x_{i}).
$$ 
For the maximum at $x_{i+1}$, we have
$$
(v_{i+1}-u_{i+1})(x_{i+1})\geq (v_{i+1}-u_{i+1})(x_{i})>(v_{i}-u_{i})(x_{i}).
$$
At $i=n$, we have, for some $x_n\in\mathbb{R}^2\setminus\{0\}$,
$$
(v_{n}-u_{n})(x_{n})=\max_{x\in\mathbb{R}^{2}\setminus\{0\}}(v_{n}-u_{n})(x)>(v_{n-1}-u_{n-1})(x_{n-1})>0.
$$ 
Applying the maximum principle to the n-th equation at $x_{n}$ yields 
$$
e^{(v_{n}-v_{n-1})(x_{n})}-e^{(u_{n}-u_{n-1})(x_{n})}-(e^{-bv_{n}(x_{n})}-e^{-bu_{n}(x_{n})})\leq 0,
$$
which implies $e^{-bv_{n}(x_{n})}-e^{-bu_{n}(x_{n})}>0$, and then $v_{n}(x_{n})<u_{n}(x_{n})$. 
This contradicts $(v_{n}-u_{n})(x_{n})>0$.
Thus we have $v_{1}\geq u_{1}$. 

\no{\em Step 2:} Assume $v_{i-1}\leq u_{i-1}$. We want to prove that $v_{i}\leq u_{i}$. If not, $v_{i}>u_{i}$ at some point in $\mathbb{R}^{2}\setminus\{0\}$. By the assumption $\gamma_{i}^{(2)}>\gamma_{i}^{(1)}$, the boundary conditions imply that $(v_{i}-u_{i})(x_{i})=\max\limits_{x\in\mathbb{R}^{2}\setminus\{0\}}(v_{i}-u_{i})(x)>0$ for some $x_{i}\in\mathbb{R}^{2}\setminus\{0\}$. The maximum principle at $x_{i}$ yields 
$$
e^{(v_{i}-v_{i-1})(x_{i})}-e^{(u_{i}-u_{i-1})(x_{i})}-(e^{(v_{i+1}-v_{i})(x_{i})}-e^{(u_{i+1}-u_{i})(x_{i})})\leq 0.
$$ 
By the assumption $v_{i-1}\leq u_{i-1}$ and $(v_{i}-u_{i})(x_{i})>0$, we have $(v_{i+1}-v_{i})(x_{i})-(u_{i+1}-u_{i})(x_{i})>0$, i.e., $(v_{i+1}-u_{i+1})(x_{i})>(v_{i}-u_{i})(x_{i})$. The assumption  $\gamma_{i+1}^{(2)}>\gamma_{i+1}^{(1)}$ again implies that $(v_{i+1}-u_{i+1})(x_{i+1})=\max\limits_{x\in\mathbb{R}^{2}\setminus\{0\}}(v_{i+1}-u_{i+1})(x)>0$ for some $x_{i+1}\in\mathbb{R}^{2}\setminus\{0\}$. Thus repeating the argument of Step 1 gives a contradiction. Hence we have $v_{i}<u_{i}$. By induction, we have $v_{i}< u_{i}$ for all $i$. This proves the lemma.
\end{proof}

To solve (\ref{A.1}) we use the variational method. 
First, we study equation (\ref{A.1}) on the ball $B_{R}$ with centre $0$ and radius $R$. 
In the following, we say that the $\gamma_{i}$ satisfy (\ref{A.3}) 
$\textit{strictly}$ if 
$$
-2< a\gamma_1,\quad -2<\gamma_{i+1}-\gamma_i  \ (1\leq i\leq n-1),
\quad b\,\gamma_n<2.
$$ 
We fix a large $R>0$ and consider $(\ref{A.1})$ in the domain $B_{R}\setminus\{0\}$
with the boundary conditions
\begin{equation}\label{A.4}
    \lim_{r\to0}\dfrac{u_{i}(r)}{\log r}=\gamma_{i}\text{ and }u_{i}(x)=0,  \ x\in\partial B_{R}.
\end{equation}

\begin{lemma}\label{Alemma2}
    Suppose the $\gamma_{i}$ satisfy (\ref{A.3}) strictly. Then the system (\ref{A.1}) in 
    $B_{R}\setminus\{0\}$
     has a unique solution $u=(u_{1},...,u_{n})$ satisfying (\ref{A.4}).
\end{lemma} 

\begin{proof}
    Let $w_{i}(x)=u_{i}(x)-\gamma_{i}\log\frac{|x|}{R}$. We obtain a system for $w_{i}$ without singularities at the origin by considering
\begin{equation*}
\begin{cases}
    \Delta w_{1}=e^{a(w_{1}+\gamma_{1}\log\frac{|x|}{R})}-e^{w_{2}-w_{1}+(\gamma_{2}-\gamma_{1})\log\frac{|x|}{R}}\\
    \hspace{2em}\cdots\cdots\\
    \Delta w_{i}=e^{w_{i}-w_{i-1}+(\gamma_{i}-\gamma_{i-1})\log\frac{|x|}{R}}-e^{w_{i+1}-w_{i}+(\gamma_{i+1}-\gamma_{i})\log\frac{|x|}{R}}\\
    \hspace{2em}\cdots\cdots\\
    \Delta w_{n}=e^{w_{n}-w_{n-1}+(\gamma_{n}-\gamma_{n-1})\log\frac{|x|}{R}}-e^{-b(w_{n}+\gamma_{n}\log\frac{|x|}{R})}
\end{cases}
\end{equation*}
with $w|_{\partial B_{R}}=0$. We establish existence by the variational  method. 
Consider the nonlinear functional 
 \begin{equation*}
 \begin{aligned}
 I(w)=&\tfrac{1}{2}\sum_{i=1}^{n}\int_{B_{R}}|\nabla w_{i}|^{2}
 +
 \tfrac{1}{a}\int_{B_{R}}e^{a(w_{1}+\gamma_{1}\log\frac{|x|}{R})}dx
 \\
 &
 \quad
 +
 \sum_{i=2}^{n}\int_{B_{R}}e^{w_{i}-w_{i-1}+(\gamma_{i}-\gamma_{i+1})\log\frac{|x|}{R}}dx
 +
 \tfrac{1}{b}\int_{B_{R}}e^{-b(w_{n}+\gamma_{n}\log\frac{|x|}{R})}dx
 \\
 &\doteqdot \tfrac{1}{2}\sum_{i=1}^{n}\int_{B_{R}}|\nabla w_{i}|^{2}dx+\sum_{i=1}^{n}\int_{B_{R}}g_{i}(x,w)dx
 \end{aligned}
 \end{equation*}
on the Sobolev space $H_{0}^{1}(B_{R})$.
    The functional $I$ is $C^{1}$ on $H_{0}^{1}(B_{R})$ if $\gamma_{1}>-\frac{2}{a}$, $\gamma_{n}<\frac{2}{b}$ and $\gamma_{i-1}-\gamma_{i}<2$, $i=2,...,n$ by the Moser-Trudinger theorem. 
    It is easy to see that any critical point of $I$ is a solution of the above system. The functional is coercive, i.e., $I(w)\geq\frac{1}{2}\sum_{i=1}^{n}\int_{B_{R}}|\nabla w_{i}|^{2}dx$. 
    Suppose that $w^{(m)}=(w_{1}^{(m)},...,w_{n}^{(m)})$ is a minimizing sequence for $I$, i.e., $I(w^{m})\to\inf_{H_{0}^{1}(B_{R})}I(w)$. Then $\{w^{(m)}\}$ is bounded on $H_{0}^{1}(B_{R})$. 
    Thus, up to a subsequence, $w^{(m)}\to w=(w_{1},...,w_{n})$ and $w^{(m)}(x)\to w(x)$, a.e. Thus have
$$
\int_{B_{R}}|\nabla w_{i}|^{2}dx\leq \lim\limits_{m\to +\infty}\int_{B_{R}}|\nabla w_{i}^{(m)}|^{2}dx,
$$ 
and by Fatous's Lemma 
$$
\int_{B_{R}}g_{i}(x,w)dx\leq \lim\limits_{m\to+\infty}\int_{B_{R}}g_{i}(x,w^{(m)})dx.
$$ 
Thus 
$I(w)\leq \lim\limits_{m\to+\inf}I(w^{(m)})$, i.e., $w=(w_{1},...,w_{n})$ 
is a minimizer for the functional $I$. Let $u_{i}(x)=w_{i}(x)+\gamma_{i}\log\frac{|x|}{R}$. Thus $u=(u_{1},...,u_{n})$ is a solution of the system (\ref{A.1}) on the ball $B_{R}$ with the boundary conditions (\ref{A.4}). 

Next we shall apply Lemma \ref{Alemma1} to prove the uniqueness of the solution. 
Suppose $u(x)$ is a solution of (\ref{A.1}) and (\ref{A.4}).  Choose a sequence of strict parameters
$\gamma^{(k)}=(\gamma_1^{(k)},\dots,\gamma_n^{(k)})$ 
 such that $\gamma_i<\gamma_i^{(k+1)}<\gamma^{(k)}_i,1\leq i\leq n$, and a sequence of solutions $v_i^{(k)}$ on $B_R$ satisfying (\ref{A.4}). The argument of Lemma \ref{Alemma1} implies that $v_i^{(k)}(x)<v_i^{(k+1)}(x)<u_i(x),1\leq i\leq n,$ for all $x\in B_R\setminus\{0\}$. Let $v_i(x)=\lim\limits_{k\to\infty}v_i^{(k)}(x),x\in B_R\setminus\{0\}$. 
 Then it is easy to see that both $v(x)=(v_1,\dots,v_n)$ and $u(x)$ are solutions of (\ref{A.1}) on $B_R$ with the boundary condition (\ref{A.4}), and $v(x)\leq u(x)$ holds. Note that $v(x)$ is a minimal solution, that is, if $u(x)$ is a solution of (\ref{A.1}) and (\ref{A.4}) on $B_R,$ then $v(x)\leq u(x),x\in B_R\setminus\{0\}$.

    Since $v_i(x)\leq u_i(x)$ for $x\in B_R,$ we have $(\nabla u_i\cdot\nu)(x)\leq(\nabla v_i\cdot\nu)(x)$ for $x\in\partial B_R,$ where $\nu$ is the outward normal at $x\in\partial B_R$.
    By integrating (\ref{A.1}), we have
\begin{equation}\label{A.5}
\begin{split}
        \int_{B_R}(e^{av_1(x)}-e^{-bv_n(x)})dx&=\int_{\partial B_R}\sum_{i=1}^n(\nabla v_i\cdot\nu)(x)d\sigma-\sum_{i=1}^n 2\pi\gamma_i\\&\geq\int_{\partial B_R}\sum_{i=1}^n(\nabla u_i\cdot\nu)(x)d\sigma-\sum_{i=1}^n 2\pi\gamma_i\\&=\int_{B_R}(e^{au_1(x)}-e^{-bu_n(x)})dx.
\end{split}
\end{equation}
Then we have $u(x)\equiv v(x),$ that is, uniqueness holds.
\end{proof} 

We denote the solution in Lemma \ref{Alemma2} by $u^R(x)$. From the uniqueness, it follows that $u^R(x)=u^R(|x|)$. 
By studying  $u^R(x)$ as $R\to+\infty$
we shall prove our existence theorem for (\ref{A.1}).

\begin{theorem}\label{Atheorem1}
Subject to conditions (\ref{A.2}) and (\ref{A.3}),
the system (\ref{A.1}) has a unique solution on $\mathbb{R}^2\setminus\{0\}$.
\end{theorem}

\begin{proof}
We first consider $\gamma_{i}$ such that (\ref{A.3}) holds strictly and also
$\gamma_{i}<0$ for all $i$. 
We shall show that $u^{R}(x)$ converges
for any sequence $R\to+\infty$.
Note that, since $\gamma_{i}<0$, we have $u_{i}^{R}(x)=u_{i}^{R}(|x|)\geq 0$ by Lemma \ref{Alemma1},
 and also that $u_{i}^{R}(x)$  increases as $R$ increases.
 
By integrating (\ref{A.1}), we have
    \begin{equation}\label{A.radint}
    \begin{split}
    &R\frac{d u_{1}^{R}}{dr}(R)-\gamma_{1}=\int_{0}^{R}\rho(e^{au_{1}^{R}(\rho)}-1)d\rho-\int_{0}^{R}\rho(e^{(u_{2}^{R}-u_{1}^{R})(\rho)}-1)d\rho\\
    &R\frac{d u_{i}^{R}}{dr}(R)-\gamma_{i}=\int_{0}^{R}\rho(e^{u_{i}^{R}-u_{i-1}^{R})(\rho)}-1)d\rho-\int_{0}^{R}\rho(e^{(u_{i+1}^{R}-u_{i}^{R})(\rho)}-1)d\rho\\
    &R\frac{d u_{n}^{R}}{dr}(R)-\gamma_{n}=\int_{0}^{R}\rho(e^{u_{n}^{R}-u_{n-1}^{R})(\rho)}-1)d\rho-\int_{0}^{R}\rho(e^{-bu_{n}^{R}(\rho)}-1)d\rho.
    \end{split}
    \end{equation}
 Summing both sides, we have
    \begin{equation}\label{A.radintsum}
    \begin{split}
        \int_{0}^{R}\rho(e^{au_{1}^{R}(\rho)}-1)d\rho=\int_{0}^{R} \rho(e^{-bu_{n}^{R}(\rho)}-1)d\rho+R\sum_{i=1}^{n}\frac{d u_{i}^{R}}{dr}(R)-\sum_{i=1}^{n}\gamma_{i}.
    \end{split}
    \end{equation}
Since $u_{i}^{R}(|x|)\geq 0$, we have $\frac{d u_{i}^{R}}{dr}(R)\leq 0$. 
Thus 
    \begin{equation*}
        \int_{0}^{R}\rho(e^{au_{1}^{R}(\rho)}-1)d\rho\leq -\sum_{i=1}^{n}\gamma_{i},
    \end{equation*}
    which implies $u_{1}^{R}(x)$ is bounded, hence $u_1(x)=\lim\limits_{R\to+\infty}u_1^{R}(x)$
    exists. We obtain
    also
    \begin{equation*}
        \int_{0}^{\infty}\rho(e^{au_{1}(\rho)}-1)d\rho\leq\lim\limits_{R\to\infty}\int_{0}^{R}\rho(e^{au_{1}^{R}(\rho)}-1)d\rho\leq -\sum_{i=1}^{n}\gamma_{i}.
    \end{equation*}
    Further, (\ref{A.radintsum}) implies that
    \begin{equation*}
        \int_{0}^{R}\rho(1-e^{-bu_{n}^{R}(\rho)})d\rho+R\sum_{i=1}^{n}\Big(-\tfrac{d u_{i}^{R}}{dr}(R)\Big)
    \end{equation*}
is bounded as $R\to+\infty$. 
Since each term of the above is nonnegative,  
$R\sum\limits_{i=1}^{n}|\frac{d}{dr}u_{i}^{R}(R)|$  and 
$\int_{0}^{R}\rho(1-e^{-bu_{n}^{R}(\rho)})d\rho$
 are bounded as $R\to+\infty$. By 
(\ref{A.radint}), we can prove by induction on $i$ that $u_{i}^{R}(x)$ is bounded as $R\to\infty$, and 
$\lim\limits_{R\to\infty}Ru_{i}^{R}(R)\to0$.

Let $u(x)=\lim\limits_{R\to+\infty}u^{R}(x)$. Clearly $u(x)$ is a solution of (\ref{A.1}), 
and the integrals of $e^{au_{1}(x)}-1$, $e^{u_{i+1}(x)-u_{i}(x)}-1$ and $1-e^{-bu_{n}(x)}$ 
over $\mathbb{R}^{2}$ are all finite. 
From this, it is not difficult to see that
$\lim\limits_{|x|\to\infty}u_{i}(|x|)=0$, $1\leq i\leq n$. By integrating (\ref{A.1}) in a neighbourhood of $0$, we see that, for any $\epsilon>0$, there is a $\delta>0$ such that if $0<r<\delta$, then
    \begin{equation}\label{A.9}
      \left|r\tfrac{d u_{i}^{R}}{dr}(r)-\gamma_{i}\right|\leq \epsilon
    \end{equation}
for any $R>1$. When $R\to\infty$, we obtain $\lim\limits_{r\to0}\dfrac{u_{i}(x)}{\log |x|}=\gamma_{i}$, $1\leq i\leq n$.

   Similar arguments work when $\gamma_{i}>0$ for all $i$, that is, $u^{R}(x)$ converges to $u(x)$ as $R\to+\infty$ and the integrals of $1-e^{au_{1}(x)}$, $e^{u_{i+1}(x)-u_{i}(x)}-1$ and $e^{-bu_{n}(x)}-1$ over $\mathbb{R}^{2}$ are all finite. In this case, by (\ref{A.radintsum}), we have 
 $$
 \int_{0}^{R}\rho(e^{-bu_{n}^{R}(\rho)}-1)d\rho\leq\sum_{i=1}^{n}\gamma_{i},
 $$ 
 because $u_{n}^{R}(\rho)<0$ and $\frac{d u_{i}^{R}}{dR}(R)\geq 0$. Clearly the limit $u(x)$ satisfies (\ref{A.1}) and the boundary condition (\ref{A.2}). 
 
 For general $\gamma$ such that (\ref{A.3}) holds strictly, we can choose $\gamma^{(1)}$ and $\gamma^{(2)}$ such that  
$\gamma_{i}^{(2)}<\gamma_{i}<\gamma_{i}^{(1)}$
and
$\gamma_{i}^{(2)}<0,\gamma_{i}^{(1)}>0$.
Let $v_{i}^{R}$, $i=1,2,$ be the solution of (\ref{A.1})  corresponding to $\gamma^{(i)}$. Then we have 
$$
v_{1,i}^{R}\leq u_{i}^{R}\leq v_{2,i}^{R}.
$$ 
Since $v_{i}^{R}$ exists as $R\to+\infty$, we can prove, by standard estimates for linear elliptic equations, that $u_{i}^{R}$ converges to $u$ as $R\to+\infty$ and 
$v_{1}(x)\leq u(x)\leq v_{2}(x)$. 
Since both $v_{i}(x)\to0$ as $R\to+\infty$, we have $u(x)\to0$ as $R\to+\infty$. The behaviour $\lim\limits_{x\to 0}\dfrac{u_{i}(x)}{\log |x|}=\gamma_{i}$ can be established by the standard argument (\ref{A.9}) as before. 

 Finally, we consider a $\gamma$ which does not satisfy (\ref{A.3}) strictly.  It is easy to see there is a sequence of strict parameters $\gamma^{(k)}=(\gamma_{1}^{(k)},...,\gamma_{n}^{(k)})$ tending to $\gamma$ and either 
   $\gamma_{i}<\gamma_{i}^{(k+1)}<\gamma_{i}^{(k)}$ for all $k$ or 
   $\gamma_{i}^{(k+1)}>\gamma_{i}^{k}>\gamma_{i}$. 
   Without loss of generality, we assume the former case occurs. Then the solution $u^{(k)}$ with $\gamma=\gamma^{(k)}$ satisfies 
   $u^{(k)}(x)<u^{(k+1)}(x)<\cdots$. 
   By (\ref{A.radintsum}) as before, we can prove that $u^{(k)}(x)$ converges to $u(x)$ and this $u(x)$ is a solution of (\ref{A.1}) with the boundary condition (\ref{A.2}). 
 
   The uniqueness can be proved by the procedure of Lemma \ref{Alemma1}. As in Lemma \ref{Alemma2}, let $v_{i}(x)\leq u_{i}(x)$ be two solutions of (\ref{A.1}) and (\ref{A.2}), where $v(x)$ is the minimal solution obtained by the argument of Lemma \ref{Alemma2}. Then, instead of (\ref{A.5}), we have
   \begin{equation*}
   \begin{split}
    \int_{\mathbb{R}^{2}}(e^{av_{1}(x)}-e^{-bv_{n}(x)})dx=-\sum_{i=1}^{n}\gamma_{i}=\int_{\mathbb{R}^{2}}(e^{au_{1}(x)}-e^{-bu_{n}(x)})dx     
   \end{split}
   \end{equation*}
   which implies $v_{1}(x)=u_{1}(x)$, and $v_{n}(x)=u_{n}(x)$. By the method
   of Lemma \ref{Alemma1}, we can prove $v_{i}(x)=u_{i}(x)$ by induction on $i$.
   
This completes the proof of Theorem \ref{Atheorem1}. 
   \end{proof}
   
\subsection{Radial solutions without boundary conditions}\label{10.2}\ 

The purpose of this subsection is to prove the following theorem (cf.\ Section 9 of \cite{GIL2}, where the same result for the case $n=2$ of (\ref{A.1}) was proved).
  
\begin{theorem}\label{Atheorem2}
    If $u(x)$ is a radial solution of (\ref{A.1}) in $\mathbb
{R}^2\setminus\{0\}$, then:

\no (i) $\lim\limits_{|x|\to0}\dfrac{u_i(x)}{\log|x|}=\gamma_i$ exists, 
and the $\gamma_i$ satisfy (\ref{A.3}),

\no(ii) $e^{au_i}-1,e^{u_{i+1}-u_i}-1$ and $e^{-bu_n}-1$ are all in $L^1(\mathbb{R}^2)$,  

\no (iii) $\lim\limits_{|x|\to\infty}u_i(x)=0$.

\no From (i) and (iii),  $u(x)$ satisfies (\ref{A.2}).     
\end{theorem}

The proof will follow from Lemmas \ref{Alemmazero} and \ref{Alemmainfinity} below. 
We begin by establishing an identity which will be used in the proof of Lemma 
\ref{Alemmainfinity}.

\begin{lemma}\label{Alemma3}
    Suppose $u=(u_1,\dots,u_n)$ is a smooth solution of (\ref{A.1}) in $\Omega\subseteq\mathbb{R}^2$. Then the following Pohozaev identity holds:
\begin{equation*}
\begin{aligned}
    &\tfrac{2}{a}\int_{\Omega}(e^{au_{1}}-1)dx+2\sum_{i=2}^{n}\int_{\Omega}(e^{u_{i}-u_{i-1}}-1)dx+\tfrac{2}{b}\int_{\Omega}(e^{-bu_{n}}-1)dx\\
    &=\tfrac{1}{a}\int_{\partial\Omega}(x\cdot \nu)(e^{au_{1}}-1)ds+\sum_{i=2}^{n}\int_{\partial \Omega}(x\cdot \nu)(e^{u_{i}-u_{i-1}}-1)ds\\
    &+\tfrac{1}{b}\int_{\partial \Omega}(x\cdot \nu)(e^{-bu_{n}}-1)ds-\int_{\partial\Omega}(x\cdot \nabla u_{i})\tfrac{\partial u_{i}}{\partial v}ds+\frac{1}{2}\int_{\partial \Omega}(x\cdot \nu)|\nabla u_{i}|^{2}ds.
\end{aligned}
\end{equation*}
\end{lemma}

\begin{proof}
    We multiply (\ref{A.1}) by $x\cdot\nabla u_i$ and then integrate over 
    $\Omega$. Integration by parts gives the left hand term
\begin{equation*}
\begin{split}
        &\int_\Omega(x\cdot\nabla u_i)\Delta u_idx\\
        =&\int_{\partial\Omega}(x\cdot\nabla u_i)\tfrac{\partial u_i}{\partial\nu}ds-\int_\Omega|\nabla u_i|^2dx-\tfrac{1}{2}\int_\Omega x\cdot\nabla(|\nabla u_i|^2)dx\\
        =&\int_{\partial\Omega}(x\cdot\nabla u_i)\tfrac{\partial u_i}{\partial\nu}ds-\tfrac{1}{2}\int_{\partial\Omega}(x\cdot\nu)|\nabla u_i|^2ds.
\end{split}
\end{equation*}
For the right hand terms, we obtain
\begin{equation*}
\begin{split}
        &\int_\Omega(x\cdot\nabla u_1)(e^{au_1}-e^{u_2-u_1})dx+\sum_{i=2}^{n-1}\int_{B_{r,1}}(x\cdot\nabla u_i)(e^{u_i-u_{i-1}}-e^{u_{i+1}-u_i})dx\\
        &+\int_\Omega(x\cdot\nabla u_n)(e^{u_n-u_{n-1}}-e^{-bu_n})dx
\\
        =&\int_\Omega x\cdot\nabla u_1e^{au_1}dx+\sum_{i=2}^{n}\int_{\Omega}x\cdot(\nabla u_i-\nabla u_{i-1})e^{u_i-u_{i-1}}dx-\int_\Omega x\cdot\nabla u_ne^{-bu_n}dx
\\
        =&\tfrac{1}{a}\int_{\partial\Omega}(x\cdot\nu)(e^{au_1}-1)ds+\sum_{i=2}^{n}\int_{\partial\Omega}(x\cdot\nu)(e^{u_i-u_{i-1}}-1)ds
        \\
        &+\tfrac{1}{b}\int_{\partial\Omega}(x\cdot\nu)(e^{-bu_n}-1)ds
-\tfrac{2}{a}\int_\Omega(e^{au_1}-1)dx
    \\
        &-2\sum_{i=2}^n\int_\Omega(e^{u_i-u_{i-1}}-1)dx-\tfrac{2}{b}\int_\Omega(e^{-bu_n}-1)dx.
\end{split}
\end{equation*}
Equating these gives the identity stated in the lemma.
\end{proof}

\begin{lemma}\label{Alemmazero}
Assume that $u(x)$ is a smooth radial solution of (\ref{A.1}) on $B_R\setminus\{0\}$. Then there exist $\beta_i>0$ such that $|u_i(x)|\leq-\beta_i\log|x|$ for small $|x|$, $i=1,\dots,n$. 
Moreover, the limits $\lim\limits_{x\to 0}\dfrac{u_i(x)}{\log|x|}$ exist, $i=1,\dots,n$,
and $e^{au_1},e^{u_i-u_{i-1}},e^{-bu_n}\in L^1(B_R),i=2,\dots,n$.
It follows that
the $\gamma_i$ satisfy (\ref{A.3}).
\end{lemma}

\begin{proof}
    \no{\em Step 1:} We shall prove that there exist $\beta_i>0$ such that $|u_i(x)|\leq-\beta_i\log|x|$ for small $|x|$.
Without loss of generality we take $R=1$. For any $0<|x_0|<\frac{1}{4},$ we choose $r_0=\frac{1}{2}|x_0|$ and $|y_0-x_0|=\frac{1}{4}r_0$. Then $x_0\in B_{r_0}(y_0)\subset B_1$ and $0\not\in B_{r_0}(y_0)$. We consider the system (\ref{A.1}) on the ball $B_{r_0}(y_0)$. With $r=|x-y_0|,$ the function $w(r)=-\log r(r_0-r)$ is a smooth function on $B_{r_0}(y_0)\setminus\{0\},$ and satisfies 
    $
    w(r)\to+\infty\mbox{ as }r\to0\mbox{ or }r\to r_0,
    $
    and 
    $$
    w''+\dfrac{1}{r}w'=\frac{r_0}{r(r_0-r)^2},\ 0<r<r_0.
    $$
For $i=1,\dots,n,$ let $w_i(x)=\beta_i w(r),$ where $\beta_i$ are positive constants to be chosen later. For $x\in B_{r_0}(y_0),$ if we choose $r_0$ small enough, then 
$$
\Delta w_i=\dfrac{\beta_ir_0}{r(r_0-r)^2}\leq\dfrac{1}{(r(r_0-r))^2}.
$$ 
On the other hand, for $i=2,\dots,n-1,$ we have \begin{equation*}
    \begin{split}
        &e^{aw_1}-e^{w_2-w_1}=\dfrac{1}{(r(r_0-r))^{a\beta_1}}-\dfrac{1}{(r(r_0-r))^{\beta_2-\beta_1}},\\
         &e^{w_i-w_{i-1}}-e^{w_{i+1}-w_i}=\dfrac{1}{(r(r_0-r))^{\beta_i-\beta_{i-1}}}-\dfrac{1}{(r(r_0-r))^{\beta_{i+1}-\beta_i}},\\
          &e^{w_n-w_{n-1}}-e^{-bw_n}=\dfrac{1}{(r(r_0-r))^{\beta_n-\beta_{n-1}}}-\dfrac{1}{(r(r_0-r))^{-b\beta_n}}.
    \end{split}
\end{equation*}If we choose $a\beta_1>\beta_2-\beta_1>\cdots>\beta_n-\beta_{n-1}\geq2,$ and $r_0$ small enough, then on $B_{r_0}(y_0),w(x)=(w_1(x),\dots,w_n(x))$ is a supersolution, that is, 
\begin{equation*}
\begin{split}
    &\Delta w_1\leq e^{aw_1}-e^{w_2-w_1}\\
    &\Delta w_i\leq e^{w_i-w_{i-1}}-e^{w_{i+1}-w_i},i=2,\dots,n-1,\\
    &\Delta w_n\leq e^{w_n-w_{n-1}}-e^{-bw_n}.
\end{split}
\end{equation*}
Since $w_i(r)=+\infty, $ if $r=0$ or $r=r_0,$ the supersolution property of $w$ implies that $u_i(x)\leq w_i(x)$ for all $x\in B_{r_0}(y_0)$. The proof of this inequality can be carried out as in Lemma \ref{Alemma1} by applying the maximum principle. 

At $x=x_0,$ we have 
$
u_i(x_0)\leq-\beta_i\log|x_0|
$
for some $\beta_i$. If we let 
$
\hat{u}_i=-u_{n+1-i},\ 1\leq i\leq n,
$
then $\hat{u}_i$ satisfies (\ref{A.1}) with the pair $(\hat{a},\hat{b})=(b,a)$. Then 
$
\hat{u}_i(x_0)\leq-\hat{\beta}_i\log|x_0|
$
for some $\hat{\beta}_i$. This completes the proof of Step 1.

\no{\em Step 2:} $e^{au_1}, e^{u_i-u_{i-1}}, e^{-bu_n}\in L^1(B_1), i=2,\dots,n$.
As in the proof of Lemma \ref{Alemma3}, we have
\begin{equation*}
    \begin{split}
    &\sum_{i=1}^n\int_{\partial B_{r,1}}(x\cdot\nabla u_i)\tfrac{\partial u_i}{\partial\nu}ds-\tfrac{1}{2}\sum_{i=1}^n\int_{\partial B_{r,1}}(x\cdot\nu)|\nabla u_i|^2ds\\
        =&\tfrac{1}{a}\int_{\partial B_{r,1}}(x\cdot\nu)e^{au_1}ds+\sum_{i=2}^{n}\int_{\partial B_{r,1}}(x\cdot\nu)e^{u_i-u_{i-1}}ds+\tfrac{1}{b}\int_{\partial B_{r,1}}(x\cdot\nu)e^{-bu_n}ds\\
        &-\tfrac{2}{a}\int_{B_{r,1}}e^{au_1}dx-2\sum_{i=2}^n\int_{B_{r,1}}e^{u_i-u_{i-1}}dx-\tfrac{2}{b}\int_{B_{r,1}}e^{-bu_n}dx,
    \end{split}
    \end{equation*}
 where $B_{r,1}=\{x \st r<|x|<1\}$. 
 
 Noting that $x\cdot\nu=|x|$ on the ball, we may rewrite this as follows:
\begin{equation*}
\begin{split}
       &\tfrac{2}{a}\int_{B_{r,1}}e^{au_1}dx+2\sum_{i=2}^n\int_{B_{r,1}}e^{u_i-u_{i-1}}dx+\tfrac{2}{b}\int_{B_{r,1}}e^{-bu_n}dx
\\
       &+\tfrac{1}{a}\int_{\partial B_{r,1}}re^{au_1}ds+\sum_{i=2}^{n}\int_{\partial B_{r,1}}re^{u_i-u_{i-1}}ds+\tfrac{1}{b}\int_{\partial B_{r,1}}re^{-bu_n}ds
\\
       =&\tfrac{1}{a}\int_{\partial B_{1}}e^{au_1}ds+\sum_{i=2}^{n}\int_{\partial B_{1}}e^{u_i-u_{i-1}}ds+\tfrac{1}{b}\int_{\partial B_{1}}e^{-bu_n}ds
       -\sum_{i=1}^n\int_{\partial B_1}(\tfrac{\partial u_i}{\partial\nu})^2ds
\\
       &+\sum_{i=1}^n\int_{\partial B_r}r(\tfrac{\partial u_i}{\partial\nu})^2ds+\tfrac{1}{2}\sum_{i=1}^n\int_{\partial B_1}|\nabla u_i|^2ds-\tfrac{1}{2}\sum_{i=1}^n\int_{\partial B_r}r|\nabla u_i|^2ds
\\
    \leq&\tfrac{1}{2}\sum_{i=1}^n\int_{\partial B_r}r|\nabla u_i|^2ds+O(1).
\end{split}
\end{equation*}
    By Step 1, we may choose $r_m\to 0$ such that $| \tfrac{d u_i}{dr}(r_m) |\leq r_m^{-1}$. Since each term on the left hand side is positive,  $\lim\limits_{m\to+\infty}\int_{B_{r_m,1}}e^{u_i-u_{i-1}}dx$ $(i=2,\dots,n)$
    and
    $\lim\limits_{m\to+\infty}\int_{B_{r_m,1}}e^{-bu_{n}}dx$ exist. Hence $e^{au_1}, e^{u_i-u_{i-1}}, e^{-bu_n}\in L^1(B_1), i=2,\dots,n$.

\no{\em Step 3:} We prove that $\lim\limits_{x\to0}\frac{u_i(x)}{\log|x|}$ exists.
Recall that if $h(x)$ is a harmonic function on $B_R\setminus\{0\}$ and satisfies $|h(x)|\leq \beta \log\frac{1}{|x|}$ for small $x$, then $h(x)=c\log\frac{1}{|x|}+$ a smooth harmonic function on $B_R$.

    Let $w(x)$ be a solution of $\Delta w+g(x)=0,x\in B_R\setminus\{0\}$ and $|w(x)|\leq\beta\log\frac{1}{|x|}$ for small $x$. Then $w(x)=\int_{B_R}G(x-y)g(y)dy+h(x),$ where $h$ is harmonic on $B_R\setminus\{0\}$ with $|h(x)|\leq\beta\log\frac{1}{|x|}$,
and $G$ is the Green's function. Then it is not difficult to show that $\lim\limits_{x\to 0}\frac{w(x)}{\log|x|}=\lim\limits_{x\to0}\frac{h(x)}{\log|x|}$.
    \end{proof}
    
\begin{lemma}\label{Alemmainfinity}
Suppose $u(x)$ is a smooth radial solution of (\ref{A.1}) on 
$\mathbb{R}^{2}\setminus\{0\}$. Then 
$\lim\limits_{|x|\to+\infty}u(x)=0$. 
\end{lemma}

\begin{proof}
\no{\em Step 1:} We prove that $u(x)$ is bounded as $|x|\to+\infty$.
We construct a supersolution as in Lemma \ref{Alemmazero}. Choose any $a\beta_{1}>\beta_{2}-\beta_{1}>...>\beta_{n}-\beta_{n-1}\geq 2$, and consider $$v(x)=-\log(r-s_{1})(s_{2}-r), r=|x|$$ where $s_{2}>s_{1}>0$ and $\frac{s_2-s_1}{2}=\alpha<1$. Then
\begin{equation*}
\begin{split}
    \Delta v(x)
    &=\dfrac{s_{1}}{r(r-s_{1})^{2}}+\dfrac{s_{2}}{r(s_{2}-r)^{2}}\\
    &=\big(\dfrac{s_{1}}{r}(s_{2}-r)^{2}+\dfrac{s_{2}}{r}(r-s_{1})^{2}\big)\big((r-s_{1})(s_{2}-r)\big)^{-2}.
\end{split}
\end{equation*}
By noting that $(r-s_{1})(s_{2}-r)\leq (\frac{s_{2}-s_{1}}{2})^{2}=\alpha^{2}<1$ for $s_{1}\leq r\leq s_{2}$, we can find $s_{1}$ large enough that $w=(w_{1},...,w_{n})$, $w_{i}(x)=\beta_{i}v(x)$ is a supersolution of (\ref{A.1}) for $s_{1}<r<s_{2}$, i.e., 
\begin{equation*}
\begin{split}
    &\Delta w_{1}\leq e^{aw_{1}}-e^{w_{2}-w_{1}},\\
    &\Delta w_{i}\leq e^{w_{i}-w_{i-1}}-e^{w_{i+1}-w_{i}}, i=2,...,n-1\\
    &\Delta w_{n}\leq e^{w_{n}-w_{n-1}}-e^{-bw_{i}}
\end{split}
\end{equation*}
where $s_{i}$ depends only on $\beta_{i}$ and $\alpha$. Thus the maximum principle implies 
$$
u_{i}(x)\leq \beta_{i}v(x),s_{1}\leq |x|\leq s_{2}.
$$ 
In particular, if $s_{1}$ and $s_{2}$ are chosen so that $|x|=\frac{s_{1}+s_{2}}{2}$, we have proved that $u_{i}(x)$ is bounded above. Similarly, we let $\hat{u}_{i}=-u_{n+1-i}$ and obtain that $u_{i}(x)$ is bounded below.

\no{\em Step 2:} We prove that $\lim\limits_{|x|\to+\infty}u_{i}(x)=0$.
Set 
$
\overline{l}=\max\limits_{1\leq i\leq n}\{\overline{\lim}_{|x|\to+\infty}u_{i}(x)\}.
$
Suppose $\overline{\lim}_{|x|\to+\infty}u_{i_{0}}(x)=\overline{l}$. Take $x_{m}\in\mathbb{R}^{2}$ such that $|x_{m}|\to+\infty$ and $\lim\limits_{m\to+\infty}u_{i_{0}}(x_{m})=\overline{l}$. Then we define 
$$
v_{i}^{(m)}(x)=u_{i}(x_{m}+x),1\leq i\leq n.
$$
Since $u_{i}(x)$ is bounded,  standard p.d.e.\ estimates show that there is a subsequence (still denoted by $v^{(m)}$) of $v^{(m)}$ such that $v^{(m)}\to v$ as $m\to+\infty$ and such that $v=(v_{1},...,v_{n})$ is a smooth solution of (\ref{A.1}) in $\mathbb{R}^{2}$. Further, $v$ is bounded and 
$$
v_{i}(x)\leq v_{i_{0}}(0)=\overline{l},\forall x\in\mathbb{R}^{2}\mbox{\text{ and }}1\leq i\leq n.
$$
If $i_{0}=n$, by applying the maximum principle to $v_{n}$ we obtain
$$
-bv_{n}(0)\geq v_{n}(0)-v_{n-1}(0).
$$ 
Thus 
$
\overline{l}\geq v_{n-1}(0)\geq (1+b)v_{n}(0)=(1+b)\overline{l},
$
which implies $\overline{l}\leq 0$.

If $1<i_{0}<n$, by applying the maximum principle to $v_{i_{0}}$ we obtain
$$
v_{i_{0}+1}(0)-v_{i_{0}}(0)\geq v_{i_{0}}(0)-v_{i_{0}-1}(0)\geq 0.
$$ 
Thus $v_{i_{0}+1}(0)\geq v_{i_{0}}(0)$. Hence $v_{i_{0}+1}(0)=v_{i_{0}}(0)$. Repeating this argument, we have $v_{i_{0}}(0)=\cdots=v_{n}(0)=\overline{l}$, which implies $\overline{l}\leq 0$.

If $i_{0}=1$, then again the maximum principle yields 
$$
\overline{l}\geq v_{2}(0)\geq (1+a)v_{1}(0)=(1+a)\overline{l},
$$ 
and then we have $\overline{l}\leq 0$ as well.

Applying this argument to $\hat{u}_{i}(x)=-u_{n+1-i}(x)$, we obtain 
\[
-\underline{\lim}_{|x|\to+\infty}u_{i}(x)=
\overline{\lim}_{|x|\to+\infty}\hat{u}_{i}(x)\leq 0. 
\]
Thus $\underline{\lim}_{|x|\to+\infty}u_{n+1-i}(x)\geq0$ for all $i$. This implies $\lim\limits_{|x|\to+\infty}u_{i}(x)=0$. This proves Lemma \ref{Alemmainfinity}.
\end{proof}

Lemma \ref{Alemmazero} gives (i) and (ii) of Theorem \ref{Atheorem2}, and 
Lemma \ref{Alemmainfinity} gives (iii). Thus we have completed the proof of Theorem \ref{Atheorem2}.

{\em

\noindent
Department of Mathematics\newline
Faculty of Science and Engineering\newline
Waseda University\newline
3-4-1 Okubo, Shinjuku, Tokyo 169-8555\newline
JAPAN,

\noindent
Department of Mathematical Sciences\newline
Indiana University-Purdue University, Indianapolis\newline
402 N. Blackford St.\newline
Indianapolis, IN 46202-3267\newline
USA

\noindent
and

\noindent
St. Petersburg University\newline
7/9 Universitetskaya nab.\newline
199034 St. Petersburg\newline
RUSSIA,
   
\noindent
Taida Institute for Mathematical Sciences\newline
Center for Advanced Study in Theoretical Sciences  \newline
National Taiwan University \newline
Taipei 10617\newline
TAIWAN
}


\begin{thebibliography}{99}

\bibitem{BaDo01}
V.~Balan and J.~Dorfmeister,
\emph{Birkhoff decompositions and Iwasawa decompositions for loop groups},
Tohoku Math. J.
{\textbf 53}
(2001),
593--615.

\bibitem{BoHo06}
A. Borisov and R. P. Horja,
\emph{Mellin-Barnes integrals as Fourier-Mukai transforms},
Adv. Math.
\textbf{207}
(2006),
876--927.

\bibitem{CeVa91}
S.~Cecotti and C.~Vafa,
\emph{Topological---anti-topological fusion},
Nuclear Phys. B 
\textbf{367}
(1991),
359--461.

\bibitem{CeVa92a}
S.~Cecotti and C.~Vafa,
\emph{On classification of $N=2$ supersymmetric theories},
Comm. Math. Phys.
\textbf{158}
(1993),
569--644.

\bibitem{FIKN06}
A.~S.~Fokas, A.~R.~Its, A.~A.~Kapaev, and V.~Yu.~Novokshenov,
\emph{Painlev\'e Transcendents: The Riemann-Hilbert Approach},
Mathematical Surveys and Monographs 128,
Amer. Math. Soc.,
2006.

\bibitem{Gu97}
M.~A.~Guest,
\emph{Harmonic Maps, Loop Groups, and Integrable Systems},
LMS Student Texts 38,
Cambridge Univ. Press,
1997.

\bibitem
{GH1}
M. A. Guest and N.-K. Ho, 
\emph{A Lie-theoretic description of the solution space of the tt*-Toda equations}, 
Math. Phys. Anal. Geom.
\textbf{20}
(2017),
article 24.

\bibitem
{GH2}
M. A. Guest and N.-K. Ho, 
\emph{Kostant, Steinberg, and the Stokes matrices of the tt*-Toda equations}, 
Sel. Math. New Ser.
\textbf{25}
(2019),
article 50.

\bibitem{GIL1}
M.~A.~Guest, A.~R.~Its, and C.-S.~Lin,
\emph{Isomonodromy aspects of the tt*
equations of Cecotti and Vafa I. Stokes data},
Int. Math. Res. Notices
\textbf{2015}
(2015),
11745--11784.

\bibitem{GIL2}
M.~A.~Guest, A.~R.~Its, and C.-S.~Lin,
\emph{Isomonodromy aspects of the tt*
equations of Cecotti and Vafa II. Riemann-Hilbert problem},
Comm. Math. Phys. 
\textbf{336}
(2015),
337--380.

\bibitem
{GIL3} 
M.~A.~Guest, A.~R.~Its, and C.-S.~Lin,
\emph{Isomonodromy aspects of the tt*
equations of Cecotti and Vafa
III.  Iwasawa factorization and asymptotics},  
Comm. Math. Phys.
\textbf{374}
(2020),
923--973.

\bibitem{GuLi14}
M.~A.~Guest and C.-S.~Lin,
\emph{Nonlinear PDE aspects of the tt*
equations of Cecotti and Vafa},
J. reine angew. Math.
\textbf{689}
(2014),
1--32.

\bibitem{Ho17}
S. A. Horocholyn,
\emph{On the Stokes matrices of the tt*-Toda equation},
Tokyo J. Math.
\textbf{40}
(2017),
185--202.

\bibitem{JMU81}  
M.~Jimbo, T.~Miwa, and K.~Ueno,
\emph{Monodromy preserving deformation of linear ordinary differential equations with rational coefficients I},
Physica D
\textbf{2}
(1981),
306--352;   
II: \textbf{2}
(1981),
407--448; 
III:  \textbf{4}
(1981),
26--46.

\bibitem{Ke99}
P. Kellersch,
\emph{Eine Verallgemeinerung der Iwasawa Zerlegung in Loop Gruppen},
Dissertation, Technische Universit\"at M\"unchen, 
1999. 
(Differential Geometry --- Dynamical Systems Monographs 4, Balkan Press, 2004.)

\bibitem
{Ki89} 
A. V. Kitaev, 
{\em Method of isomonodromic deformations for \lq\lq degenerate\rq\rq\  third Painlev\'e equation},  
Jour. Sov. Math. \textbf{46} (1989), 2077--2083.

\bibitem
{Kr80} 
I. M. Krichever, 
{\em An analogue of d'Alembert's formula for the equations of the principal chiral field 
and for the Sine-Gordon equation},
Sov. Math., Dokl. 
\textbf{22}
(1980),
79--84.

\bibitem
{Mi81}  
A. V. Mikhailov,
\emph{The reduction problem and the inverse scattering method},
Physica D
\textbf{3}
(1981),
73--117.

\bibitem
{MoXX} 
T. Mochizuki, 
{\em Harmonic bundles and Toda lattices with opposite sign}, 
arXiv:1301.1718

\bibitem
{Mo14} 
T. Mochizuki, 
{\em Harmonic bundles and Toda lattices with opposite sign II},  
Comm. Math. Phys. \textbf{328} (2014), 1159--1198.

\bibitem
{No85}
V. Yu. Novokshenov,
{\em Asymptotics as $t\to\infty$
of the solution of the Cauchy problem for a two-dimensional generalization of the 
Toda lattice},
Math. USSR, Izv. \textbf{24} (1985), 347--382.

\bibitem{PrSe86} 
A. Pressley and G. B. Segal,
\emph{Loop Groups},
Oxford Univ. Press, 1986.

\bibitem{TrWi98}
C.~A.~Tracy and H.~Widom, 
\emph{Asymptotics of a class of solutions to the cylindrical Toda equations},
Comm. Math. Phys.
\textbf{190}
(1998),
697--721.

\end{thebibliography}
\end{document}